\documentclass[12pt,oneside]{amsart}

\usepackage{geometry}                
\usepackage{graphicx}
\usepackage{amssymb}
\usepackage{color}
\usepackage{bbm, dsfont}

\usepackage{tikz}
\usetikzlibrary{3d}

\usepackage[hidelinks]{hyperref}

\hypersetup{
	colorlinks=false,
	pdfborder={0 0 0},
	pdfborderstyle={/S/U/W 0},
}

\usetikzlibrary{patterns.meta}

\usepackage{verbatim}

\usepackage{amssymb,amsthm,amsmath,amssymb,wrapfig,dsfont}
\DeclareFontFamily{OML}{rsfs}{\skewchar\font'177}
\DeclareFontShape{OML}{rsfs}{m}{n}{ <5> <6> rsfs5 <7> <8> <9>
	rsfs7 <10> <10.95> <12> <14.4> <17.28> <20.74> <24.88> rsfs10 }{}
\DeclareMathAlphabet{\mathfs}{OML}{rsfs}{m}{n}

\usepackage{oplotsymbl}

\usepackage{ulem}
\usetikzlibrary{calc}

\usepackage{amsmath}
\usepackage{mathrsfs}

\usepackage{wasysym}

\usepackage{verbatim}
\usepackage{amsmath,amssymb,amsthm}             
\usepackage{amssymb}             
\usepackage{amsfonts}            
\usepackage{mathrsfs}            
\usepackage{amsthm}
\usepackage{algorithm}  
\usepackage{algorithmicx}  
\usepackage{algpseudocode}  
\usepackage{mathtools}
\usepackage{commath}
\usepackage{physics}
\usepackage{bm}
\usepackage{graphicx}

\usepackage{float}
\usepackage{listings}
\usepackage{subfigure}
\usepackage{multirow}
\usepackage{color}
\usepackage{bbm}
\usepackage[export]{adjustbox}
\usepackage{enumerate}
\usepackage{bbm}

\usepackage{amsmath}
\usepackage{mathrsfs}

\newcommand{\RNum}[1]{\uppercase\expandafter{\romannumeral #1\relax}}

\newcommand{\specificthanks}[1]{\@fnsymbol{#1}}

\usetikzlibrary{shapes}

\DeclareFontFamily{U}{mathx}{}
\DeclareFontShape{U}{mathx}{m}{n}{<-> mathx10}{}
\DeclareSymbolFont{mathx}{U}{mathx}{m}{n}
\DeclareMathAccent{\widehat}{0}{mathx}{"70}
\DeclareMathAccent{\widecheck}{0}{mathx}{"71}

\usepackage{amsfonts}

\usetikzlibrary{arrows.meta}
\usetikzlibrary{shapes,snakes}

\usepackage{newclude}

\newtheorem{singletheorem}{Theorem}
\newtheorem{theorem}{Theorem}[section]
\newtheorem{lemma}[theorem]{Lemma}

\newtheorem{proposition}[theorem]{Proposition}
\newtheorem{conjecture}[theorem]{Conjecture}

\theoremstyle{definition}
\newtheorem{definition}[theorem]{Definition}
\theoremstyle{remark}
\newtheorem{remark}[theorem]{Remark}

\numberwithin{equation}{section}

\newcounter{cnstcnt}
\newcommand{\cl}{%
	\refstepcounter{cnstcnt}%
	\ensuremath{c_{\thecnstcnt}}}
\newcommand{\cref}[1]{\ensuremath{c_{\ref{#1}}}}

\newcounter{newcnstcnt}
\newcommand{\Cl}{%
	\refstepcounter{newcnstcnt}%
	\ensuremath{C_{\thenewcnstcnt}}}
\newcommand{\Cref}[1]{\ensuremath{C_{\ref{#1}}}}

\usepackage{xcolor}

\definecolor{mygreen}{RGB}{65, 117, 5}

\definecolor{myorange}{RGB}{245, 166, 35}

\newcommand{\cev}[1]{\reflectbox{\ensuremath{\vec{\reflectbox{\ensuremath{#1}}}}}}

\begin{document}

	\title{Minimal harmonic measure on 2D lattices}

	\author{Zhenhao Cai$^1$}
	\address[Zhenhao Cai]{School of Mathematical Sciences, Peking University, Beijing, China}
	\email{caizhenhao@pku.edu.cn}
	\thanks{$^1$School of Mathematical Sciences, Peking University, Beijing, China}

	\author{Eviatar B. Procaccia$^2$}
	\address[Eviatar B. Procaccia]{Faculty of Data and Decision Sciences, Technion - Israel Institute of Technology, Haifa, Israel}
	\email{eviatarp@technion.ac.il}
	\thanks{$^2$Faculty of Data and Decision Sciences, Technion - Israel Institute of Technology, Haifa, Israel}

	\author{Yuan Zhang$^3$}
	\address[Yuan Zhang]{Center for Applied Statistics and School of Statistics, Renmin University of China, Beijing, China}
	\email{zhang\_probab@ruc.edu.cn}
	\thanks{$^3$Center for Applied Statistics and School of Statistics, Renmin University of China, Beijing, China}

	\maketitle

		\tableofcontents

	\begin{abstract}
		We study the harmonic measure (i.e. the limit of the hitting distribution of a simple random walk starting from a distant point) on three canonical two-dimensional lattices: the square lattice $\mathbb{Z}^2$, the triangular lattice $\mathscr{T}$ and the hexagonal lattice $\mathscr{H}$. In particular, for the least positive value of the harmonic measure of any $n$-point set, denoted by $\mathcal{M}_n(\mathscr{G})$, we prove in this paper that 
		\begin{equation*}
			[\lambda(\mathscr{G})]^{-n+c\sqrt{n}} \le	\mathcal{M}_n(\mathscr{G})\le  [\lambda(\mathscr{G})]^{-n+C\sqrt{n}},
		\end{equation*}
		where $\lambda(\mathbb{Z}^2)=(2+\sqrt{3})^2$, $\lambda(\mathscr{T})=3+2\sqrt{2}$ and $\lambda(\mathscr{H})=(\tfrac{3+\sqrt{5}}{2})^3$.
		Our results confirm a stronger version of the conjecture proposed by Calvert, Ganguly and Hammond (2023) which predicts the asymptotic of the exponent of $\mathcal{M}_n(\mathbb{Z}^2)$. Moreover, these estimates also significantly extend the findings in our previous paper with Kozma (2023) that $\mathcal{M}_n(\mathscr{G})$ decays exponentially for a large family of graphs $\mathscr{G}$ including $\mathscr{T}$, $\mathscr{H}$ and $\mathbb{Z}^d$ for all $d\ge 2$.

	\end{abstract}

	\section{Introduction}\label{section_intro}

In this paper, we study the least positive value of (discrete) harmonic measures (from infinity) on two-dimensional lattices, especially on the square lattice $\mathbb{Z}^2$, the triangular lattice $\mathscr{T}$ and the hexagonal lattice $\mathscr{H}$ (corresponding to the only three regular tessellations in $\mathbb{R}^2$; see e.g. \cite[Chapter 4]{ball1917mathematical}).

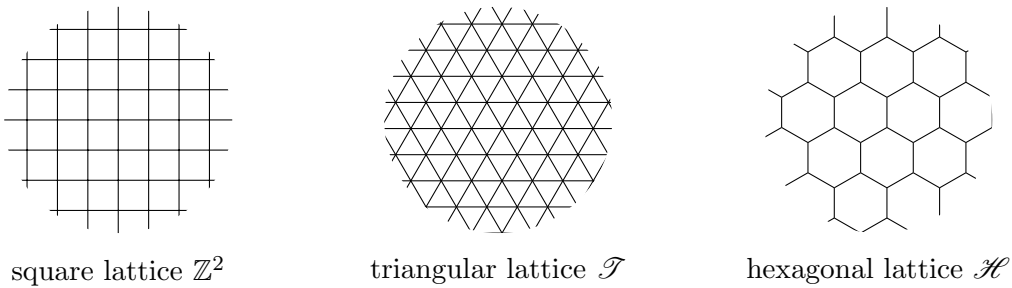
\begin{figure}[h!]		
	\begin{tikzpicture}[scale=0.5]
		

		\begin{scope}
			\clip (0,0) circle (3cm);
			\begin{scope}[scale= 0.8]
					\draw[] (-10,-10) grid (10,10);
			\end{scope}
		\end{scope}
			\begin{scope}
			\node at (0,-4) {\small{square lattice $\mathbb{Z}^2$}};
		\end{scope}


		\newcommand*\rows{30}
		\begin{scope}[xshift=10cm]
			\clip (0,0) circle (3cm);
			\begin{scope}[xshift=-5.5cm,yshift=-3cm,scale=0.8]
				\foreach \row in {0, 1, ...,\rows} {
					\draw ($\row*(0.5, {0.5*sqrt(3)})$) -- ($(\rows,0)+\row*(-0.5, {0.5*sqrt(3)})$);
					\draw ($\row*(1, 0)$) -- ($(\rows/2,{\rows/2*sqrt(3)})+\row*(0.5,{-0.5*sqrt(3)})$);
					\draw ($\row*(1, 0)$) -- ($(0,0)+\row*(0.5,{0.5*sqrt(3)})$);
				}	
			\end{scope}
		\end{scope}
		\begin{scope}[xshift=10cm]
			\node at (0,-4) {\small{triangular lattice $\mathscr{T}$}};
		\end{scope}

		
		\begin{scope}[xshift=20cm]
			\clip (0,0) circle (3cm);
			\begin{scope}[xshift=3cm,yshift=-5cm,scale=0.8]
				\begin{scope}[rotate=90]
					\foreach \i in {0,...,10} 
					\foreach \j in {0,...,10} {
						\foreach \a in {0,120,-120} \draw (3*\i,2*sin{60}*\j) -- +(\a:1);
						\foreach \a in {0,120,-120} \draw (3*\i+3*cos{60},2*sin{60}*\j+sin{60}) -- +(\a:1);}
				\end{scope}
			\end{scope}
		\end{scope}
		\begin{scope}[xshift=20cm]
			\node at (0,-4) {\small{hexagonal lattice $\mathscr{H}$}};
		\end{scope}
		
	\end{tikzpicture}
	\caption{Illustrations for $\mathbb{Z}^2$, $\mathscr{T}$ and $\mathscr{H}$. \label{fig:lattices}}
\end{figure}

	To facilitate understanding, we first review the definition
of harmonic measure. For any simple (i.e. each pair of vertices is connected by at most one edge, and no vertex is connected to itself by an edge), locally finite (i.e. the degree of each vertex is finite) and connected graph $\mathscr{G}=(\mathfs{V},\mathfs{E})$ (where \(\mathfs{V}\) and \(\mathfs{E}\) denote the sets of vertices and edges respectively), a random walk on $\mathscr{G}$ starting from $x\in \mathfs{V}$ is a discrete-time Markov process $\{S_n\}_{n\ge 0}$ with law $\mathbb{P}_x$ such that $\mathbb{P}_x(S_0=x)=1$ and 
\begin{equation}
	\mathbb{P}_x(S_{n+1}=z\mid S_n=y ) = [\mathrm{deg}(y) ]^{-1}\cdot \mathbbm{1}_{y\sim z}, \ \ \forall y,z\in \mathfs{V}\ \text{and}\ n\ge 0,
\end{equation}
where ``$y\sim z$'' means ``$\{y,z\}\in \mathfs{E}$'', and $\mathrm{deg}(y):= |\{v\in \mathfs{V}: v\sim y\}|$. In general, the harmonic measure of a finite set is the limit distribution of the first hitting position in this set by a random walk with a faraway starting point. Precisely, we assume that in $\mathfs{V}$ there exists a predetermined vertex, denoted by $\bm{0}$. For any finite $A\subset \mathfs{V}$ and $y\in A$, the harmonic measure of $A$ at $y$ is defined as 
\begin{equation}\label{harmonic measure}
	\mathbb{H}_A(y):= \lim\nolimits_{\| x\|\to \infty} \mathbb{P}_x(S_{\tau_A}=y \mid \tau_A<\infty),
\end{equation}
where $\|x \|$ is the graph distance between $x$ and $\bm{0}$, and $\tau_A$ is the first time when $\{S_n\}_{n\ge 0}$ hits $A$. Although the existence of the limit in (\ref{harmonic measure}) may fail for some graphs (readers may refer to \cite{boivin2013existence} and \cite[Section 10.7]{lyons2017probability} for related results), it is established for $\mathbb{Z}^d$ with $d\ge 2$ (see e.g. \cite[Theorem 2.1.3]{lawler2013intersections}), and the proof can be readily generalized to $\mathscr{T}$ and $\mathscr{H}$. Thus, the harmonic measure is well-defined on the graphs that we focus on in this paper, namely, $\mathbb{Z}^2$, $\mathscr{T}$ and $\mathscr{H}$.

In probability theory and statistical mechanics, the harmonic measure has been widely studied due to the abundance of models driven by it. In particular, the extremal values of harmonic measures have been found to be of special significance in these studies. In the study of Diffusion Limited Aggregation (DLA) \cite{witten1981diffusion,witten1983diffusion} and Stationary DLA \cite{mu2022scaling,procaccia2020stationary}, a discrete version of the Beurling circular projection theorem \cite{beurling1933etudes} concerning the maximum of harmonic measures played an important role in the derivation of upper bounds on the growth rate (see \cite{benjamini2017upper,kesten1987hitting, kesten1987long,kesten1990upper,lawler2013intersections,lawler2004beurling,procaccia2019stationary}). Precisely, the discrete Beurling estimate establishes that $\max_{x\in A}\mathbb{H}_A(x)\le Cn^{-\frac{1}{2}}$ for all connected sets $A\subset\mathbb{Z}^2$ of cardinality $n$. In fact, this bound is sharp up to a constant factor, since the harmonic measure at the endpoint of a line segment of length $n$ is of the order $n^{-\frac{1}{2}}$. However, it remains an open problem to prove that this line segment is the unique maximizer among all connected $n$-point sets, or even a maximizer at all (see Conjecture \ref{conj_existence} for a relevant question). In addition to extremal values,  the typical values of harmonic measures of connected sets were analyzed in \cite{lawler1993discrete} as a discrete version of Makarov's Theorem \cite{makarov1985distortion}. These results were recently used in \cite{losev2023long} to estimate the growth rate of the Dielectric Breakdown Model \cite{niemeyer1984fractal}. Excellent accounts for related topics can be found in \cite{lawler1994random,olegovic2021diffusion}.

The continuous harmonic measure in $\mathbb{R}^2$, which is an analogue of the discrete harmonic measure where the random walk is replaced by a Brownian motion, has also been heavily used in models of statistical mechanics due to its strong connection with the powerful theory of complex conformal functions. Notable models include Hastings-Levitov Aggregation \cite{berger2022growth,johansson2012scaling,norris2012hastings, silvestri2017fluctuation} and Aggregate Loewner Evolution (ALE) \cite{sola2019one}. Among these models, the stationary Hastings-Levitov$(0)$ model is the only one in the DLA class (where particles of approximately equal size attach to the aggregate according to the harmonic measure or its analogue) for which an exact growth rate has been established (see \cite{berger2022growth}). Mainly, arms of the aggregate that reach height $n$ typically contain an order of $n^{3/2}$ particles (see also \cite{procaccia2021dimension}, which includes computer simulations and further discussions). Notably, this result matches the aforementioned upper bound in \cite{kesten1987long} for the growth rate of DLA.

Recently, in order to bound the collapse time in the Harmonic Activation and Transport model (HAT), \cite{calvert2023existence,Calvert2021CollapseAD} studied the least positive value of harmonic measures of general $n$-point sets, which can be disconnected or even spread out. Particularly, it was proved in \cite[Theorem 1.9]{Calvert2021CollapseAD} that $\mathcal{M}_n(\mathbb{Z}^2)\ge e^{-Cn\log(n)}$, where
\begin{equation}\label{def_Mn}
	\mathcal{M}_n(\mathscr{G}):= \inf\big\{\mathbb{H}_A(x):A\subset \mathfs{V}, |A|=n,x\in A,\mathbb{H}_A(x)>0 \big\} 
\end{equation}
is the least positive value of the harmonic measure that a vertex can have in a set of cardinality $n$ on the graph $\mathscr{G}=(\mathfs{V},\mathfs{E})$; moreover, if restricting to connected sets in $\mathbb{Z}^2$, then this lower bound for $\mathcal{M}_n(\mathbb{Z}^2)$ can be improved to $e^{-Cn}$. Additionally, in \cite{Calvert2021CollapseAD} it was conjectured that the exponentially decaying lower bound for $\mathcal{M}_n(\mathbb{Z}^2)$ can be established without the requirement of connectivity. Later, this conjecture was solved in \cite{psi} by Kozma and the authors of this paper. Precisely, in \cite[Theorem 1]{psi} they established the vertex-removal stability for the harmonic measure on $\mathbb{Z}^2$ (i.e. it is feasible to remove some vertex while changing the harmonic measure by a uniformly bounded factor) and then employed it to prove $e^{-Cn}\le \mathcal{M}_n(\mathbb{Z}^2)\le e^{-cn}$ without the connectivity condition (see \cite[Theorem 2]{psi}). Moreover, in the case of $\mathbb{Z}^d$ for $d\ge 3$, \cite[Theorem 3]{psi} also proved $e^{-Cn}\le \mathcal{M}_n(\mathbb{Z}^d)\le e^{-cn}$, despite the vertex-removal stability no longer holding there. Furthermore, as stated in \cite[Remark 3.8]{psi}, the proof of the vertex-removal stability on $\mathbb{Z}^2$ can be generalized to a large class of planar graphs, where the invariance principle holds and the probability for the random walk to surround each face is uniformly bounded away from $0$. Therefore, since $\mathscr{T}$ and $\mathscr{H}$ both satisfy these properties, one has  
\begin{equation}
	e^{-Cn}\le \mathcal{M}_n(\mathscr{T})\le e^{-cn}\ \ \text{and}\ \  e^{-Cn}\le \mathcal{M}_n(\mathscr{H})\le e^{-cn}.
\end{equation}

For any graph $\mathscr{G}$ where $\mathcal{M}_n(\mathscr{G})$ decays exponentially, a natural question is: 
\begin{center}
	\textit{What is the exact decay rate of $\mathcal{M}_n(\mathscr{G})$?}
\end{center}
In the case of $\mathbb{Z}^2$, the asymptotic of the exponent of $\mathcal{M}_n(\mathbb{Z}^2)$ was conjectured by Calvert, Ganguly and Hammond through the following inspiring construction. Precisely, in \cite[Example 1.6]{Calvert2021CollapseAD} they proposed an increasing sequence of sets $\{D_n^{\mathbb{Z}^2}\}_{n\ge 2}$ with $\bm{0}:=(0,0)\in D_n^{\mathbb{Z}^2}$ and $|D_n^{\mathbb{Z}^2}|=n$ named ``square spiral'' (i.e. the set consisting of black and red dots in Figure \ref{fig:spiral_Z2}), and proved that 
\begin{equation}\label{hm_spiral}
	\mathbb{H}_{D_n^{\mathbb{Z}^2}}(\bm{0}) = (2+\sqrt{3})^{-2n+o(n)}.
\end{equation}
Furthermore, they also pointed out that in this example, almost every additional vertex lengthens the shortest path from the outermost boundary to the origin $\bm{0}$ by two steps, which intuitively represents the most efficient way. In light of this, they conjectured that $D_n^{\mathbb{Z}^2}$ approximately achieves the least positive value of the harmonic measure of any $n$-point set (see \cite[Conjecture 1.5]{Calvert2021CollapseAD}). I.e., 
\begin{equation}\label{new_1.6}
	\mathcal{M}_n(\mathbb{Z}^2)= (2+\sqrt{3})^{-2n+o(n)}.
\end{equation}
The main result of this paper gives an affirmative answer to a stronger version of this conjecture, and extends it to $\mathscr{T}$ and $\mathscr{H}$:
\begin{singletheorem}\label{theorem1.1}
	There exist constants $\Cl\label{thm1_1}>\cl\label{thm1_2}>0$ such that for $\mathscr{G}\in \{\mathbb{Z}^2,\mathscr{T}, \mathscr{H} \}$ and any sufficiently large $n$, 
	\begin{equation}\label{M_n}
		[\lambda(\mathscr{G})]^{-n+\cref{thm1_2}\sqrt{n}} \le	\mathcal{M}_n(\mathscr{G})\le  [\lambda(\mathscr{G})]^{-n+\Cref{thm1_1}\sqrt{n}},
	\end{equation}
	where $\lambda(\mathbb{Z}^2)=(2+\sqrt{3})^2$, $\lambda(\mathscr{T})=3+2\sqrt{2}$ and $\lambda(\mathscr{H})=(\tfrac{3+\sqrt{5}}{2})^3$.
\end{singletheorem}

\begin{remark}\label{remark_s_s}
		Theorem \ref{theorem1.1} not only confirms the conjecture (\ref{new_1.6}), but also reveals that the $o(n)$ term in the exponent of (\ref{new_1.6}) is positive and of order $\sqrt{n}$. 
\end{remark}

\subsection{Existence of minimizing sets}

\begin{figure}[h!]
	\begin{tikzpicture}[scale=0.27]
		\draw[step=1cm,blue,thin, dotted] (-11,-9) grid (7,9);
		\node at (0,0) [circle,fill=red,inner sep=1.5pt]{}; \node at (-0.5,0.5) [scale=0.8]{$\bm{0}$};
		\node at (0,1) [circle,fill=black,inner sep=1.5pt]{};       \node at (0, -1)[circle,fill=black,inner sep=1.5pt]{} ;\node at (-1,1) [circle,fill=black,inner sep=1.5pt]{};\node at (-2,0) [circle,fill=black,inner sep=1.5pt]{};\node at (1,0) [circle,fill=black,inner sep=1.5pt]{};\node at (-1,-2) [circle,fill=black,inner sep=1.5pt]{};\node at (-2,-3) [circle,fill=black,inner sep=1.5pt]{};\node at (-3,-2) [circle,fill=black,inner sep=1.5pt]{};\node at (-4,-1) [circle,fill=black,inner sep=1.5pt]{};\node at (-5,0) [circle,fill=black,inner sep=1.5pt]{};\node at (-4,1) [circle,fill=black,inner sep=1.5pt]{};\node at (-3,2) [circle,fill=black,inner sep=1.5pt]{};\node at (-2,3) [circle,fill=black,inner sep=1.5pt]{};\node at (-1,4) [circle,fill=black,inner sep=1.5pt]{};\node at (0,4) [circle,fill=black,inner sep=1.5pt]{};\node at (1,3) [circle,fill=black,inner sep=1.5pt]{};\node at (2,2) [circle,fill=black,inner sep=1.5pt]{};\node at (3,1) [circle,fill=black,inner sep=1.5pt]{};\node at (4,0) [circle,fill=black,inner sep=1.5pt]{};\node at (3,-1) [circle,fill=black,inner sep=1.5pt]{};\node at (2,-2) [circle,fill=black,inner sep=1.5pt]{};\node at (1,-3) [circle,fill=black,inner sep=1.5pt]{};\node at (0,-4) [circle,fill=black,inner sep=1.5pt]{};\node at (-1,-5) [circle,fill=black,inner sep=1.5pt]{};\node at (-2,-6) [circle,fill=black,inner sep=1.5pt]{};\node at (-3,-5) [circle,fill=black,inner sep=1.5pt]{};\node at (-4,-4) [circle,fill=black,inner sep=1.5pt]{};\node at (-5,-3) [circle,fill=black,inner sep=1.5pt]{};\node at (-6,-2) [circle,fill=black,inner sep=1.5pt]{};\node at (-7,-1) [circle,fill=black,inner sep=1.5pt]{};
		\node at (-8,0) [circle,fill=black,inner sep=1.5pt]{}; \node at (-7,1) [circle,fill=black,inner sep=1.5pt]{};\node at (-6,2) [circle,fill=black,inner sep=1.5pt]{};
		\node at (-5,3) [circle,fill=black,inner sep=1.5pt]{};\node at (-4,4) [circle,fill=black,inner sep=1.5pt]{};\node at (-3,5) [circle,fill=black,inner sep=1.5pt]{};\node at (-2,6) [circle,fill=black,inner sep=1.5pt]{};\node at (-1,7) [circle,fill=black,inner sep=1.5pt]{};\node at (0,7) [circle,fill=black,inner sep=1.5pt]{};\node at (1,6) [circle,fill=black,inner sep=1.5pt]{};\node at (2,5) [circle,fill=black,inner sep=1.5pt]{};\node at (3,4) [circle,fill=black,inner sep=1.5pt]{};\node at (4,3) [circle,fill=black,inner sep=1.5pt]{};\node at (5,2) [circle,fill=black,inner sep=1.5pt]{};\node at (6,1) [circle,fill=black,inner sep=1.5pt]{};\node at (7,0) [circle,fill=black,inner sep=1.5pt]{};\node at (6,-1) [circle,fill=black,inner sep=1.5pt]{};\node at (5,-2) [circle,fill=black,inner sep=1.5pt]{};\node at (4,-3) [circle,fill=black,inner sep=1.5pt]{};\node at (3,-4) [circle,fill=black,inner sep=1.5pt]{};\node at (2,-5) [circle,fill=black,inner sep=1.5pt]{};\node at (1,-6) [circle,fill=black,inner sep=1.5pt]{};\node at (0,-7) [circle,fill=black,inner sep=1.5pt]{};\node at (-1,-8) [circle,fill=black,inner sep=1.5pt]{};\node at (-2,-9) [circle,fill=black,inner sep=1.5pt]{};\node at (-3,-8) [circle,fill=black,inner sep=1.5pt]{};\node at (-4,-7) [circle,fill=black,inner sep=1.5pt]{};
		\node at (-5,-6) [circle,fill=black,inner sep=1.5pt]{};\node at (-6,-5) [circle,fill=black,inner sep=1.5pt]{};\node at (-7,-4) [circle,fill=black,inner sep=1.5pt]{};\node at (-8,-3) [circle,fill=black,inner sep=1.5pt]{};\node at (-9,-2) [circle,fill=black,inner sep=1.5pt]{};\node at (-10,-1) [circle,fill=black,inner sep=1.5pt]{};
		\node at (-11,0) [circle,fill=black,inner sep=1.5pt]{};\node at (-10,1) [circle,fill=black,inner sep=1.5pt]{};\node at (-9,2) [circle,fill=black,inner sep=1.5pt]{};\node at (-8,3) [circle,fill=black,inner sep=1.5pt]{};\node at (-7,4) [circle,fill=black,inner sep=1.5pt]{};\node at (-6,5) [circle,fill=black,inner sep=1.5pt]{};\node at (-5,6) [circle,fill=black,inner sep=1.5pt]{};\node at (-4,7) [circle,fill=black,inner sep=1.5pt]{};\node at (-3,8) [circle,fill=black,inner sep=1.5pt]{};\node at (-2,9) [circle,fill=black,inner sep=1.5pt]{};

		\node at (-10,0) [circle,fill=black,inner sep=1.5pt]{};
		\node at (-7,0) [circle,fill=black,inner sep=1.5pt]{};
		\node at (-4,0) [circle,fill=black,inner sep=1.5pt]{};
		\node at (3,0) [circle,fill=black,inner sep=1.5pt]{};
		\node at (6,0) [circle,fill=black,inner sep=1.5pt]{};
		\node at (0,3) [circle,fill=black,inner sep=1.5pt]{};
		\node at (-1,3) [circle,fill=black,inner sep=1.5pt]{};
		\node at (0,6) [circle,fill=black,inner sep=1.5pt]{};
		\node at (-1,6) [circle,fill=black,inner sep=1.5pt]{};
		\node at (-2,-2) [circle,fill=black,inner sep=1.5pt]{};
		\node at (-2,-5) [circle,fill=black,inner sep=1.5pt]{};
		\node at (-2,-8) [circle,fill=black,inner sep=1.5pt]{};

		\draw [pink,very thick] (-1,8) -- (-2,8);
		\draw [pink,very thick] (-2,8) -- (-2,7);
		\draw [pink,very thick] (-2,7) -- (-3,7);
		\draw [pink,very thick] (-3,7) -- (-3,6);
		\draw [pink,very thick] (-3,6) -- (-4,6);
		\draw [pink,very thick] (-4,6) -- (-4,5);
		\draw [pink,very thick] (-4,5) -- (-5,5);
		\draw [pink,very thick] (-5,5) -- (-5,4);
		\draw [pink,very thick] (-5,4) -- (-6,4);
		\draw [pink,very thick] (-6,4) -- (-6,3);
		\draw [pink,very thick] (-6,3) -- (-7,3);
		\draw [pink,very thick] (-7,3) -- (-7,2);
		\draw [pink,very thick] (-7,2) -- (-8,2);
		\draw [pink,very thick] (-8,2) -- (-8,1);
		\draw [pink,very thick] (-8,1) -- (-9,1);
		\draw [pink,very thick] (-9,1) -- (-9,-1);
		\draw [pink,very thick] (-9,-1) -- (-8,-1);
		\draw [pink,very thick] (-8,-1) -- (-8,-2);
		\draw [pink,very thick] (-8,-2) -- (-7,-2);
		\draw [pink,very thick] (-7,-2) -- (-7,-3);
		\draw [pink,very thick] (-7,-3) -- (-6,-3);
		\draw [pink,very thick] (-6,-3) -- (-6,-4);
		\draw [pink,very thick] (-6,-4) -- (-5,-4);
		\draw [pink,very thick] (-5,-4) -- (-5,-5);
		\draw [pink,very thick] (-5,-5) -- (-4,-5);
		\draw [pink,very thick] (-4,-5) -- (-4,-6);
		\draw [pink,very thick] (-4,-6) -- (-3,-6);
		\draw [pink,very thick] (-3,-6) -- (-3,-7);
		\draw [pink,very thick] (-3,-7) -- (-1,-7);
		\draw [pink,very thick] (-1,-7) -- (-1,-6);
		\draw [pink,very thick] (-1,-6) -- (0,-6);
		\draw [pink,very thick] (0,-6) -- (0,-5);
		\draw [pink,very thick] (0,-5) -- (1,-5);
		\draw [pink,very thick] (1,-5) -- (1,-4);
		\draw [pink,very thick] (1,-4) -- (2,-4);
		\draw [pink,very thick] (2,-4) -- (2,-3);
		\draw [pink,very thick] (2,-3) -- (3,-3);
		\draw [pink,very thick] (3,-3) -- (3,-2);
		\draw [pink,very thick] (3,-2) -- (4,-2);
		\draw [pink,very thick] (4,-2) -- (4,-1);
		\draw [pink,very thick] (4,-1) -- (5,-1);
		\draw [pink,very thick] (5,-1) -- (5,1);
		\draw [pink,very thick] (5,1) -- (4,1);
		\draw [pink,very thick] (4,1) -- (4,2);
		\draw [pink,very thick] (4,2) -- (3,2);
		\draw [pink,very thick] (3,2) -- (3,3);
		\draw [pink,very thick] (3,3) -- (2,3);
		\draw [pink,very thick] (2,3) -- (2,4);
		\draw [pink,very thick] (2,4) -- (1,4);
		\draw [pink,very thick] (1,4) -- (1,5);
		\draw [pink,very thick] (1,5) -- (0,5);
		\draw [pink,very thick] (0,5) -- (-2,5);
		\draw [pink,very thick] (-2,5) -- (-2,4);
		\draw [pink,very thick] (-2,4) -- (-3,4);
		\draw [pink,very thick] (-3,4) -- (-3,3);
		\draw [pink,very thick] (-3,3) -- (-4,3);
		\draw [pink,very thick] (-4,3) -- (-4,2);
		\draw [pink,very thick] (-4,2) -- (-5,2);
		\draw [pink,very thick] (-5,2) -- (-5,1);
		\draw [pink,very thick] (-5,1) -- (-6,1);
		\draw [pink,very thick] (-6,1) -- (-6,0);
		\draw [pink,very thick] (-6,0) -- (-6,-1);
		\draw [pink,very thick] (-6,-1) -- (-5,-1);
		\draw [pink,very thick] (-5,-1) -- (-5,-2);
		\draw [pink,very thick] (-5,-2) -- (-4,-2);
		\draw [pink,very thick] (-4,-2) -- (-4,-3);
		\draw [pink,very thick] (-4,-3) -- (-3,-3);
		\draw [pink,very thick] (-3,-3) -- (-3,-4);
		\draw [pink,very thick] (-3,-4) -- (-2,-4);
		\draw [pink,very thick] (-2,-4) -- (-1,-4);
		\draw [pink,very thick] (-1,-4) -- (-1,-3);
		\draw [pink,very thick] (-1,-3) -- (-0,-3);
		\draw [pink,very thick] (0,-3) -- (0,-2);
		\draw [pink,very thick] (0,-2) -- (1,-2);
		\draw [pink,very thick] (1,-2) -- (1,-1);
		\draw [pink,very thick] (1,-1) -- (2,-1);
		\draw [pink,very thick] (2,-1) -- (2,1);
		\draw [pink,very thick] (2,1) -- (1,1);
		\draw [pink,very thick] (1,1) -- (1,2);
		\draw [pink,very thick] (1,2) -- (-2,2);
		\draw [pink,very thick] (-2,2) -- (-2,1);
		\draw [pink,very thick] (-2,1) -- (-3,1);
		\draw [pink,very thick] (-3,1) -- (-3,-1);
		\draw [pink,very thick] (-3,-1) -- (-1,-1);
		\draw [pink,very thick] (-1,-1) -- (-1,0);
		\draw [pink,very thick] (-1,0) -- (0,0);
		
		\draw [->] (1.5,7.5) to (-0.5,7.9);
		\node at (2.5,7.2) []{$v_n^{\mathbb{Z}^2}$};
		
	\end{tikzpicture}
	\caption{(\cite[Figure 3]{Calvert2021CollapseAD}) An illustration for the square spiral $D_n^{\mathbb{Z}^2}$  \label{fig:spiral_Z2}}
\end{figure}
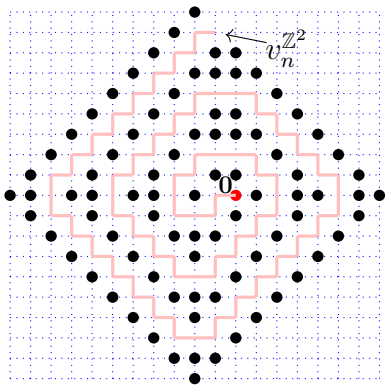

Although Theorem \ref{theorem1.1} gives a fairly precise estimate for $\mathcal{M}_n(\mathscr{G})$ with $\mathscr{G}\in \{\mathbb{Z}^2, \mathscr{T}, \mathscr{H}\}$, it is open to show that the minimum in $\mathcal{M}_n(\mathscr{G})$ can be achieved by some specific $n$-point set. Naturally, a more profound question is to examine the uniqueness or even to identify the minimizing set.

\begin{conjecture}\label{conj_existence}
	Let $\mathscr{G}\in \{\mathbb{Z}^2,\mathscr{T}, \mathscr{H} \}$. For all $n\ge 1$, there exists a set $M\subset \mathfs{V}$ satisfying ${\bf 0}\in M$, $|M|=n$ and $\mathbb{H}_M({\bm 0})=\mathcal{M}_n(\mathscr{G})$. 
\end{conjecture}

In an upcoming work \cite{yam2024}, it is proved that for the non-planar lattices $\mathbb{Z}^d$ with $d>2$, the minimizing set for harmonic measures exists, even though the exact decay rate of $\mathcal{M}_n(\mathbb{Z}^d)$ for $d>2$ is not yet known. In fact, the proof in \cite{yam2024} for this existence highly relies on the transience of the graph, which indicates that an extremely faraway vertex has almost no influence on the harmonic measure at $\bm{0}$. This property can exclude all sets with excessively large diameters from the collection of candidates for the minimizing set and thus ensures its existence. However, for the two-dimensional lattices considered in this paper, i.e. $\mathscr{G}\in \{\mathbb{Z}^2,\mathscr{T}, \mathscr{H}\}$, it is not a priori clear that a connected set would always be preferable to a sparse set for achieving the minimum. This is because, due to the recurrence of such $\mathscr{G}$, adding a faraway vertex to a set containing $\bm{0}$ will reduce the harmonic measure at ${\bm 0}$ by a non-negligible constant factor (approximately $1/2$ by symmetry).

An intriguing conclusion from Theorem \ref{theorem1.1} is that for $\mathscr{G}\in \{\mathbb{Z}^2,\mathscr{T}, \mathscr{H}\}$, there exists an increasing sequence of integers $\{n_k\}_{k\ge1}$ with $n_{k+1}\le n_k+C\sqrt{n_k}$ such that every $\mathcal{M}_{n_k}(\mathscr{G})$ can be achieved by some $n_k$-point set $M_{n_k}$. To see this, suppose that there is no minimizing set of cardinality $n$, then we can find a sequence of $n$-point sets $\{M_{n}^{(m)}\}_{m\ge 1}$ with increasing diameters such that $\lim_{m\to \infty}\mathbb{H}_{M_{n}^{(m)}}({\bf 0})=\mathcal{M}_n(\mathscr{G})$. A moment of thought shows that for each $m\ge 1$, it is feasible to find a subset of $M_{n}^{(m)}\setminus \{\bm{0}\}$ whose distance to its complement in $M_{n}^{(m)}$ converges to infinity as $m$ increases. By the recurrence and symmetry of $\mathscr{G}$, removing such a distant subset will approximately double the harmonic measure at $\bm{0}$ (see Lemma \ref{lemma_remove_distant}), which implies $2\mathcal{M}_n(\mathscr{G})\ge \mathcal{M}_{n-1}(\mathscr{G})$. Combined with Theorem \ref{theorem1.1}, this yields that in the interval $[n,n+C\sqrt{n}]$ (where $C$ is a sufficiently large constant), there is at least one integer $n_\dagger$ such that the minimizing set for $\mathcal{M}_{n_\dagger}(\mathscr{G})$ exists (otherwise, one has $2^{C\sqrt{n}}\mathcal{M}_{n+C\sqrt{n}}(\mathscr{G})\ge \mathcal{M}_{n}(\mathscr{G})$, which together with $\lambda(\mathscr{G})>2$ causes a contradiction to Theorem \ref{theorem1.1}). In conclusion, the existence stated in Conjeture \ref{conj_existence} holds at least for a sequence of integers increasing with a limited speed. For global existence, it requires a more careful consideration of the influence that a vertex near $\bm{0}$ has on the harmonic measure at $\bm{0}$.

\textbf{Equivalent definition of $\mathcal{M}_n$.} Unless otherwise stated, throughout this paper we assume that $\mathscr{G}\in \{\mathbb{Z}^2,\mathscr{T}, \mathscr{H}\}$, and denote its vertex set and edge set by $\mathfs{V}$ and $\mathfs{E}$ respectively. Let $\bm{0}$ represent the origin of $\mathscr{G}$. Since $\mathscr{G}$ is vertex-transitive, 
\begin{equation}
	\mathbb{H}_A(y)= \mathbb{H}_{A-y}(\bm{0}), \ \ \forall A\subset \mathfs{V}\ \text{and}\ y\in A,
\end{equation}
where $A-y:= \{z-y: z\in A\}$. Thus, $\mathcal{M}_n$ in (\ref{def_Mn}) can be
equivalently defined as 
\begin{equation}\label{new_M_n}
	\mathcal{M}_n(\mathscr{G})= \inf\nolimits_{A\in \mathcal{A}_n(\mathscr{G})} \mathbb{H}_A(\bm{0}),
\end{equation}
where $\mathcal{A}_n(\mathscr{G}):= \{A\subset \mathfs{V}: |A|=n, \mathbb{H}_A(\bm{0})>0\}$.

\textbf{Statements about constants.} We use notations $C,C',c,c',...$ to represent local constants that vary depending on the context. Moreover, we employ the numbered notations $C_1,C_2,c_1,c_2,...$ to designate global constants, which stay unchanged throughout the paper. For the sake of clarity, we assign the uppercase letter $C$ (potentially with subscripts or superscripts) to denote large constants, while the lowercase letter $c$ is used to denote small ones.

\textbf{Notations for up-to-constant bounds.} For functions $f,g:\mathbb{R}\to \mathbb{R}$, we say $f$ is bounded from below (resp. above) by $g$ up to a constant if there exists $c>0$ (resp. $C>0$) such that $f\ge cg$ (resp. $f\le Cg$), which is written as $f\gtrsim g$ (resp. $f\lesssim g$). In addition, we denote $f\asymp g$ in the case when $g\lesssim f\lesssim g$.

\section{Proof outline of Theorem \ref{theorem1.1}}\label{subsection_outline}

According to (\ref{new_M_n}), to establish the upper bounds in Theorem \ref{theorem1.1}, it suffices to find a family of sets $A_n\in \mathcal{A}_n(\mathscr{G})$ such that $\mathbb{H}_{A_n}(\bm{0})\le [\lambda(\mathscr{G})]^{-n+C\sqrt{n}}$ for $n\ge 1$. For $\mathscr{G}=\mathbb{Z}^2$, we choose $A_n$ as the square spiral $D_n^{\mathbb{Z}^2}$ (see Figure \ref{fig:spiral_Z2}). Referring to (\ref{new_1.6}), it was already proved in \cite{Calvert2021CollapseAD} that $\mathbb{H}_{D_n^{\mathbb{Z}^2}}(\bm{0})\le [\lambda(\mathscr{G})]^{-n+o(n)}$, which is almost sufficient for the upper bounds in Theorem \ref{theorem1.1}, albeit with insufficient control over the second leading term $o(n)$ in the exponent. As demonstrated in Section \ref{section1.2.1}, through a careful computation (see Lemma \ref{lemma_length_tunnel}) on the length of the tunnel (i.e. the pink path in Figure \ref{fig:spiral_Z2}), we improve the previous upper bound for $D_n^{\mathbb{Z}^2}$ to $\mathbb{H}_{D_n^{\mathbb{Z}^2}}(\bm{0})\le [\lambda(\mathscr{G})]^{-n+C\sqrt{N}}$ and thus obtain the upper bound in Theorem \ref{theorem1.1} for $\mathscr{G}=\mathbb{Z}^2$. Furthermore, for the case when $\mathscr{G}\in \{\mathscr{T},\mathscr{H}\}$, we construct the analogues $D_n^{\mathscr{T}}$ and $D_n^{\mathscr{H}}$ (see Figure \ref{fig:spirals}) of the square spiral $D_n^{\mathbb{Z}^2}$. By applying similar arguments as in the case $\mathscr{G}=\mathbb{Z}^2$, we establish that $\mathbb{H}_{D_n^{\mathscr{G}}}(\bm{0})\le [\lambda(\mathscr{G})]^{-n+C\sqrt{N}}$ for $\mathscr{G}\in \{\mathscr{T},\mathscr{H}\}$, thereby concluding the upper bounds in Theorem \ref{theorem1.1}.

As for the lower bounds in Theorem \ref{theorem1.1}, their proofs are significantly more challenging and occupy the majority of this paper. Unlike bounding $\mathcal{M}_n(\mathscr{G})$ from above, exploring some specific examples in $\mathcal{A}_n(\mathscr{G})$ is no longer sufficient here. Instead, it requires to confirm that $[\lambda(\mathscr{G})]^{-n+c\sqrt{n}}$ is a uniform lower bound for harmonic measures at $\bm{0}$ of all sets in $\mathcal{A}_n(\mathscr{G})$, which are not necessarily connected and can be arbitrarily spread out. However, in two-dimensional lattices, it is indeed not the case that scattered vertices (i.e. vertices that are far away from the others) can be easily ignored in the computation of harmonic measures. To see this, consider the scenario where $A\in \mathcal{A}_n(\mathscr{G})$ ($\mathscr{G}\in \{\mathbb{Z}^2,\mathscr{T},\mathscr{H}\}$) consists of a giant set $A'$ and a single vertex $x$ such that $A'$ and $x$ are far away from each other. Due to the recurrence and translation invariance of $\mathscr{G}$, the harmonic measure of $A$ at $x$ is approximately $1/2$ (see e.g. \cite[Section 2.5]{lawler2013intersections}). In other words, adding $x$ to $A'$ will roughly decrease the harmonic measure of $A'$ by a factor of $1/2$. This phenomenon indicates that in the analysis of the harmonic measure of a large set, there is a significantly strong correlation between different clusters (we employ ``cluster'' as a synonym of ``connected component'', unless stated otherwise) within the set, which undoubtedly increases the difficulty of estimation. Actually, in our previous work with Kozma \cite{psi}, which proved that $e^{-Cn}\le \mathcal{M}_n(\mathscr{G})\le e^{-cn}$, the same challenge has been encountered. The way we solved it is to remove a vertex (or an entire cluster) from the set and then estimate how such a removal influences (or more accurately, increases) the harmonic measure at some fixed vertex. Inspired by this approach, in the first step of our proof for the lower bounds in Theorem \ref{theorem1.1}, we aim to reduce a general set (which may be very spread out) to a denser set by some proper removals, thereby simplifying the subsequent estimates.

\noindent (Warning: for clarity, the arguments in this section may slightly differ from those in the formal proofs presented in later sections.)


\vspace*{0.05cm}

\textbf{Remove all distant subsets} (achieved by Step 1 in Section \ref{section_outline}). For any $A\in \mathcal{A}_n(\mathscr{G})$, we consider the ``distant subset'' whose distance to its complement in $A$ is significantly larger than its diameter. By a quantitative removal argument inspired by \cite{psi}, we show that removing a distant subset $D$ with $\bm{0}\notin D$ from $A$ will increase the harmonic measure at $\bm{0}$ by at most a factor of $2+\epsilon$ for some $\epsilon>0$ (i.e. $\frac{\mathbb{H}_{A\setminus D}(\bm{0})}{\mathbb{H}_{A}(\bm{0})}\le 2+\epsilon$), which is much smaller than the corresponding quantity expected by Theorem \ref{theorem1.1} for a subset of the same cardinality as $D$, namely, $[\lambda(\mathscr{G})]^{|D|+o(|D|)}$ (see Lemma \ref{lemma_remove_distant}). We remove distant subsets from $A$ repeatedly until the set remains unchanged (note that each removal may generate new distant subsets).

\vspace*{0.05cm}

After applying the removal scheme in the first step, we end up with a set $A'$ with no distant subset, which will be dubbed as a ``packed set" (see Definition \ref{def_sparse}). Intuitively, the absence of distant subset indeed limits the sparsity of $A'$ by requiring all subsets of $A'$ to stay close to each other in pairs. In Lemma \ref{lemma_packed3}, we reflect such a restriction on pairwise distances into an inductive proof, which shows that that the diameter of any packed $A'$ is at most $|A'|^C$. As a result, in the subsequent proof it suffices to estimate the harmonic measure of a set $A'\subset B(n^C)$, where $B(n^C)$ is the Euclidean ball of radius $n^C$ centered at $\bm{0}$.

	\vspace*{0.05cm}

\textbf{Remove the clusters in a specific angle} (achieved by Steps 2 and 3 in Section \ref{section_outline}). We aim to create a vacant corridor in the box $B(n^C)$ through the removal argument so that the random walk starting from a faraway vertex can easily reach a small box $B(c\sqrt{n})$ without hitting $A'$. Precisely, we divide the annulus $B(n^C)\setminus B(c\sqrt{n}) $ into $C'\sqrt{n}$ equal sectors and find the one such that the clusters in $A'$ intersecting it have the smallest total cardinality, which is $O(\sqrt{n})$ according to the pigeonhole principle. (Actually, to avoid duplicate counting during the implementation of the pigeonhole principle, we need to ensure that every cluster in $A'$ cannot intersect too many sectors at the same time. As shown in Section \ref{section_outline}, we achieve this by removing all large clusters from $A'$ in advance.) Then we obtain $A''$ by removing all clusters that intersect this chosen sector from $A'$, and derive $\frac{\mathbb{H}_{A''}(\bm{0})}{\mathbb{H}_{A'}(\bm{0})}\le e^{O(\sqrt{n})}$ using the technique in \cite[Section 3.3]{psi} (where each removed vertex costs a constant factor on the upper bound). Note that the factor $e^{O(\sqrt{n})}$ can be swallowed by the second leading term in the exponent of our desired bound.

\vspace*{0.05cm}

A key ingredient of the remaining proof is to estimate the probability of bypassing a given cluster for a random walk (see Lemma \ref{lemma_prob_surrounding}). In other words, for any $m\in \mathbb{N}^+$, we aim to compute the smallest probability that a random walk, starting from an arbitrary vertex $x$ in the outer boundary of a connected $m$-point set $D$, reaches another predetermined vertex $y$ in the outer boundary before hitting $D$. A natural way to derive such a bound is to construct a proper path connecting $x$ and $y$ within the outer boundary of $D$, and then force the random walk to move in a ``one-dimensional" manner along this path. In fact, the construction of such a proper bypass intrinsically depends on the geometry of the graph. Specifically, in Section \ref{section_length_bypass} we establish an upper bound for the length of a circuit within the outer boundary of a set (see Proposition \ref{lemma_length_bypass} and Lemma \ref{lemma_exist_sl}). This estimate can be viewed as a quantitative extension of earlier results on boundary connectivity (see e.g. \cite[Lemma 2.1]{deuschel1996surface}, \cite[Lemma 2.23]{kesten1984aspects} and \cite{timar2013boundary}), and is interesting in its own right as a basic property of the graph.

	\vspace*{0.05cm}

	\textbf{Reach the origin through the vacant sector} (achieved by Step 4 in Section \ref{section_outline}). To derive a lower bound for $\mathbb{H}_{A''}(\bm{0})$ (recall that $A''$ is the set constructed in the previous step with a vacant sector), we force the random walk starting from a faraway vertex to first hit the outer arc of the vacant sector (with probability $c'n^{-\frac{1}{2}}$ by symmetry) and then reach $B(c\sqrt{n})$  through the vacant sector (with probability at least $n^{-C''}$; see (\ref{ineq_I2})). Next, let the random walk move along the geodesic to $\bm{0}$ (whose length is $O(\sqrt{n})$) while bypassing every encountered cluster of $A''$ in turn (see (\ref{ineq_I3}) and (\ref{ineq_I4})). This happens with probability at least $[\lambda(\mathscr{G})]^{-|A''|+c''\sqrt{|A''|}}$ by the estimate for the bypass probability presented in the last paragraph.

	\vspace*{0.05cm}

	In conclusion, combining these estimates, we obtain the desired low bound for the harmonic measure at $\bm{0}$ for all $A\in \mathcal{A}_n(\mathscr{G})$ and thus conclude Theorem \ref{theorem1.1}.

\subsection{Organization of the paper}

In Section \ref{section_notation}, we fix some necessary notations and review some useful results. The upper bounds in Theorem \ref{theorem1.1} are established in Section \ref{section1.2.1}. In Section \ref{section_outline}, we present a conditional proof for the lower bounds in Theorem \ref{theorem1.1} assuming several technical lemmas, namely, Lemmas \ref{lemma_remove_distant}, \ref{lemma_packed3}, \ref{lemma_remove_interlock} and \ref{lemma_prob_surrounding}. We establish Lemma \ref{lemma_packed3} in Section \ref{section_diameter} to bound the diameters of packed sets. In Section \ref{section_length_bypass}, we first present Proposition \ref{lemma_length_bypass}, which is a crucial geometric property concerning the bypass of a cluster, and then use it to derive Lemma \ref{lemma_prob_surrounding}. Finally, we confirm Lemmas \ref{lemma_remove_distant} and \ref{lemma_remove_interlock} in Section \ref{section_qra}, thereby concluding Theorem \ref{theorem1.1}.

	\section{Preliminaries}\label{section_notation}

To facilitate the exposition, in this section we collect some necessary
notations and useful results for graphs and random walks.

\subsection{Basic notations and properties for graphs} 
\begin{itemize}
	
	\item  \textbf{Degree.} We denote by $\mathrm{deg}(\mathscr{G})$ the degree of each vertex on $\mathscr{G}$ (especially, $\mathrm{deg}(\mathbb{Z}^2)=4$, $\mathrm{deg}(\mathscr{T})=6$ and $\mathrm{deg}(\mathscr{H})=3$).

	\item \textbf{Oriented edge.} For any $x\sim y\in \mathfs{V}$, let $\vec{e}_{x,y}$ be the oriented edge starting from $x$ and ending at $y$. Let $\vec{\mathfs{E}}$ be the collection of all oriented edges.

	\item  \textbf{Path.} A path $\eta$ on $\mathscr{G}$ is a sequence of vertices $(x_0,...,x_n)$ such that $x_i\sim x_{i+1}$ for all $0\le i\le n-1$. The length of $\eta$ is $\mathbf{L}(\eta):=n$, and the range of $\eta$ in $\mathfs{V}$ (resp. in $\vec{\mathfs{E}}$) is $\mathbf{R}^{\mathrm{v}}(\eta):=\{x_i\}_{i=0}^{n}$ (resp. $\mathbf{R}^{\mathrm{e}}(\eta):=\big\{\vec{e}_{x_i,x_{i+1}}\big\}_{i=0}^{n-1}$). We also write $\eta(i):=x_i$ for $0\le i\le n$. For convenience, let $\eta(-1):=x_n$ be the last vertex of $\eta$. Note that $(x)$ is a path of length $0$, whose range in $\mathfs{V}$ (resp. in $\vec{\mathfs{E}}$) is $\{x\}$ (resp. $\emptyset$).

	\item  \textbf{Sub-path.} For any path $\eta$ and $0\le t_1\le t_2\le \mathbf{L}(\eta)$, we define the sub-path $\eta[t_1,t_2]$ of $\eta$ as the path of length $t_2-t_1$ satisfying
	\begin{equation*}
		\eta[t_1,t_2](s)= \eta(t_1+s),\ \forall\ 0\le s\le t_2-t_1.
	\end{equation*}
	For completeness, we set $\eta[t,t']=\emptyset$ when $t>t'$ or $\{t,t'\}\not \subset [0, \mathbf{L}(\eta)]$.

	\item  \textbf{Reversed path.} For any path $\eta$, the reversed path of $\eta$ is defined as 
	\begin{equation*}
		\cev{\eta}:=\eta(\mathbf{L}(\eta)-s),\ \forall\ 0\le s\le  \mathbf{L}(\eta). 
	\end{equation*}

	\item  \textbf{Concatenation.} For any paths $\eta_1,\eta_2$ with $\eta_1(-1)=\eta_2(0)$, define the concatenation of $\eta_1$ and $\eta_2$ as $\eta_1\circ \eta_2:=(\eta_1(0),...,\eta_1(-1),\eta_2(1),...,\eta(-1))$. To avoid confusion, we set $\eta \circ \eta' =\eta' \circ \eta = \eta$ when $\eta'=\emptyset$.

	\item  \textbf{Edge circuit.} An edge circuit is a path $\eta$ such that $\eta(0)=\eta(-1)$ and
	$$
	\vec{e}_{\eta(i),\eta(i+1)}\neq \vec{e}_{\eta(j),\eta(j+1)},  \ \forall 0\le i<j \le \mathbf{L}(\eta)-1.
	$$
	Note that $\mathbf{L}(\eta)=|\mathbf{R}^{\mathrm{e}}(\eta)|$ for any edge circuit $\eta$.

	\item  \textbf{Graph distance.} For any $A_1,A_2,F\subset \mathfs{V}$, the graph distance between $A_1$ and $A_2$ in $F$ is $\mathbf{d}^{F}(A_1,A_2):= \inf\{n\ge 0:\exists\ \text{path}\ \eta\ \text{such that}\ \eta(0)\in A_1,\eta(-1)\in A_2,\mathbf{R}^{\mathrm{v}}(\eta)\subset F,\mathbf{L}(\eta)=n\}$ (set $\inf \emptyset =\infty$ for completeness). When $F=\mathfs{V}$, we abbreviate $\mathbf{d}(A_1,A_2):=\mathbf{d}^{\mathfs{V}}(A_1,A_2)$. Moreover, when $A_i= \{x\}$ for some $i\in \{1,2\}$ and $x\in \mathfs{V}$, we may omit the braces.

	\item  \textbf{Connectivity.} For any $A,F\subset \mathfs{V}$, we say $A$ is connected in $F$ if $$\mathbf{d}^{A\cap F}(x_1,x_2)<\infty, \ \ \forall x_1,x_2\in A.$$
	When $F=\mathfs{V}$, we omit the phrase ``in $F$''. As a default, $\emptyset$ is connected.

	\item  \textbf{Diameter.} For any $A\subset \mathfs{V}$, the diameter of $A$ is defined as 
	$$\mathrm{diam}(A):= \max\nolimits_{x_1,x_2\in A} \mathbf{d}(x_1,x_2). $$

	\item  \textbf{Cluster.} For any $A\subset \mathfs{V}$ and non-empty, connected $A'\subset A$, we say $A'$ is a cluster of $A$ if $\mathbf{d}^A(x,A')=\infty$ for all $x\in A\setminus A'$.

	\item \textbf{Complement and two specific clusters.} For any $A\subset \mathfs{V}$, we denote its complement by $A^c:= \mathfs{V}\setminus A$. When $A$ is finite, $A^c$ contains a unique infinite cluster, which we denote by $A_{\infty}^c$. We also denote by $A^c_{\bm{0}}$ the cluster of $A^c\cup \{\bm{0}\}$ containing $\bm{0}$.

	\item  \textbf{Neighborhood.} For any $x\in \mathfs{V}$, the neighborhood of $x$ is $N(x):=\{y\in \mathfs{V}:y\sim x\}$. For any finite $A \subset \mathfs{V}$ containing $x$, we define the neighborhood (resp. exterior neighborhood, $\bm{0}$-exposed neighborhood) of $x$ outside $A$ as $N^A(x):=N(x)\cap A^c$ (resp. $N^A_{\infty}(x):=N(x)\cap A^c_{\infty}$, $N^A_{\bm{0}}(x):=N(x)\cap A^c_{\bm{0}}$).

	\item  \textbf{Boundary.} For any finite $A\subset \mathfs{V}$, the outer (resp. inner) boundary of $A$ is $\partial^{\mathrm{o}}A:= \cup_{x \in A  }N^A(x)$ (resp. $\partial^{\mathrm{i}}A:= \{x\in A: N^A(x)\neq \emptyset\}$). The exterior outer (resp. inner) boundary of $A$ is defined as $\partial_{\infty}^{\mathrm{o}}A:= \cup_{x \in A  }N^A_{\infty}(x)$ (resp. $\partial_{\infty}^{\mathrm{i}}A:= \{x\in A: N^A_{\infty}(x)\neq \emptyset\}$). We define the $\bm{0}$-exposed inner boundary of $A$ as $\partial_{\bm{0}}^{\mathrm{i}}A:= \{x\in A: N^A_{\bm{0}}(x)\neq \emptyset\}$.

	\item  \textbf{Planar Embedding.} Let $\mathbf{I}:\mathscr{G}\to \mathbb{R}^2$ be the canonical embedding where the distance between the endpoints of every edge equals $1$ (see Figure \ref{fig:lattices}). For any edge $e=\{x,y\}\in \mathfs{E}$, $\mathbf{I}(e)$ (i.e. the image of $e$ under $\mathbf{I}$) is defined as the line segment on $\mathbb{R}^2$ with endpoints $\mathbf{I}(x)$ and $\mathbf{I}(y)$. For any path $\eta$, $\mathbf{I}(\eta)$ (i.e. the image of $\eta$ under $\mathbf{I}$) is defined as the broken line on $\mathbb{R}^2$ connecting $\eta(i)$ for $0\le i\le \mathbf{L}(\eta)$ successively.

	\item \textbf{Balls.} For $x,y\in \mathfs{V}$, we denote the Euclidean distance between $\mathbf{I}(x)$ and $\mathbf{I}(y)$ by $|x-y|$. For $x\in \mathfs{V}$ and $R\ge 0$, let $B_x(R):= \{y\in \mathfs{V}: | x-y |\le R\}$ (resp. $\mathbf{B}_x(R):= \{y\in \mathfs{V}: \mathbf{d}(x,y)\le R\}$) be the ball of the Euclidean distance (resp. graph distance) with center $x$ and radius $R$. When $x=\{\bm{0}\}$, we may omit the subscript and denote $B(R):= B_{\bm{0}}(R)$, $\mathbf{B}(R):= \mathbf{B}_{\bm{0}}(R)$. There exist $\Cl\label{ball_1}(\mathscr{G})> 1>\cl\label{ball_2}(\mathscr{G})>0$ such that
	\begin{equation}\label{ineq_ball}
		B_x(\cref{ball_2}R)	\subset 	\mathbf{B}_x(R) \subset 	B_x(\Cref{ball_1}R),\ \ \forall  x\in \mathfs{V}\ \text{and} \ R>0. 
	\end{equation}

	\item   \textbf{Face.} $\mathbb{R}^2\setminus \cup_{e\in \mathfs{E}} \mathbf{I}(e)$ is composed of countably many regions. Each region is called a face. As shown in Figure \ref{fig:lattices}, for $\mathbb{Z}^2$ (resp. $\mathscr{T}$, $\mathscr{H}$), every face is a square (resp. triangle, hexagon). For any face $\mathcal{S}$, let $\mathbf{v}(\mathcal{S})$ be the collection of vertices surrounding $\mathcal{S}$. We also denote by $\vec{\mathbf{e}}(\mathcal{S})$ the collection of oriented edges $\vec{e}_{x,y}$ with $x,y\in \mathbf{v}(\mathcal{S})$ such that $y$ is the next vertex of $x$ within $\mathbf{v}(\mathcal{S})$ in the clockwise direction.

Let $\xi(\mathscr{G})$ be the number of edges surrounding a face of $\mathscr{G}$. As shown in Figure \ref{fig:lattices}, one has $\xi(\mathbb{Z}^2)=4$, $\xi(\mathscr{T})=3$ and $\xi(\mathscr{H})=6$. Since each face of $\mathscr{G}$ is a regular polygon, its interior angle $\theta(\mathscr{G})$ equals $\frac{\pi[\xi(\mathscr{G})-2]}{\xi(\mathscr{G})}$. Meanwhile, since each vertex is covered by exactly $\mathrm{deg}(\mathscr{G})$ faces, one has $\theta(\mathscr{G})=\frac{2\pi}{\mathrm{deg}(\mathscr{G})}$. Combining these two equalities for $\theta(\mathscr{G})$, we get 
     \begin{equation}\label{final_3.2}
     	[\xi(\mathscr{G})-2]\cdot [\mathrm{deg}(\mathscr{G})-2]=4.
     \end{equation}
     Notably, (\ref{final_3.2}) implies that $\mathbb{Z}^2$, $\mathscr{T}$ and $\mathscr{H}$ correspond to the only three regular tessellations in $\mathbb{R}^2$.

	\item  \textbf{$*$-Adjacency.} For any $x,y\in \mathfs{V}$, we say $x$ and $y$ are $*$-adjacent (written as $x\sim_* y$) if there exists a face $\mathcal{S}$ such that $x,y\in \mathbf{v}(\mathcal{S})$. By replacing ``$\sim$'' with ``$\sim_*$'', we get the definitions of $*$-paths, $*$-connected sets, $*$-clusters, $*$-neighborhoods ($N_*$ and $N_*^A$) and $*$-boundaries ($\partial_*^{\mathrm{o}}$ and $\partial_*^{\mathrm{i}}$).

	Note that $\mathbf{d}(x,y)\le 3$ holds for all $\mathscr{G}\in \{\mathbb{Z}^2,\mathscr{T},\mathscr{H}\}$ and all $x\sim_*y\in \mathfs{V}$. Therefore, for any $*$-connected $A\subset \mathfs{V}$, we have 
	\begin{equation}\label{ineq_2.4}
		\mathrm{diam}(A)\le 3|A|. 
	\end{equation}

	\begin{lemma}\label{new7.1}
    For any $*$-connected $A\subset \mathfs{V}$ and $x\in A$ with $|A|\ge 2$, 
	\begin{equation}
		N_{\infty}^{A}(x)\le \mathrm{deg}(\mathscr{G})-\mathbbm{1}_{\mathscr{G}=\mathscr{T}}.
	\end{equation}
\end{lemma}
	\begin{proof}
	When $\mathscr{G}\in \{\mathbb{Z}^2,\mathscr{H}\}$, (\ref{new7.1}) directly follows from facts that $N_{\infty}^{A}(x)\subset N(x)$ and $|N(x)|=\mathrm{deg}(\mathscr{G})$. When $\mathscr{G}=\mathscr{T}$, since $A$ is $*$-connected and $|A|\ge 2$, there exists $v\in A$ such that $v\sim x$ (on $\mathscr{T}$, ``$v\sim x$'' and ``$v\sim_* x$'' are equivalent). This implies that $|N_{\infty}^{A}(x)|\le |N(x)\setminus \{v\}|\le \mathrm{deg}(\mathscr{G})-1$. 
\end{proof}

	\item   \textbf{Exterior $*$-neighborhood and $*$-boundary.} For any $A\subset \mathfs{V}$ and $x\in A$, we denote $N_{\infty,*}^A(x):=N_{*}(x) \cap A_{\infty}^c$. Let $\partial_{\infty,*}^{\mathrm{o}}A:= \cup_{x \in A  }N^A_{\infty,*}(x)$.

		The following lemma directly follows from \cite[Theorems 3 and 4]{timar2013boundary}. 
	
	\begin{lemma}\label{lemma2.2}
	For any $*$-connected $A\subset \mathfs{V}$, one has 
	\begin{equation*}
		(1)\ \partial_{\infty}^{\mathrm{i}} A\ \text{is}\ *\text{-connected}; \ \ (2)\ \partial_{\infty,*}^{\mathrm{o}}A\ \text{is}\ \text{connected}.
	\end{equation*}
	\end{lemma}

	Note that $|N_*(z)|\le 12$ for all $\mathscr{G}\in \{\mathbb{Z}^2,\mathscr{T},\mathscr{H}\}$ and all $z\in \mathfs{V}$. Therefore, for any $*$-connected $A\subset \mathfs{V}$ and $x,y\in\partial_{\infty,*}^{\mathrm{o}} A$, by Item (2) of Lemma \ref{lemma2.2},
	\begin{equation}\label{final_3.5_boundary}
		\mathbf{d}^{\partial_{\infty,*}^{\mathrm{o}} A}(x,y)\le |\partial_{\infty,*}^{\mathrm{o}} A| \le \sum\nolimits_{z\in A} |N_*(z) |\le 12|A|. 
	\end{equation}

\end{itemize}

\subsection{Properties of the random walk}

In this subsection, we will cite some lemmas from \cite{psi,lawler2010random,popov2021two} concerning random walks. As stated by the remark in \cite[Page 10]{lawler2010random}, although these lemmas are presented and proved for $\mathbb{Z}^d$ ($d\ge 2$), it is not hard to extend them to other regular lattices. Therefore, here we just give the statements for $\mathscr{G}\in \{\mathbb{Z}^2, \mathscr{T},\mathscr{H}\}$ without repeating the proofs.

\textbf{Stopping times and Green's functions.} Recall that the law of the random walk $\{S_n\}_{n\ge 0}$ starting from $x\in \mathfs{V}$ is denoted by $\mathbb{P}_x$. Let $\mathbb{E}_x$ be the expectation under $\mathbb{P}_x$. For any non-empty $A\subset \mathfs{V}$, we also denote $\tau_A:=\inf\{n\ge 0:S_n\in A\}$ and $\tau_A^+:=\inf\{n\ge 1:S_n\in A\}$. When $A=\{y\}$ for some $y\in \mathfs{V}$, we may omit the braces. The Green's function for a non-empty set $A\subset \mathfs{V}$ is defined as 
\begin{equation}
	G_A(x,y):=\mathbb{E}_x\Big(\sum\nolimits_{0\le i\le \tau_A} \mathbbm{1}_{S_i=y} \Big), \ \ \forall x,y\in A^c. 
\end{equation}

\begin{lemma}[{\cite[Proposition 6.3.5]{lawler2010random}}]\label{lemma_green}
	There exists $\Cl\label{asym_green}(\mathscr{G})>0$ such that for any $R\ge 1$ and $x\in B(R)$, 
	\begin{equation}
		G_{\partial^{\mathrm{i}}B(R)}(\bm{0},x) = \Cref{asym_green}(\mathscr{G})\big[\ln(R)-\ln(|x|\vee 1)\big]+ O((|x|\vee 1)^{-1}).  
	\end{equation}
\end{lemma}

\begin{lemma}[{\cite[Proposition 6.4.1]{lawler2010random}}]\label{lemma_hit}
	For $R_2>R_1>0$ and $x\in B(R_2)\setminus B(R_1)$, 
	\begin{equation}
		\mathbb{P}_x\big( \tau_{B(R_1)\cup \partial^{\mathrm{i}}B(R_2)} = \tau_{\partial^{\mathrm{i}}B(R_2)} \big) =  \frac{\ln(|x|)-\ln(R_1)+ O(R_1^{-1})}{\ln(R_2)-\ln(R_1)}.  
	\end{equation}
\end{lemma}

\begin{lemma}[{\cite[Theorem 3.17]{popov2021two}}]\label{lemma_2srw_hm}
	For any $R\ge 1$, $A\subset B(R)$ and $x\in [B(3R)]^c$, 
	\begin{equation}
		\mathbb{P}_x(\tau_A=\tau_y)= \mathbb{H}_A(y)\cdot \bigg[ 1+O\bigg( \frac{R}{\mathbf{d}(x,A)}\bigg) \bigg].
	\end{equation}
\end{lemma}

\begin{lemma}[{\cite[Lemma 2.4]{psi}}]\label{lemma_last-exit}
	For any $A_1\subset A_2\subset \mathscr{G}$, $x\in A_2\setminus A_1$ and $y\in A_1$, 
	\begin{equation}
		\mathbb{P}_{x}\left(\tau_{A_1}= \tau_{y} \right) = \sum\nolimits_{v\in A_2\setminus A_1} G_{A_1}(x,v) \mathbb{P}_v\big(\tau^+_{A_2}=\tau_y \big). 
	\end{equation}
\end{lemma}

As shown in the next lemma, the harmonic measure of a ball is comparable to the uniform distribution on its inner boundary. 

\begin{lemma}\label{lemma_uniform}
	There exists constants  $\Cl\label{compare_uniform1}(\mathscr{G})>1>\cl\label{compare_uniform2}(\mathscr{G})>0$ such that for any $R>0$ and $y\in \partial^{\mathrm{i}}B(R)$, 
	\begin{equation}\label{new_2.11}
		\cref{compare_uniform2} R^{-1}	 \le \mathbb{H}_{B(R)}(y) \le \Cref{compare_uniform1} R^{-1}. 
	\end{equation}
\end{lemma}
\begin{proof}
	Without loss of generality, we assume that $R$ is sufficiently large. Arbitrarily take $x\in \partial^{\mathrm{i}}B(4R)$. For any $y\in \partial^{\mathrm{i}}B(R)$, by Lemma \ref{lemma_last-exit} (with $A_1=B(R)$ and $A_2=B(R)\cup \{x\}$), one has 
	\begin{equation}\label{new_2.12_1}
		\mathbb{P}_x\big(\tau_{B(R)}=\tau_{y}\big)= G_{B(R)}(x,x)   \mathbb{P}_x\big(\tau_{B(R)}=\tau_{y}<\tau_x^+\big). 
	\end{equation}
	In addition, by reversing the random walk, 
	\begin{equation}\label{new_2.12_2}
		\mathbb{P}_x\big(\tau_{B(R)}=\tau_{y}<\tau_x^+\big) = \mathbb{P}_y\big(\tau_{B(R)}^+>\tau_x\big). 
	\end{equation}
	For the probability on the right-hand side, by the strong Markov property, 
	\begin{equation}\label{new_2.13}
		\mathbb{P}_y\big(\tau_{B(R)}^+>\tau_x\big) = \sum\nolimits_{w\in \partial^{\mathrm{i}}B(2R)} \mathbb{P}_y\big(\tau_{B(R)}^+> \tau_{\partial^{\mathrm{i}} B(2R)}=\tau_w\big) \mathbb{P}_w\big(\tau_x<\tau_{B(R)} \big). 
	\end{equation}
	Moreover, by the discrete Harnack's inequality (see e.g. \cite[Theorem 6.3.9]{lawler2010random}), for an arbitrarily fixed $w_{\dagger}\in \partial^{\mathrm{i}}B(2R)$, 
	\begin{equation*}
		\mathbb{P}_w\big(\tau_x<\tau_{B(R)} \big) \asymp \mathbb{P}_{w_\dagger}\big(\tau_x<\tau_{B(R)} \big),\ \ \forall w\in \partial^{\mathrm{i}}B(2R). 
	\end{equation*}
	Combined with (\ref{new_2.12_1})-(\ref{new_2.13}), it implies 
	\begin{equation}\label{new_2.14}
		\mathbb{P}_x\big(\tau_{B(R)}=\tau_{y}\big) \asymp G_{B(R)}(x,x)  \mathbb{P}_{w_\dagger}\big(\tau_x<\tau_{B(R)}\big)\mathbb{P}_y\big(\tau_{B(R)}^+> \tau_{\partial^{\mathrm{i}} B(2R)}\big). 
	\end{equation}

	We claim that for any $y\in \partial^{\mathrm{i}}B(R)$, 
	\begin{equation}\label{new_2.15}
		\mathbb{P}_y\big(\tau_{B(R)}^+> \tau_{\partial^{\mathrm{i}} B(2R)}\big)  \asymp  R^{-1}. 
	\end{equation}
	In fact, by the strong Markov property, $\mathbb{P}_y\big(\tau_{B(R)}^+> \tau_{\partial^{\mathrm{i}} B(2R)}\big)$ equals
	\begin{equation}\label{new_2.18}
		\begin{split}
			\sum\nolimits_{z\in \partial^{\mathrm{i}} B(R+100)} \mathbb{P}_y\big(\tau_{B(R)}^+> \tau_{\partial^{\mathrm{i}} B(R+100)}= \tau_z \big)  	\mathbb{P}_z\big(\tau_{B(R)}> \tau_{\partial^{\mathrm{i}} B(2R)}\big).
		\end{split}	
	\end{equation}
	Moreover, by Lemma \ref{lemma_hit}, one has 
	\begin{equation}\label{new_2.17}
		\mathbb{P}_{z}\big(\tau_{B(R)}> \tau_{\partial^{\mathrm{i}} B(2R)} \big) \asymp R^{-1}, \ \ \forall z\in  \partial^{\mathrm{i}} B(R+100). 
	\end{equation} 
	On the one hand, it follows from (\ref{new_2.18}) and (\ref{new_2.17}) that 
	\begin{equation}\label{new_2.19}
		\mathbb{P}_y\big(\tau_{B(R)}^+> \tau_{\partial^{\mathrm{i}} B(2R)}\big) \le \max_{z\in \partial^{\mathrm{i}} B(R+100)}\mathbb{P}_z\big(\tau_{B(R)}> \tau_{\partial^{\mathrm{i}} B(2R)}\big) \lesssim R^{-1}. 
	\end{equation}
	On the other hand, since the random walk from $y$ may hit $\partial^{\mathrm{i}} B(R+100)$ before returning $B(R)$ within $O(1)$ steps, we have 
	\begin{equation}\label{new_2.20}
		\mathbb{P}_y\big(\tau_{B(R)}^+> \tau_{\partial^{\mathrm{i}} B(2R)}\big) \gtrsim \min_{z\in \partial^{\mathrm{i}} B(R+100)}\mathbb{P}_z\big(\tau_{B(R)}> \tau_{\partial^{\mathrm{i}} B(2R)}\big)  \gtrsim  R^{-1}. 
	\end{equation}
	By (\ref{new_2.19}) and (\ref{new_2.20}), we conclude the claim (\ref{new_2.15}).

	Combining (\ref{new_2.14}) and (\ref{new_2.15}), one has 
	\begin{equation*}
		\mathbb{P}_x\big(\tau_{B(R)}=\tau_{y}\big) \asymp \mathbb{P}_x\big( \tau_{B(R)}=\tau_{y'}\big), \ \ \forall y,y'\in \partial^{\mathrm{i}}B(R), 
	\end{equation*}
	which implies that 
	\begin{equation}\label{new_2.16}
		\mathbb{P}_x\big(\tau_{B(R)}=\tau_{y}\big)\asymp |\partial^{\mathrm{i}}B(R)|^{-1} \asymp  R^{-1}, \ \ \forall y \in \partial^{\mathrm{i}}B(R). 
	\end{equation}
	By (\ref{new_2.16}) and Lemma \ref{lemma_2srw_hm}, we obtain the desired bound (\ref{new_2.11}).
\end{proof}

\subsection{Estimates for traversing probabilities}

In the proof of Theorem \ref{theorem1.1}, we need to estimate the probability that a random walk moves along a tunnel (e.g. $D_n^{\mathbb{Z}^2}$ in Figure \ref{fig:spiral_Z2}). To achieve these estimates, we present the following inequality for the number sequences with specific recursive structure. 
\begin{lemma}\label{lemma_number_sequence}
	Let $\{i_1,...,i_k\}$ be an increasing sequence of integers in $[1,m]$ ($m\in \mathbb{N}^+$). Let $a>2$ and $b>1$. Assume that $\{q_i\}_{i=0}^{m+1}$ satisfies the following conditions:
	\begin{enumerate}
		\item $q_0=1$ and $0\le q_i\le 1$ for all $1\le i\le m+1$;

		\item When $1\le i\le m$ and $i\notin \{i_1,...,i_k\}$, $q_i\ge  a^{-1}(q_{i-1}+q_{i+1})$;

		\item When $1\le i\le m$ and $i\in \{i_1,...,i_k\}$, $q_i\ge   a^{-1}b(q_{i-1}+q_{i+1})$.

	\end{enumerate}
	Let $\alpha=\alpha(a):= \frac{a+\sqrt{a^2-4}}{2}$. Then we have 
	\begin{equation}
		q_{m}\ge  b^{k}(\alpha-1)\alpha^{-m-1}. 
	\end{equation}	
\end{lemma}
\begin{proof}
	For each $0\le i\le m+1$, we denote 
	\begin{equation*}
		\mu_i:= \big| \big\{ j: 1\le j\le i\ \text{and} \ j\in \{i_1,...,i_k\}  \big\}  \big|.
	\end{equation*}
	Note that $\mu_0=0$ and $\mu_m=k$. Let $\widetilde{q}_i:=b^{-\mu_i}q_i$. We claim that
	\begin{equation}\label{new_5.4}
		\widetilde{q}_i \ge a^{-1}\left(\widetilde{q}_{i-1}+\widetilde{q}_{i+1}\right), \ \forall 1\le i\le m. 
	\end{equation}
	In fact, if $i\notin \{i_1,...,i_k\}$, by Condition (2), $\mu_i=\mu_{i-1}$ and $b^{-\mu_{i}}\ge b^{-\mu_{i+1}}$, we have 
	\begin{equation*}
		\begin{split}
			\widetilde{q}_i = b^{-\mu_i} q_i \ge a^{-1}(b^{-\mu_{i-1}}q_{i-1}+b^{-\mu_{i+1}} q_{i+1})= a^{-1}(\widetilde{q}_{i-1}+\widetilde{q}_{i+1}). 
		\end{split}
	\end{equation*}
	Similarly, if $i \in \{i_1,...,i_k\}$, it follows from Condition (3) and $\mu_i=\mu_{i-1}+1$ that 
	\begin{equation*}
		\begin{split}
			\widetilde{q}_i  = b^{-\mu_i} q_i \ge &b^{-\mu_i}\cdot a^{-1}b(q_{i-1}+q_{i+1})\\
			\ge &a^{-1}(b^{-\mu_{i-1}}q_{i-1}+b^{-\mu_{i+1}} q_{i+1})=a^{-1}(\widetilde{q}_{i-1}+\widetilde{q}_{i+1}). 
		\end{split}
	\end{equation*}
	To sum up, we conclude the claim (\ref{new_5.4}).

	Recall that $\alpha= \frac{a+\sqrt{a^2-4}}{2}\in (1,a)$. (\ref{new_5.4}) implies that for any $1\le i\le m$, 
	\begin{equation}\label{new_2.23}
		\widetilde{q}_i- \alpha \widetilde{q}_{i-1} \ge \alpha (\widetilde{q}_{i+1}- \alpha \widetilde{q}_{i}). 
	\end{equation}
	Therefore, we have 
	\begin{equation}\label{new_2.24}
		\widetilde{q}_{1}- \alpha \widetilde{q}_{0} \ge 	\alpha ^{m}(\widetilde{q}_{m+1}- \alpha \widetilde{q}_{m}).  
	\end{equation}
	Combined with Condition (1), it gives
	\begin{equation}\label{new_5.5}
		\widetilde{q}_{m}\ge  \alpha^{-1}\big[\widetilde{q}_{m+1} -\alpha^{-m} (\widetilde{q}_{1}- \alpha \widetilde{q}_{0} )   \big] \ge (\alpha-1) \alpha^{-m-1}. 
	\end{equation}
	By (\ref{new_5.5}) and $\widetilde{q}_{m}=b^{-k}q_m$, we conclude this lemma.
\end{proof}

As a direct application of Lemma \ref{lemma_number_sequence}, we obtain the following estimate for the probability that a random walk moves along a given path. The analogue of this estimate for $\mathbb{Z}^d$ ($d\ge 2$) was presented in \cite[Lemma 4.3]{psi}.

\begin{lemma}\label{lemma_move_along}
	For any path $\eta$ on $\mathscr{G}$, we have 
	\begin{equation}\label{new_2.27}
		\mathbb{P}_{\eta(0)}\big( \tau_{\eta(-1)}<\tau_{\partial^{\mathrm{o}}\mathbf{R}^{\mathrm{v}}(\eta)}\big) \gtrsim \big[\widehat{\lambda}(\mathscr{G})\big]^{-\mathbf{L}(\eta)}, 
	\end{equation}
	where $\widehat{\lambda}(\mathbb{Z}^2)=2+\sqrt{3}$, $\widehat{\lambda}(\mathscr{T})=3+2\sqrt{2}$ and $\widehat{\lambda}(\mathscr{H})=\frac{3+\sqrt{5}}{2}$.  
\end{lemma}
\begin{proof}
	For $0\le i\le \mathbf{L}(\eta)$, we denote $q_i:=\mathbb{P}_{\eta(\mathbf{L}(\eta)-i)}\big( \tau_{\eta(-1)}<\tau_{\partial^{\mathrm{o}}\mathbf{R}^{\mathrm{v}}(\eta)} \big)$. Note that 
	\begin{equation*}
		q_{0}=\mathbb{P}_{\eta(-1)}\big( \tau_{\eta(-1)}<\tau_{\partial^{\mathrm{o}}\mathbf{R}^{\mathrm{v}}(\eta)} \big)=1\ \ \text{and}\ \  q_{\mathbf{L}(\eta)}=\mathbb{P}_{\eta(0)}\big( \tau_{\eta(-1)}<\tau_{\partial^{\mathrm{o}}\mathbf{R}^{\mathrm{v}}(\eta)} \big).
	\end{equation*}
	We also set $q_{\mathbf{L}(\eta)+1}:=0$. Then it directly follows that $\{q_i\}_{i=0}^{\mathbf{L}(\eta)+1}$ satisfies Condition (1) in Lemma \ref{lemma_number_sequence}. Moreover, by the Markov property, we have
	\begin{equation*}
		q_i\ge [\mathrm{deg}(\mathscr{G})]^{-1}(q_{i-1}+q_{i+1}), \ \ \forall 1\le i\le \mathbf{L}(\eta).
	\end{equation*} 
	I.e., Conditions (2) and (3) in Lemma \ref{lemma_number_sequence} hold with $a=\mathrm{deg}(\mathscr{G})$ and $k=0$. Thus, applying Lemma \ref{lemma_number_sequence} and noting that $\widehat{\lambda}(\mathscr{G})=\alpha(\mathrm{deg}(\mathscr{G}))$ for all $\mathscr{G}\in \{\mathbb{Z}^2, \mathscr{T},\mathscr{H} \} $, we conclude this lemma.
\end{proof}

As a supplement of Lemma \ref{lemma_number_sequence}, the subsequent lemma provides an upper bound for the number sequence with a similar structure as in Lemma \ref{lemma_number_sequence}.

\begin{lemma}\label{lemma_number_sequence2}
	For $a>2$ and $m\in \mathbb{N}^+$, if the number sequence $\{q_i\}_{i=0}^{m+1}$ satisfies the following conditions:
	\begin{enumerate}
		\item $q_0=1$, $0\le q_{m+1}\le (a-1)q_m$ and $0\le q_i\le 1$ for all $1\le i\le m$;

		\item $q_i=  a^{-1}(q_{i-1}+q_{i+1})$ for all $1\le i\le m$,

	\end{enumerate}
	then we have (recalling that $\alpha= \frac{a+\sqrt{a^2-4}}{2}$)
	\begin{equation}\label{final_3.25}
		q_{m}\le  (\alpha+1-a)^{-1}\alpha^{-m+1}. 
	\end{equation}	
\end{lemma}
\begin{proof}
	For any $1\le i\le m$, similar to (\ref{new_2.23}) and (\ref{new_2.24}), Condition (2) implies that $q_i- \alpha q_{i-1} = \alpha (q_{i+1}- \alpha q_{i})$. Combined with Condition (1), it yields
	\begin{equation}\label{final_3.26}
	- \alpha \le 	q_{1}- \alpha q_{0} = 	\alpha ^{m}(q_{m+1}- \alpha q_{m}) \le \alpha ^{m}(a-1-\alpha)q_m,
	\end{equation}
	This implies the desired bound (\ref{final_3.25}) since $a-1-\alpha=\frac{1}{2}(a-2-\sqrt{a^2-4})<0$.
\end{proof}

	\section{Proof of the upper bounds in Theorem \ref{theorem1.1}}\label{section1.2.1}

In this section, we establish the upper bounds in Theorem \ref{theorem1.1} through the analysis of harmonic measures of the square spiral and its analogues.

Recall the square spiral $D_n^{\mathbb{Z}^2}$ in Figure \ref{fig:spiral_Z2}. Inspired by $D_n^{\mathbb{Z}^2}$, we construct its analogues for $\mathscr{T}$ and $\mathscr{H}$, which we denote by $\{D_n^{\mathscr{T}}\}_{n\ge 2}$ and $\{D_n^{\mathscr{H}}\}_{n\ge 2}$ respectively, as shown in Figure \ref{fig:spirals}. Note that $|D_n^{\mathscr{T}}|=|D_n^{\mathscr{H}}|=n$. It directly follows from (\ref{new_M_n}) that $\mathcal{M}_n(\mathscr{G})\le  \mathbb{H}_{D_n^{\mathscr{G}}}(\bm{0})$ for all $\mathscr{G}\in \{\mathbb{Z}^2,\mathscr{T},\mathscr{H} \}$. Thus, for the upper bounds in Theorem \ref{theorem1.1}, it suffices to prove that for all sufficiently large $n$,
\begin{equation}\label{new_3.1}
	\mathbb{H}_{D_n^{\mathscr{G}}}(\bm{0}) \le   [\lambda (\mathscr{G})]^{-n+C\sqrt{n}}. 
\end{equation}

\include*{tikz_other_spirals}

According to Figures \ref{fig:spiral_Z2} and \ref{fig:spirals}, here are some useful observations on the random walk that starts from a faraway vertex and first hits $D_n^{\mathscr{G}}$ at $\bm{0}$: 
\begin{enumerate}
	\item[(a)]  It must cross a winding tunnel $\eta_n^{\mathscr{G}}$ (see pink paths in Figures \ref{fig:spiral_Z2} and \ref{fig:spirals}), which starts from $v_n^{\mathscr{G}}$ and ends at $\bm{0}$);

	\item[(b)]  When it is passing through the tunnel $\eta_n^{\mathscr{G}}$, at each position there are exactly two choices for the next step, namely going forward or backward along $\eta_n^{\mathscr{G}}$.

\end{enumerate}
By Observation (a) we have 
\begin{equation}\label{new_3_2}
	\mathbb{H}_{D_n^{\mathscr{G}}}(\bm{0}) \le  \mathbb{P}_{v_n^{\mathscr{G}}}\big(\tau_{D_n^{\mathscr{G}}}=\tau_{\bm{0}}\big).
\end{equation}
In fact, the probability $\mathbb{P}_{v_n^{\mathscr{G}}}\big(\tau_{D_n^{\mathbb{Z}^2}}=\tau_{\bm{0}}\big)$ can be estimated using Lemma \ref{lemma_number_sequence2} as follows. For $0\le i\le l_n^{\mathscr{G}}:= \mathbf{L}(\eta_n^{\mathscr{G}})$, we define $q_i:=\mathbb{P}_{\eta_n^{\mathscr{G}}(l_n^{\mathscr{G}}-i)}\big(\tau_{D_n^{\mathscr{G}}}=\tau_{\bm{0}}\big)$. Let 
\begin{equation}
	q_{l_n^{\mathscr{G}}+1}:=\sum\nolimits_{x\in [D_n^{\mathscr{G}}\cup \mathbf{R}^{\mathrm{v}}(\eta_n^{\mathscr{G}})]^c_{\infty}: x\sim v_n^{\mathscr{G}}}  \mathbb{P}_{x}\big(\tau_{D_n^{\mathscr{G}}}=\tau_{\bm{0}}\big).
\end{equation}
Note that $\{x\in [D_n^{\mathscr{G}}\cup \mathbf{R}^{\mathrm{v}}(\eta_n^{\mathscr{G}})]_{\infty}^c: x\sim v_n^{\mathscr{G}} \}\subset N(x)\setminus \{\eta_n^{\mathscr{G}}(1)\}$. Therefore, since a random walk starting from any vertex in $[D_n^{\mathscr{G}}\cup \mathbf{R}^{\mathrm{v}}(\eta_n^{\mathscr{G}})]_{\infty}^c$ must reach $v_n^{\mathscr{G}}$ before hitting $\bm{0}$ (as in Observation (a)), one has $q_{l_n^{\mathscr{G}}+1}\le [\mathrm{deg}(\mathscr{G})-1]q_{l_n^{\mathscr{G}}}$. Meanwhile, it follows from the Markov property and the definition of $q_{l_n^{\mathscr{G}}+1}$ that 
\begin{equation}
	q_{l_n^{\mathscr{G}}}=  [\mathrm{deg}(\mathscr{G})]^{-1}(q_{l_n^{\mathscr{G}}-1}+q_{l_n^{\mathscr{G}}+1}).
	\end{equation}
Moreover, by the Markov property and Observation (b), we have 
\begin{equation}
	q_i =  [\mathrm{deg}(\mathscr{G})]^{-1}(q_{i-1}+q_{i+1}), \ \ \forall 1\le i\le l_n^{\mathscr{G}}-1.
\end{equation}
In conclusion, the number sequence $\{q_i\}_{i=0}^{l_n^{\mathscr{G}}+1}$ satisfies all conditions in Lemma \ref{lemma_number_sequence2} (with $a=\mathrm{deg}(\mathscr{G})$) and thus (recalling that $\widehat{\lambda}(\mathscr{G})=\alpha(\mathrm{deg}(\mathscr{G}))$),
\begin{equation}
	q_{l_n^{\mathscr{G}}}= \mathbb{P}_{v_n^{\mathscr{G}}}\big(\tau_{D_n^{\mathscr{G}}}=\tau_{\bm{0}}\big) \lesssim  [\widehat{\lambda}(\mathscr{G})]^{-l_n^{\mathscr{G}}}. 
\end{equation}
Combined with (\ref{new_3_2}), this implies that 
\begin{equation}\label{new_3.5}
	\mathbb{H}_{D_n^{\mathscr{G}}}(\bm{0}) \lesssim  [\widehat{\lambda}(\mathscr{G})]^{-l_n^{\mathscr{G}}}. 
\end{equation}
To obtain (\ref{new_3.1}), it remains to establish a lower bound for $l_n^{\mathscr{G}}$.

\begin{lemma}\label{lemma_length_tunnel}
	There exists $C(\mathscr{G})>1$ such that for all sufficiently large $n$,  
	\begin{equation}\label{ln0}
		l_n^{\mathscr{G}}\ge  \nu(\mathscr{G}) n- C\sqrt{n},
	\end{equation}
	where $\nu(\mathscr{G}):=\frac{\mathrm{deg}(\mathscr{G})-2\cdot \mathbbm{1}_{\mathscr{G}=\mathscr{T}}}{\mathrm{deg}(\mathscr{G})-2}$, i.e., $\nu(\mathbb{Z}^2)=2$, $\nu(\mathscr{T})=1$ and $\nu(\mathscr{H})=3$.  
\end{lemma}
\begin{proof}
	Let $\mathfrak{T}_{n}^{\mathscr{G}}$ be the collection of all tuples $(x,y)$ satisfying $x\in \mathbf{R}^{\mathrm{v}}(\eta_{n}^{\mathscr{G}})$, $y\in D_n^{\mathscr{G}}$ and $x\sim y$. As shown in Figures \ref{fig:spiral_Z2} and \ref{fig:spirals}, for each $1\le i\le l_n^{\mathscr{G}}-1$, $\eta_{n}^{\mathscr{G}}(i)$ has exactly $\mathrm{deg}(\mathscr{G})-2$ neighbors in $D_n^{\mathscr{G}}$. Moreover, each of $\eta_{n}^{\mathscr{G}}(0)$ and $\eta_{n}^{\mathscr{G}}(-1)$ has no more than $\mathrm{deg}(\mathscr{G})$ neighbors in $D_n^{\mathscr{G}}$. These observations imply that  
	\begin{equation}\label{ln1}
		|\mathfrak{T}_{n}^{\mathscr{G}}|\le [\mathrm{deg}(\mathscr{G})-2](l_n^{\mathscr{G}}-1) +2\mathrm{deg}(\mathscr{G}).  
	\end{equation}

	Now we estimate $|\mathfrak{T}_{n}^{\mathscr{G}}|$ in another way. Referring to Figures \ref{fig:spiral_Z2} and \ref{fig:spirals}, the majority of vertices in $D_n^{\mathscr{G}}$ have exactly $\mathrm{deg}(\mathscr{G})-2\cdot \mathbbm{1}_{\mathscr{G}=\mathscr{T}}$ neighbors that are contained in $\mathbf{R}^{\mathrm{v}}(\eta_{n}^{\mathscr{G}})$, with the exception of vertices in the outermost layer and the corner vertices in $D_n^{\mathscr{G}}$. Moreover, the total number of these exceptional vertices is $O(\sqrt{n})$ since $\mathrm{diam}(D_n^{\mathscr{G}})$ is of order $\sqrt{n}$. As a result, we have (recalling that $|D_n^{\mathscr{G}}|=n$)
	\begin{equation}
		\begin{split}
			|\mathfrak{T}_{n}^{\mathscr{G}}|\ge & \big[\mathrm{deg}(\mathscr{G})-2\cdot \mathbbm{1}_{\mathscr{G}=\mathscr{T}}\big](|D_n^{\mathscr{G}}|-C\sqrt{n})\\
			\ge & \big[\mathrm{deg}(\mathscr{G})-2\cdot \mathbbm{1}_{\mathscr{G}=\mathscr{T}}\big]n-C'\sqrt{n}. 
		\end{split}
	\end{equation}
	Combined with (\ref{ln1}), it concludes this lemma. 
\end{proof}

By (\ref{new_3.5}) and Lemma \ref{lemma_length_tunnel}, we obtain (\ref{new_3.1}): for any sufficiently large $n\ge 1$, 
\begin{equation}
	\mathbb{H}_{D_n^{\mathscr{G}}}(\bm{0}) \le C \big[\widehat{\lambda}(\mathscr{G})\big]^{-\nu(\mathscr{G})n+C'\sqrt{n}}\le  [\lambda(\mathscr{G})]^{-n+C''\sqrt{n}}, 
\end{equation}
where in the second inequality we used the fact that $\lambda(\mathscr{G})= \big[\widehat{\lambda}(\mathscr{G})\big]^{\nu(\mathscr{G})}$ for all $\mathscr{G}\in \{\mathbb{Z}^2, \mathscr{T}, \mathscr{H}\}$. Recalling that (\ref{new_3.1}) is sufficient for the upper bounds in Theorem \ref{theorem1.1}, we complete the proof.         \qed

\section{Proof of the lower bounds in Theorem \ref{theorem1.1}}\label{section_outline}

In this section, we present a conditional proof for the lower bounds
in Theorem \ref{theorem1.1} assuming several technical lemmas, which will be established in later sections. Unless otherwise stated, we assume that $n$ is a sufficiently large integer. As described in Section \ref{subsection_outline}, our proof is conducted in the following steps.

\textbf{Step 1: remove all distant subsets.} In the first step, we aim to remove the subset of a set $A$ for which the distance between this subset and its complement in $A$ is significantly larger than its diameter. The following lemma, as a quantitative vesion of the removal argument in \cite{psi}, provides an upper bound for the increase of the harmonic measure after removing such a distant subset.

We denote $\mathcal{A}(\mathscr{G}):= \{A\subset \mathfs{V}: |A|\ge 2, \mathbb{H}_A(\bm{0})>0\}$.

\begin{lemma}\label{lemma_remove_distant}
	For any $\epsilon>0$, there exists $\Cl\label{remove_distant}(\mathscr{G},\epsilon)\ge 10$ such that for any $L\ge \Cref{remove_distant}$, $A\in \mathcal{A}(\mathscr{G})$ and $D\subsetneq A$ satisfying $\bm{0}\notin D$ and $\mathbf{d}(D,A\setminus D)\ge  L[\mathrm{diam}(D)\vee 1]$, 
	\begin{equation}\label{ineq_remove_distant}
		\frac{\mathbb{H}_{A\setminus D}(\bm{0})}{\mathbb{H}_{A}(\bm{0})} \le  (2+\epsilon)\cdot \frac{\ln(L)+\ln(\mathrm{diam}(D)\vee 1)}{\ln(L)}.
	\end{equation}
\end{lemma}

We postpone the proof of Lemma \ref{lemma_remove_distant} to Section \ref{section_qra_distant}. Note that the limit of Inequality (\ref{ineq_remove_distant}) as $L\to \infty$, which states that removing a subset that is extremely far away from its complement will at most double the harmonic measure, directly follows from the recurrence and translation invariance of $\mathscr{G}$. In \cite[Section 2.5]{lawler2013intersections}, this computation was discussed in detail and was further used to construct an extremely spread-out set in $\mathcal{A}_n(\mathbb{Z}^2)$ whose harmonic measure at $\bm{0}$ is $2^{-n+o(n)}$. See also \cite[Equation (1.12)]{psi} for more discussions on this example.

With the help of Lemma \ref{lemma_remove_distant}, we can remove a distant subset during the computation of the harmonic measure. To formulate such a removal process, we define the concepts of packed sets and sparse sets as follows. Let $\Cl\label{delicate_green}(\mathscr{G}):= \Cref{remove_distant}(\mathscr{G},0.1)$.

\begin{definition}\label{def_sparse}
	For any $A\subset \mathfs{V}$, we say $A$ is packed if $\mathfs{D}(A)=\emptyset$, where  
	\begin{equation}\label{def_mathfs_D}
		\mathfs{D}(A) := \big\{\text{non-empty}\ D\subsetneq A:\bm{0}\notin D,\mathbf{d}(D,A\setminus D)> \Cref{delicate_green}[\mathrm{diam}(D)\vee 1]\big\}.
	\end{equation}
	We also say $A$ is sparse if $\mathfs{D}(A)\neq \emptyset$.
\end{definition}


In what follows, we present two useful properties of packed sets and sparse sets.

\begin{lemma}\label{lemma_packed0}
	For any sparse $A\subset \mathfs{V}$, there exists a packed set $D\in \mathfs{D}(A)$. 
\end{lemma}

\begin{proof}
	We prove this lemma by induction. When $|A|=2$, say $A=\{x_1,x_2\}$ with $\mathbf{d}(x_1,x_2)> \Cref{delicate_green}$, we only need to take $D=\{x_1\}$. For any integer $k\ge 3$, assume that this lemma holds for all sparse $A'$ with $|A'|<k$. For any sparse $A$ with $|A|=k$, by definition there exists $D'\in \mathfs{D}(A)$. If $D'$ is packed, then we take $D=D'$. Otherwise, since $|D'|<|A|=k$, by the inductive hypotheses, there exists a packed set $D''\in \mathfs{D}(D')$. Therefore, we have $\bm{0}\notin D''$ and 
	\begin{equation}\label{ineq_3.6}
		\mathbf{d}\left( D'',D' \setminus D''\right) > \Cref{delicate_green}[\mathrm{diam}(D'')\vee 1]. 
	\end{equation}
	Moreover, since $D''\subset D'$ and $D'\in \mathfs{D}(A)$, one has 
	\begin{equation}\label{ineq_3.7}
		\mathbf{d}\left( D'',A \setminus D'\right)\ge 	\mathbf{d}\left( D',A \setminus D'\right) > \Cref{delicate_green}[\mathrm{diam}(D')\vee 1]\ge \Cref{delicate_green}[\mathrm{diam}(D'')\vee 1]. 
	\end{equation}
	Combining (\ref{ineq_3.6}) and (\ref{ineq_3.7}), we get 
	\begin{equation*}
		\begin{split}
			\mathbf{d}\left( D'',A \setminus D''\right)=& \min \{\mathbf{d}\left( D'',D' \setminus D''\right),\mathbf{d}\left( D'',A \setminus D'\right)  \} \\
			> &\Cref{delicate_green}[\mathrm{diam}(D'')\vee 1], 
		\end{split}
	\end{equation*}
	which implies that $D''\in \mathfs{D}(A)$. Thus, we only need to take $D=D''$. Now we complete the induction and conclude this lemma.
\end{proof}

The next lemma shows that the diameter of a packed set is at most a polynomial function of its cardinality. We postpone its proof to Section \ref{section_diameter}.

\begin{lemma}\label{lemma_packed3}
	Let $\Cl\label{const_packed}:= \log_2(\Cref{delicate_green}+2)$. For any packed $A\subset \mathfs{V}$, we have 
	\begin{enumerate}
		\item $\mathrm{diam}(A)\le (\Cref{delicate_green}^2+\Cref{delicate_green})|A|^{\Cref{const_packed}}$;

		\item In addition, if $\bm{0}\notin A$, then $\mathrm{diam}(A)\le \Cref{delicate_green}|A|^{\Cref{const_packed}}$.  
		
	\end{enumerate}

\end{lemma}

We define the removal mapping $\mathfrak{R}:\mathcal{A}(\mathscr{G}) \to \mathcal{A}(\mathscr{G})$ as follows. For $A\in\mathcal{A}(\mathscr{G})$,   

\begin{enumerate}
	\item[(a)]  If $A$ is packed, then let $\mathfrak{R}(A)=A$;

	\item[(b)]  Otherwise, take a packed set $D_{\dagger}\in \mathfs{D}(A)$ (in a predetermined manner) and let $\mathfrak{R}(A)=A\setminus D_{\dagger}$, where the existence of $D_{\dagger}$ is ensured by Lemma \ref{lemma_packed0}. 	
	
\end{enumerate}
For $i\in \mathbb{N}^+$, let $\mathfrak{R}^{(i)}$ be the $i$-th interation of $\mathfrak{R}$. For any $A\in \mathcal{A}(\mathscr{G})$, we define 
\begin{equation}\label{def_AD}
	A_{\mathcal{D}}:= \mathfrak{R}^{(|A|-1)}(A)
\end{equation}
as the set derived from $A$ after $|A|-1$ removals under $\mathfrak{R}$. Note that $A_{\mathcal{D}}$ is packed. In fact, for any $A'\subset A$, $\mathfrak{R}(A')$ either equals to $A'$ (i.e. $A'$ is packed), or satisfies $|\mathfrak{R}(A')| \le |A'|-1$ (i.e. at least one vertex is removed by $\mathfrak{R}$). Therefore, after $|A|-1$ removals under $\mathfrak{R}$, the set $A_{\mathcal{D}}$ already becomes packed.

Let $m_{\mathcal{D}}=m_{\mathcal{D}}(A)$ be the minimal integer in $[0,|A|-1]$ such that $\mathfrak{R}^{(m_{\mathcal{D}})}(A)=A_{\mathcal{D}}$. Note that $m_{\mathcal{D}}=0$ if $A$ is packed. When $m_{\mathcal{D}}\ge 1$, for each $1\le i\le m_{\mathcal{D}}$, since $D_i=D_i(A):=\mathfrak{R}^{(i-1)}(A)\setminus \mathfrak{R}^{(i)}(A)$ is packed and does not contain $\bm{0}$ (recall Case (b) in the definition of $\mathfrak{R}$), by Item (2) of Lemma \ref{lemma_packed3} we have
\begin{equation}\label{final_5.6}
	\mathrm{diam}(D_i)\le \Cref{delicate_green}|D_i|^{\Cref{const_packed}}. 
\end{equation}
Combined with Lemma \ref{lemma_remove_distant} (taking $\epsilon=0.1$), it implies that 
\begin{equation}\label{final_5.7_add}
		\frac{\mathbb{H}_{\mathfrak{R}^{(i)}(A)}(\bm{0})}{\mathbb{H}_{\mathfrak{R}^{(i-1)}(A)}(\bm{0})} \le  2.1\bigg[1+ \frac{\ln(\Cref{delicate_green})+ \Cref{const_packed}\ln(|D_i|)}{2\ln(\Cref{delicate_green})+ \Cref{const_packed}\ln(|D_i|)} \bigg]\le 2.1\bigg[\frac{3}{2}+ \frac{ \Cref{const_packed}\ln(|D_i|)}{2\ln(\Cref{delicate_green})}\bigg]. 
\end{equation}
Recall that $\Cref{const_packed}=\log_2(\Cref{delicate_green}+2)$ and $\Cref{delicate_green}\ge 10$. Therefore, since the function $f(a)=\frac{\ln(a+2)}{\ln(a)}$ for $a\in (1,\infty)$ is decreasing, one has $\frac{\Cref{const_packed}}{\ln(\Cref{delicate_green})} \le \frac{\ln(12)}{\ln(2)\ln(10)}<1.6$. Thus, it follows from (\ref{final_5.7_add}) that 
\begin{equation}\label{3.8}
\frac{\mathbb{H}_{\mathfrak{R}^{(i)}(A)}(\bm{0})}{\mathbb{H}_{\mathfrak{R}^{(i-1)}(A)}(\bm{0})} \le	4 \bigg[ 1+ \frac{1}{2}\ln(|D_i|) \bigg]\le 4 |D_i|^{\frac{1}{2}}, 
\end{equation}
where in the last inequality we used the fact that $1+\frac{1}{2}\ln(a)\le a^{\frac{1}{2}}$ for all $a\ge 1$. Note that $m_{\mathcal{D}}\le \sum_{i=1}^{m_{\mathcal{D}}}|D_i|=|A|-|A_{\mathcal{D}}|$. By the AM-GM inequality, we have 
\begin{equation}\label{3.9}
	\prod\nolimits_{1\le i\le m_{\mathcal{D}}} |D_i|\le \Big(\frac{|A|-|A_{\mathcal{D}}|}{m_{\mathcal{D}}}\Big)^{m_{\mathcal{D}}} \le e^{\frac{1}{e}(|A|-|A_{\mathcal{D}}|)}, 
\end{equation}
where in the second inequality we used the fact that $(\frac{b}{a})^a\le e^{\frac{b}{e}}$ for all $a,b>0$. By (\ref{3.8}), (\ref{3.9}) and $m_{\mathcal{D}} \le |A|-|A_{\mathcal{D}}|$, we get 
\begin{equation}\label{3.10}
	\begin{split}
		\frac{\mathbb{H}_{A_{\mathcal{D}}}(\bm{0})}{\mathbb{H}_{A}(\bm{0})} = \prod\nolimits_{1\le i\le m_{\mathcal{D}}} \frac{\mathbb{H}_{\mathfrak{R}^{(i)}(A)}(\bm{0})}{\mathbb{H}_{\mathfrak{R}^{(i-1)}(A)}(\bm{0})} \le (4e^{\frac{1}{2e}})^{|A|-|A_{\mathcal{D}}|}. 
	\end{split}
\end{equation}


\textbf{Step 2: remove all large interlocked clusters.} In this step, we aim to remove all large clusters from the set $A_{\mathcal{D}}$ (recall $A_{\mathcal{D}}$ in (\ref{def_AD})) to facilitate the subsequent estimation. Like in Step 1, this step also involves a quantitative removal argument. Moreover, in order to limit the number of times we use this removal argument (note that such a restriction is necessary because we will lose a constant factor on the upper bound whenever the removal argument is applied), we need to merge some specific clusters and consider them as a whole in the implementation of the removal argument. To achieve this, we first introduce the concept of interlocked clusters, which is defined as the union of several $*$-clusters whose exterior outer boundaries overlap in a specific way.

\begin{itemize}
	
	\item   (interlocked sets)  For any $D_1,D_2\subset \mathfs{V}$ that are not $*$-adjacent, we say $D_1$ and $D_2$ are interlocked (denoted by $D_1 \leftrightsquigarrow D_2$) if one of the following conditions holds (see Figures \ref{fig:interlock_Z2} and \ref{fig:interlock_TH} for illustrations):

	\begin{enumerate}
		\item[Condition (1):]  There exists $x\in (D_1\cup D_2)^c_{\infty}$ such that $|N(x)\setminus (D_1\cup D_2)|=2$ and $N(x)\cap D_i\neq \emptyset$ for $i\in \{1,2\}$. If this holds, we denote $x\in \mathcal{I}^{1}(D_1,D_2)$.

		\item[Condition (2):]   There exists $x_1,x_2\in (D_1\cup D_2)^c_{\infty}$ with $x_1\sim x_2$ such that for each $i\in \{1,2\}$, $|N(x_i)\setminus D_i|=2$ and $N(x_i)\cap D_{3-i}=\emptyset$. If this holds, we denote $\{x_1,x_2\} \in \mathcal{I}^{2}(D_1,D_2)$.

	\end{enumerate}

	\item  (interlocking vertex) For any $D_1,D_2\subset \mathfs{V}$ that are interlocked, we denote 
	\begin{equation*}
		\begin{split}
			\mathcal{I}(D_1,D_2):=&\big\{ x\in (D_1\cup D_2)^c_{\infty}: x\in \mathcal{I}^{1}(D_1,D_2)\big\}\\
			&\cup \big\{ x\in (D_1\cup D_2)^c_{\infty}:  \exists x'\in (D_1\cup D_2)^c_{\infty}\ \text{with}\ \{x,x'\}\in \mathcal{I}^{2}(D_1,D_2) \big\}.
		\end{split}
	\end{equation*}

	\item (interlocked cluster) For any $A\subset \mathfs{V}$ and two $*$-clusters $\mathcal{C}_*,\mathcal{C}_*'$ of $A$, we say $\mathcal{C}_*$ and $\mathcal{C}_*'$ are in the same interlocked class if $\mathcal{C}_*=\mathcal{C}_*'$ or there exists a sequence of $*$-clusters $\mathcal{C}_*^1,...,\mathcal{C}_*^k$ of $A$ such that $\mathcal{C}_*^1=\mathcal{C}_*$, $\mathcal{C}_*^k=\mathcal{C}_*'$ and $\mathcal{C}_*^{i} \leftrightsquigarrow \mathcal{C}_*^{i+1}$ for all $1\le i\le k-1$. An interlocked cluster of $A$ is the union of all $*$-clusters of $A$ in one interlocked class. We denote by $\mathfrak{I}_A$ the collection of all interlocked clusters of $A$. 
	
\end{itemize}

\begin{remark}[properties of the interlocking vertex]\label{remark_interlock} \hfill 

	(i) For any $D_1,D_2\subset \mathfs{V}$ with $D_1 \leftrightsquigarrow D_2$, the local configurations on $\mathcal{I}(D_1,D_2)$ are highly limited. As shown in Figure \ref{fig:interlock_Z2} (where the red and blue dots belong to $D_1$ and $D_2$ respectively, and endpoints of solid line segments lie in $\cap_{i=1}^{2}\partial^{\mathrm{o}}_{\infty,*}D_i$), for each case (either $x\in \mathcal{I}^1(D_1,D_2)$ or $\{x_1,x_2\}\in \mathcal{I}^2(D_1,D_2)$), there is a unique configuration on $\mathbb{Z}^2$ (allowing for rotation if needed).

	\begin{figure}[h!]

	\begin{tikzpicture}
		
		\begin{scope}[xshift=5cm]
				\draw[step=1cm,blue,thin, dotted] (-1,-1) grid (2,1);
			
			\node at (0,0) [circle,fill=black,inner sep=1.5pt]{};
			\node at (0,-1) [circle,fill=red,inner sep=2pt]{};
			\node at (-1,0) [circle,fill=red,inner sep=2pt]{};
			\node at (1,0) [circle,fill=black,inner sep=1.5pt]{};
			\node at (2,0) [circle,fill=blue,inner sep=2pt]{};
			\node at (1,1) [circle,fill=blue,inner sep=2pt]{};

			\draw [black,very thick] (0,0) -- (0,1);
			\draw [black,very thick] (0,0) -- (1,0);
			\draw [black,very thick] (1,0) -- (1,-1);
			
			\node at (0.3,-0.3) []{$x_1$};
			\node at (1.3,-0.3) []{$x_2$};
			\node at (-1,-1) []{\color{red}$D_1$};
			\node at (2,1) []{\color{blue}$D_2$};

			\node at (0.5,-1.7) {\small{$\{x_1,x_2\}\in \mathcal{I}^2(D_1,D_2)$ on $\mathbb{Z}^2$}};

		\end{scope}
	
	\begin{scope}

	\draw[step=1cm,blue,thin, dotted] (-1,-1) grid (1,1);
	
	\node at (0,0) [circle,fill=black,inner sep=1.5pt]{};
	\node at (-1,0) [circle,fill=red,inner sep=2pt]{};
	\node at (1,0) [circle,fill=blue,inner sep=2pt]{};
	
	\draw [black,very thick] (0,0) -- (0,1);
	\draw [black,very thick] (0,0) -- (0,-1);
	
	\node at (0.3,-0.3) []{$x$};
		\node at (-1.3,-0.3) []{\color{red}$D_1$};
	\node at (1.3,-0.3) []{\color{blue}$D_2$};

		\node at (0,-1.7) {\small{$x\in \mathcal{I}^1(D_1,D_2)$ on $\mathbb{Z}^2$}};
\end{scope}

	\end{tikzpicture}
	
	\caption{Illustrations for $D_1\leftrightsquigarrow D_2$ on $\mathbb{Z}^2$ \label{fig:interlock_Z2}}
\end{figure}
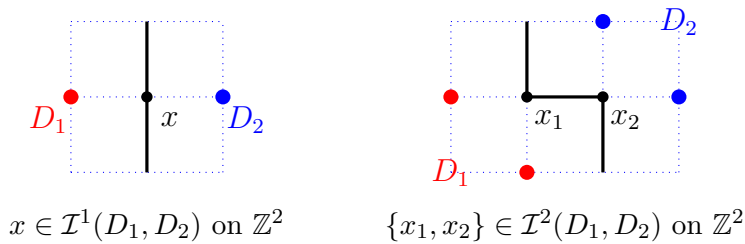

	For $\mathscr{T}$, only Condition (1) can hold. In fact, for any adjacent vertices $x_1\sim x_2$ on $\mathscr{T}$, the structure of $\mathscr{T}$ implies that there are exactly two vertices, denoted by $y_1$ and $y_2$, that are adjacent to both $x_1$ and $x_2$. Therefore, if $\{x_1,x_2\}\in \mathcal{I}^2(D_1,D_2)$ on $\mathscr{T}$, since $|N(x_1)\setminus D_1|=2$ and $x_2\in  N(x_1)\setminus D_1$, we know that at least one vertex of $\{y_1,y_2\}$, say $y_j$ (where $j\in \{1,2\}$), is contained in $D_1$. However, $y_j\sim x_2$ and $y_j\in D_1$ together cause a contradiction to the condition that $N(x_2)\cap D_1=\emptyset$. Thus, Condition (2) cannot hold on $\mathscr{T}$.

	For $\mathscr{H}$, only Condition (2) can be valid. To see this, if $x\in \mathcal{I}^1(D_1,D_2)$ on $\mathscr{H}$, \
	\begin{equation}
		N(x)=|N(x)\setminus (D_1\cup D_2)|+\sum\nolimits_{i\in \{1,2\}}|N(x)\cap D_i|\ge 2+1+1=4,
	\end{equation}
	which is incompatible to the fact that $\mathrm{deg}(\mathscr{H})=3$. All possible configurations for $x\in \mathcal{I}^1(D_1,D_2)$ on $\mathscr{T}$ and $\{x_1,x_2\}\in \mathcal{I}^2(D_1,D_2)$ on $\mathscr{H}$ are given in Figure \ref{fig:interlock_TH}.

	\begin{figure}[h!]

	\begin{tikzpicture}
		\newcommand*\sqthree{1.73205}
		
		\begin{scope}[xshift=-5cm]
			\node at (0,0) [circle,fill=black,inner sep=1.5pt]{};
			\node at (0,0) [circle,fill=black,inner sep=1.5pt]{};
			\node at (-1,0) [circle,fill=red,inner sep=2pt]{};
			\node at (-0.5,-0.5*\sqthree) [circle,fill=red,inner sep=2pt]{};
			\node at (1,0) [circle,fill=blue,inner sep=2pt]{};
			\node at (0.5,0.5*\sqthree) [circle,fill=blue,inner sep=2pt]{};

			\draw[blue,thin, dotted] (0,0) -- (1,0);
			\draw[blue,thin, dotted] (0,0) -- (-1,0);
			\draw[blue,thin, dotted] (0,0) -- (0.5,0.5*\sqthree);
			\draw[blue,thin, dotted] (0,0) -- (-0.5,0.5*\sqthree);
			\draw[blue,thin, dotted] (0,0) -- (-0.5,-0.5*\sqthree);
			\draw[blue,thin, dotted] (0,0) -- (0.5,-0.5*\sqthree);
			\draw[blue,thin, dotted] (1,0) -- (0.5,-0.5*\sqthree);
			\draw[blue,thin, dotted] (1,0) -- (0.5,0.5*\sqthree);
			\draw[blue,thin, dotted] (-1,0) -- (-0.5,-0.5*\sqthree);
			\draw[blue,thin, dotted] (-1,0) -- (-0.5,0.5*\sqthree);
			\draw[blue,thin, dotted] (-0.5,0.5*\sqthree) -- (0.5,0.5*\sqthree);
			\draw[blue,thin, dotted] (-0.5,-0.5*\sqthree) -- (0.5,-0.5*\sqthree);

			\draw [black,very thick] (0,0) -- (-0.5,0.5*\sqthree);	
			\draw [black,very thick] (0,0) -- (0.5,-0.5*\sqthree);

			\node at (-0.25,-0.25) []{$x$};
			\node at (-1,-0.8) []{\color{red}$D_1$};
			\node at (1,0.8) []{\color{blue}$D_2$};

			\node at (2.5, -2.3+0.25*\sqthree) {\small{two scenarios for $x\in \mathcal{I}^1(D_1,D_2)$ on $\mathscr{T}$}};
			
		\end{scope}

			\begin{scope}
				\node at (0,0) [circle,fill=black,inner sep=1.5pt]{};
				\node at (-1,0) [circle,fill=red,inner sep=2pt]{};
				\node at (-0.5,-0.5*\sqthree) [circle,fill=red,inner sep=2pt]{};
				\node at (0.5,-0.5*\sqthree) [circle,fill=red,inner sep=2pt]{};
				\node at (0.5,0.5*\sqthree) [circle,fill=blue,inner sep=2pt]{};

			\draw[blue,thin, dotted] (0,0) -- (1,0);
			\draw[blue,thin, dotted] (0,0) -- (-1,0);
			\draw[blue,thin, dotted] (0,0) -- (0.5,0.5*\sqthree);
			\draw[blue,thin, dotted] (0,0) -- (-0.5,0.5*\sqthree);
			\draw[blue,thin, dotted] (0,0) -- (-0.5,-0.5*\sqthree);
			\draw[blue,thin, dotted] (0,0) -- (0.5,-0.5*\sqthree);
			\draw[blue,thin, dotted] (1,0) -- (0.5,-0.5*\sqthree);
			\draw[blue,thin, dotted] (1,0) -- (0.5,0.5*\sqthree);
			\draw[blue,thin, dotted] (-1,0) -- (-0.5,-0.5*\sqthree);
			\draw[blue,thin, dotted] (-1,0) -- (-0.5,0.5*\sqthree);
			\draw[blue,thin, dotted] (-0.5,0.5*\sqthree) -- (0.5,0.5*\sqthree);
			\draw[blue,thin, dotted] (-0.5,-0.5*\sqthree) -- (0.5,-0.5*\sqthree);

			\draw [black,very thick] (0,0) -- (-0.5,0.5*\sqthree);	
			\draw [black,very thick] (0,0) -- (1,0);

			\node at (-0.25,-0.25) []{$x$};
			\node at (-1,-0.8) []{\color{red}$D_1$};
			\node at (1,0.8) []{\color{blue}$D_2$};

		\end{scope}

			\begin{scope}[xshift=5cm,yshift=0.25*\sqthree cm]
			\node at (0.5,-0.5*\sqthree) [circle,fill=black,inner sep=1.5pt]{};
			\node at (1,0) [circle,fill=black,inner sep=1.5pt]{};
			\node at (-0.5,-0.5*\sqthree) [circle,fill=red,inner sep=2pt]{};
			\node at (2,0) [circle,fill=blue,inner sep=2pt]{};
			
			\draw[blue,thin, dotted] (1,0) -- (0.5,-0.5*\sqthree);
			\draw[blue,thin, dotted] (1,0) -- (0.5,0.5*\sqthree);
			\draw[blue,thin, dotted] (-1,0) -- (-0.5,-0.5*\sqthree);
			\draw[blue,thin, dotted] (-1,0) -- (-0.5,0.5*\sqthree);
			\draw[blue,thin, dotted] (-0.5,0.5*\sqthree) -- (0.5,0.5*\sqthree);
			\draw[blue,thin, dotted] (-0.5,-0.5*\sqthree) -- (0.5,-0.5*\sqthree);
			\draw[blue,thin, dotted] (1,0) -- (2,0);
			\draw[blue,thin, dotted] (2,0) -- (2.5,-0.5*\sqthree);
			\draw[blue,thin, dotted] (2,-\sqthree) -- (2.5,-0.5*\sqthree);
			\draw[blue,thin, dotted] (2,-\sqthree) -- (1,-\sqthree);
			\draw[blue,thin, dotted] (0.5,-0.5*\sqthree) -- (1,-\sqthree);

			\draw [black,very thick] (1,0) -- (0.5,0.5*\sqthree);	
			\draw [black,very thick] (1,0) -- (0.5,-0.5*\sqthree);	
			\draw [black,very thick] (1,-\sqthree) -- (0.5,-0.5*\sqthree);

			\node at (0.6,0) []{$x_2$};
			\node at (0.9,-0.5*\sqthree) []{$x_1$};
			\node at (-0.5,-1.2) []{\color{red}$D_1$};
			\node at (2,0.3) []{\color{blue}$D_2$};

			\node at (1,-2.3) {\small{$\{x_1,x_2\}\in \mathcal{I}^2(D_1,D_2)$ on $\mathscr{H}$}};
		\end{scope}

	\end{tikzpicture}
	
	\caption{Illustrations for $D_1\leftrightsquigarrow D_2$ on $\mathscr{T}$ and $\mathscr{H}$ \label{fig:interlock_TH}}
\end{figure}
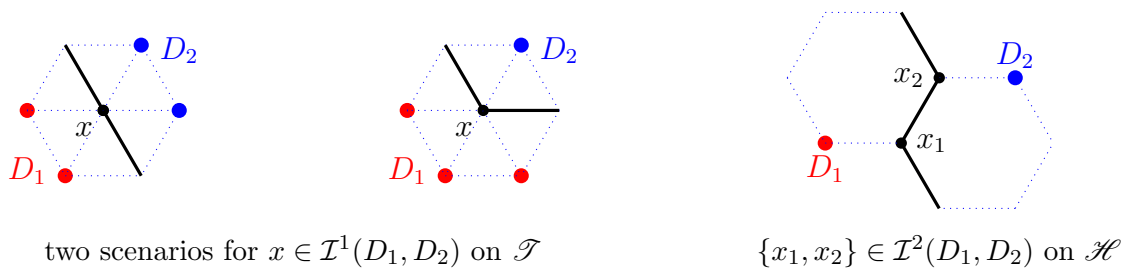

	(ii) As shown in Figures \ref{fig:interlock_Z2} and \ref{fig:interlock_TH}, when $D_1\leftrightsquigarrow D_2$ holds, there are at least $\kappa(\mathscr{G})$ edges (i.e. solid line segments in Figures \ref{fig:interlock_Z2} and \ref{fig:interlock_TH}) that have both endpoints in $\cap_{i=1}^{2}\partial^{\mathrm{o}}_{\infty,*}D_i$ and intersect $\mathcal{I}(D_1,D_2)$, where $\kappa(\mathbb{Z}^2)=\kappa(\mathscr{T}):=2$ and $\kappa(\mathscr{H}):=3$.

	(iii) Assume that $A$ is an interlocked cluster, and $D_1,D_2$ are two $*$-clusters of $A$ with $D_1\leftrightsquigarrow D_2$. For each $x\in \mathcal{I}(D_1,D_2)$, since any vertex in $N(x)\setminus (D_1\cup D_2)$ is $*$-adjacent to both $D_1$ and $D_2$, we have $N(x)\cap A= N(x)\cap (D_1\cup D_2)$ (otherwise, there exists $y\in [N(x)\cap A]\setminus (D_1\cup D_2)$ that is $*$-adjacent to both $D_1$ and $D_2$, which implies that $D_1$ and $D_2$ are included in the same $*$-cluster of $A$ and thus causes a contradiction). As a result, for any $*$-clusters $\{\mathcal{C}_j\}_{j=1}^{4}$ of $A$, 
	\begin{equation*}
		\mathcal{I}(\mathcal{C}_1,\mathcal{C}_2)\cap \mathcal{I}(\mathcal{C}_3,\mathcal{C}_4) \neq \emptyset \iff \{\mathcal{C}_1,\mathcal{C}_2\}=\{\mathcal{C}_3,\mathcal{C}_4\}.
	\end{equation*}

\end{remark}

The subsequent lemma presents a crucial estimate on the increase of the harmonic measure after removing a sufficiently large interlocked cluster. Recall the notation $\lambda(\mathscr{G})$ in Theorem \ref{theorem1.1}.

\begin{lemma}\label{lemma_remove_interlock}
	There exist $\Cl\label{remove_interlock1}(\mathscr{G}),\cl\label{remove_interlock2}(\mathscr{G})>0$ such that for any $A\in \mathcal{A}(\mathscr{G})$, $D\in \mathfrak{I}_A$ with $\bm{0}\notin D$ and $|D|\ge \Cref{remove_interlock1}$,
	\begin{equation}\label{ineq_remove_interlock}
		\frac{\mathbb{H}_{A\setminus D}(\bm{0})}{\mathbb{H}_{A}(\bm{0})} \le  \big[\lambda(\mathscr{G})\big]^{|D|-\cref{remove_interlock2}\sqrt{|D|}}. 
	\end{equation}
\end{lemma}

The main component of the proof of Lemma \ref{lemma_remove_interlock} is to compute the probability of bypassing an interlocked cluster for a random walk, which highly relies on estimating the length of the circuit surrounding a given $*$-cluster. These estimates will be detailed in Section \ref{section_length_bypass}. By taking the main result of Section \ref{section_length_bypass} (i.e.\ Proposition \ref{lemma_length_bypass}) as an input and employing the framework of the removal argument in \cite{psi}, we establish Lemma \ref{lemma_remove_interlock} in Section \ref{section_qra_interlocked}. Notably, the concept of interlocked clusters plays a crucial role in deriving the second leading term $-\cref{remove_interlock2}\sqrt{|D|}$ of the exponent on the right-hand side of (\ref{ineq_remove_interlock}).

For any $A\in \mathcal{A}(\mathscr{G})$, we recall $A_{\mathcal{D}}$ in (\ref{def_AD}), and enumerate all interlocked clusters $D\in \mathfrak{I}_{A_\mathcal{D}}$ with $\bm{0}\in D$ and $|D|\ge \Cref{remove_interlock1}(\mathscr{G})$ as $D_i'$ for $1\le i\le k$. Then we define 
\begin{equation}\label{final_5.13}
	A_\mathcal{E} := A_{\mathcal{D}} \setminus \cup_{i=1}^{k} D_i'.
\end{equation}
Note that for each $1\le i\le k$, $D_i'$ is an interlocked cluster of $A_{\mathcal{D}}\setminus \cup_{j=1}^{i-1} D_i'$. Thus, by applying Lemma \ref{lemma_remove_interlock} repeatedly, we have 
\begin{equation}\label{remove_interlock}
	\begin{split}
		\mathbb{H}_{A_{\mathcal{D}}}(\bm{0})\ge & \mathbb{H}_{A_\mathcal{E}}(\bm{0})\cdot [\lambda(\mathscr{G})]^{\sum_{i=1}^k -|D_i'|+\cref{remove_interlock2}\sqrt{|D_i'|}}\\
		\ge &\mathbb{H}_{A_\mathcal{E}}(\bm{0})\cdot [\lambda(\mathscr{G})]^{-(|A_{\mathcal{D}}|-|A_\mathcal{E}|)+\cref{remove_interlock2}\sqrt{|A_{\mathcal{D}}|-|A_\mathcal{E}|}}, 
	\end{split}
\end{equation}
where in the second line we used $\sum_{i=1}^{k}|D_i'|=|A_{\mathcal{D}}|-|A_\mathcal{E}|$ and the inequality that 
\begin{equation}\label{ineq_square_root}
	\sum\nolimits_{1\le i\le k}x_i^{\frac{1}{2}}\ge \Big( \sum\nolimits_{1\le i\le k}x_i\Big)^{\frac{1}{2}},\ \ \forall x_1,...,x_k\ge 0.
\end{equation}

Let $\cl\label{const_4.2}(\mathscr{G}):=1- \frac{\mathrm{ln}\big(4e^{\frac{1}{2e}}\big)}{\ln(\lambda(\mathscr{G}))}$. Note that $\cref{const_4.2}\in (0,1)$ since $\lambda(\mathscr{G})\ge 3+2\sqrt{2}>4 e^{\frac{1}{2e}}$. Thus, by (\ref{3.10}) and (\ref{remove_interlock}), we have 
\begin{equation}\label{ineq_3.11}
	\mathbb{H}_{A}(\bm{0})\ge \mathbb{H}_{A_\mathcal{E}}(\bm{0})\cdot [\lambda(\mathscr{G})]^{-(|A|-|A_\mathcal{E}|)+\cref{remove_interlock2}\sqrt{|A_\mathcal{D}|-|A_\mathcal{E}|}+\cref{const_4.2}(|A|-|A_\mathcal{D}|) }. 
\end{equation}

\textbf{Step 3: remove the $*$-clusters in a specific angle.} The goal of this step is to remove all $*$-clusters of $A_{\mathcal{E}}$ that intersect a chosen sector. For any $A\in \mathcal{A}_n(\mathscr{G})$, we denote by $A_{(\bm{0})}$ the interlocked cluster of $A_{\mathcal{D}}$ that contains $\bm{0}$. Note that $A_{(\bm{0})}\subset A_{\mathcal{E}}$. Let $\mathfrak{C}_A$ be the collection of $*$-clusters of $A_{\mathcal{E}}\setminus A_{(\bm{0})}$. In fact, we have $|\mathcal{C}_*|< \Cref{remove_interlock1}$ for all $\mathcal{C}_*\in \mathfrak{C}_A$; otherwise (i.e. $|\mathcal{C}_*|\ge  \Cref{remove_interlock1}$), the interlocked cluster containing $\mathcal{C}_*$ must have been removed in Step 2. Thus, it follows from (\ref{final_3.5_boundary}) that 
\begin{equation}\label{final_5.17_bypass}
	\mathbf{d}^{\partial_{\infty,*}^{\mathrm{o}} \mathcal{C}_*}(x,y) \le 12|\mathcal{C}_*|<12\Cref{remove_interlock1}, \ \ \forall x,y\in  \partial_{\infty,*}^{\mathrm{o}}\mathcal{C}_*. 
\end{equation}

Recalling that $A_{\mathcal{D}}$ is packed, by Item (1) of Lemma \ref{lemma_packed3} and (\ref{ineq_ball}) we have 
\begin{equation}\label{new_revision_5.13}
	A_{\mathcal{E}} \subset A_{\mathcal{D}}  \subset  B(\Cref{const_A_D}n^{\Cref{const_packed}}),
\end{equation}
where $ \Cl\label{const_A_D}:= \Cref{ball_1}(\Cref{delicate_green}^2+\Cref{delicate_green})$. Take a constant $\cl\label{ball_inner_length}(\mathscr{G})>0$ such that 
\begin{equation}\label{ineq_3.8}
	\mathbf{d}(\bm{0},z)\le \tfrac{\cref{remove_interlock2}\land \cref{prob_surrounding}}{100\Cref{remove_interlock1}}\cdot \sqrt{n}, \ \ \forall z\in \partial^{\mathrm{i}} B(2\cref{ball_inner_length}\sqrt{n}),
\end{equation}
where $\cref{prob_surrounding}(\mathscr{G})>0$ will be determined later in Lemma \ref{lemma_prob_surrounding}.

We decompose the annulus $\Theta_n:=B(\Cref{const_A_D}n^{\Cref{const_packed}})\setminus B(\cref{ball_inner_length}\sqrt{n})$ as follows: for any $\delta>0$ and $0 \le  l\le 2\lfloor \delta\sqrt{n}\rfloor-1$, we define 
\begin{equation*}
	\Theta_{n,\delta}^{l}:= \Big\{x \in \Theta_n:-\pi+ \tfrac{l\pi}{\lfloor \delta\sqrt{n}\rfloor}< \mathrm{Arg}(x) \le -\pi+ \tfrac{(l+1)\pi}{\lfloor \delta\sqrt{n}\rfloor} \Big\}, 
\end{equation*}
where $\mathrm{Arg}(x)$ is the main argument of $x$. Let $ \cl\label{ball_section}(\mathscr{G})>0$ be a sufficiently small constant such that for any $0\le k_1<k_2\le \lfloor \cref{ball_section}\sqrt{n}\rfloor-1$, 
\begin{equation}\label{distance_theta}
	\mathbf{d}(\Theta_{n,\cref{ball_section}}^{2k_1},\Theta_{n,\cref{ball_section}}^{2k_2})> 3\Cref{remove_interlock1}. 
\end{equation}
We abbreviate $\Theta_{n,\cref{ball_section}}^{l}$ as $\Theta_n^{l}$. For every $\mathcal{C}_*\in \mathfrak{C}_A$, since $|\mathcal{C}_*|\le \Cref{remove_interlock1}$, it follows from (\ref{ineq_2.4}) that $\mathrm{diam}(\mathcal{C}_*)\le 3\Cref{remove_interlock1}$. Thus, by (\ref{distance_theta}), $\mathcal{C}_*$ can intersect at most one of $\Theta_n^{2k}$ for $0\le k\le \lfloor \cref{ball_section}\sqrt{n}\rfloor-1$ (this is why we need Step 2 to remove all large clusters). Consequently, by the pigeonhole principle, there exists $k_\dagger\in [0,\lfloor \cref{ball_section}\sqrt{n}\rfloor -1]$ such that (letting $\mathfrak{C}_A^{\dagger}:= \{\mathcal{C}_*\in \mathfrak{C}_A:\mathcal{C}_*\cap \Theta_n^{2k_\dagger}\neq \emptyset\}$)
\begin{equation}\label{final_5.20}
	\sum\nolimits_{\mathcal{C}_*\in \mathfrak{C}_A^{\dagger}} |\mathcal{C}_*| \le (\lfloor \cref{ball_section}\sqrt{n}\rfloor)^{-1}(|A_{\mathcal{E}}|-|A_{(\bm{0})}|).
\end{equation}

In \cite[Section 3.3]{psi}, it was proved that for any $A\in \mathcal{A}(\mathbb{Z}^2)$, one can always remove a selected vertex in any $*$-cluster $\mathcal{C}_*$ of $A$ with $\bm{0}\notin \mathcal{C}_*$ while increasing the harmonic measure at $\bm{0}$ by a uniformly bounded factor. Moreover, as stated in \cite[Remark 3.8]{psi}, this result can be generalized to $\mathscr{T}$ and $\mathscr{H}$. As a consequence, by removing all vertices of $*$-clusters in $\mathfrak{C}_A^{\dagger}$ one by one and applying (\ref{final_5.20}), we obtain that there exists $\Cl\label{remove_one_vertex}(\mathscr{G})>0$ such that (letting $A_{\mathcal{E}}^{\dagger}:=A_{\mathcal{E}}\setminus \cup_{\mathcal{C}_*\in \mathfrak{C}_A^{\dagger}}\mathcal{C}_*$)
\begin{equation}\label{ineq_AD-ADD}
	\mathbb{H}_{A_{\mathcal{E}}}(\bm{0})\ge \mathbb{H}_{A_{\mathcal{E}}^{\dagger}}(\bm{0})\cdot e^{-\Cref{remove_one_vertex}\cdot \frac{|A_{\mathcal{E}}|-|A_{(\bm{0})}|}{\lfloor \cref{ball_section}\sqrt{n}\rfloor}}. 
\end{equation}

\textbf{Step 4: Conclusion.} For a random walk with a faraway starting point, there is a possible passage to first hit $A_{\mathcal{E}}^{\dagger}$ at $\bm{0}$, which consists of the following three parts:

\begin{enumerate}
	\item  The random walk first hits $B(2\Cref{const_A_D}n^{\Cref{const_packed}})$ at some $w_1\in \Phi_n^{2k_{\dagger}}$, where 
	$$\Phi_n^{l}:= \Big\{x \in \partial^{\mathrm{i}}B(2\Cref{const_A_D}n^{\Cref{const_packed}}):-\pi+ \tfrac{(l+\frac{1}{3})\pi}{\lfloor \cref{ball_section}\sqrt{n}\rfloor}< \mathrm{Arg}(x) \le -\pi+ \tfrac{(l+\frac{2}{3})\pi}{\lfloor \cref{ball_section}\sqrt{n}\rfloor} \Big\}$$
	for $0\le l\le 2\lfloor \cref{ball_section}\sqrt{n}\rfloor-1$. Note that $A_{\mathcal{E}}^{\dagger}\subset A_{\mathcal{E}}\subset B(\Cref{const_A_D}n^{\Cref{const_packed}})$ by (\ref{new_revision_5.13}).

	\item The random walk starts from $w_1$ and hits some $w_2\in B(2\cref{ball_inner_length}\sqrt{n})\cup \partial_{\infty,*}^{\mathrm{o}} A_{(\bm{0})}$ before crossing the ray $\zeta^{2k_{\dagger}}$ or $\zeta^{2k_{\dagger}+1}$, where $\zeta^0:= \{x=(-a,0):a>0\}$, 
	$$
	\zeta^{l}:= \Big\{x\in \mathbb{R}^2: \mathrm{Arg}(x)=-\pi+ \tfrac{l\pi}{\lfloor \cref{ball_section}\sqrt{n}\rfloor}  \Big\},\ \ \forall 1\le l\le 2\lfloor \cref{ball_section}\sqrt{n}\rfloor.
	$$
	Note that every $*$-cluster of $A_{\mathcal{E}}^{\dagger}$ intersecting $\Theta_n^{2k_\dagger}$ must be contained in $A_{(\bm{0})}$.

	\item  According to the position of $w_2$, we have the following two subcases (see Figure \ref{fig:proof_thm1} for an illustration, where the pink and red regions represent the $*$-clusters in $A_{(\bm{0})}$ and $A_{\mathcal{E}}^{\dagger}\setminus A_{(\bm{0})}$ respectively). 
	\begin{itemize}

		\item[(a)] When $w_2\in [\partial^{\mathrm{i}} B(2\cref{ball_inner_length}\sqrt{n})]\setminus [\partial_{\infty,*}^{\mathrm{o}} A_{(\bm{0})}]$ (which is contained in $A_\infty^c$ since $[A_{\mathcal{E}}^{\dagger}\setminus A_{(\bm{0})}]\cap \Theta_n^{2k_\dagger}=\emptyset$), let the random walk start from $w_2$, move along a specific path $\eta_{\dagger}$ from $w_2$ to $\partial_{\infty,*}^{\mathrm{o}} A_{(\bm{0})}$ with $\mathbf{R}^{\mathrm{v}}(\eta_{\dagger})\subset A_{\infty}^c$, and then first hit $A_{\mathcal{E}}^{\dagger}$ at $\bm{0}$. The path $\eta_{\dagger}$ is constructed as follows. By (\ref{ineq_3.8}), there exists a path $\eta_{\diamond}$ from $w_2$ to $\bm{0}$ such that $\mathbf{L}(\eta_{\diamond})\le \tfrac{\cref{remove_interlock2}\land \cref{prob_surrounding}}{100\Cref{remove_interlock1}}\cdot \sqrt{n}$ (e.g.\ the geodesic from $w_2$ to $\bm{0}$). Let $t_{\diamond}$ be the first time when $\eta_{\diamond}$ hits $\partial_{\infty,*}^{\mathrm{o}} A_{(\bm{0})}$. We define a sequence of hitting times and exit times recursively as follows. We first set $t_0^+:=0$. For each $j\ge 1$, suppose that we already construct $0=t_0^+\le ... \le t_{j-1}^{+}$. Then consider the hitting time
		\begin{equation}
			t_{j}^{-}:= \min\big\{\text{integer}\ t\in (t_{j-1}^{+},t_{\diamond}):\eta_{\diamond}(t)\in A_{\mathcal{E}}^{\dagger}\setminus A_{(\bm{0})} \big\},
		\end{equation}
		where we set $\min \emptyset :=\infty $ for completeness. If $t_{j}^-=\infty$, we stop the construction, and define $j_{\diamond}:=j$ and $t_{j_{\diamond}}^{-}:=t_{\diamond}$. Otherwise, let $\mathcal{C}_*^{j}$ be the $*$-cluster of $A_{\mathcal{E}}^{\dagger}\setminus A_{(\bm{0})}$ such that $\eta_{\diamond}(t_{j}^{-})\in \mathcal{C}_*^{j}$, and define the exit time
		\begin{equation}
			t_{j}^{+}:= \min \big\{\text{integer}\ t\in [t_{j}^{-},t_{\diamond}]: \mathbf{R}^{\mathrm{v}}(\eta_{\diamond}[t+1,t_{\diamond}]) \cap \mathcal{C}_*^{j} = \emptyset \big\}.
		\end{equation}
	 Since $\eta_{\diamond}(t_0^+)$($=w_2\in A_{\infty}^c$) and $\eta_{\diamond}(t^-_{j_{\diamond}})$($=\eta_{\diamond}(t_\diamond) \in  A_{\infty}^c$) are both contained in $\cap_{j=1}^{j_{\diamond}-1}(\mathcal{C}_*^{j})_{\infty}^c$, it follows from the construction that 
	 \begin{equation*}
	 	  \eta_{\diamond}( t_{j}^{-}-1),\eta_{\diamond}( t_{j}^{+}+1) \in \partial^{\mathrm{o}}_{\infty}\mathcal{C}_*^{j}, \ \ \forall 1\le j\le j_{\diamond}-1.
	 	  	 \end{equation*}
		Thus, by (\ref{final_5.17_bypass}), for each $1\le j\le j_{\diamond}-1$ there exits a path $\eta^j_{\dagger}$ from $\eta_{\diamond}( t_{j}^{-}-1)$ to $\eta_{\diamond}( t_{j}^{+}+1)$ such that $\mathbf{R}^{\mathrm{v}}(\eta^j_{\dagger})\subset \partial^{\mathrm{o}}_{\infty,*}\mathcal{C}_*^{j}$ and $\mathbf{L}(\eta^j_{\dagger})\le 12\Cref{remove_interlock1}$. The path $\eta_\dagger$ is defined as 
		\begin{equation}\label{final_5.24_dagger}
			\eta_{\dagger}:=\eta_{\diamond}[0,t_{1}^{-}-1]\circ \eta^1_{\dagger} \circ \eta_{\diamond}[t_{1}^{+}+1,t_{2}^{-}-1]\circ ... \circ  \eta^{j_{\diamond}-1}_{\dagger} \circ \eta_{\diamond}[t_{j_{\diamond}-1}^{+}+1,t_{\diamond}].
		\end{equation}

		\item[(b)] When $w_2\in \partial_{\infty,*}^{\mathrm{o}} A_{(\bm{0})}$, let the random walk start from $w_2$ and then first hits $A_{\mathcal{E}}^{\dagger}$ at $\bm{0}$. 	For convenience, we set $\eta_{\dagger}=(w_2)$ in this case.

	\end{itemize}

\end{enumerate} 
Consequently, by the strong Markov property, we get 
\begin{equation}\label{new_311}
	\mathbb{H}_{A_{\mathcal{E}}^{\dagger}}(\bm{0}) \ge \prod\nolimits_{j\in \{1,2,3,4\}}\mathbb{I}_j,
\end{equation}
where (letting $\tau^{(k_{\dagger})}$ be the first time when the random walk crosses $\zeta^{2k_{\dagger}}$ or $\zeta^{2k_{\dagger}+1}$)
\begin{equation*}
	\begin{split}
		&\mathbb{I}_1:= \sum\nolimits_{w \in \Phi_n^{2k_{\dagger}}}  \mathbb{H}_{B(2\Cref{const_A_D}n^{\Cref{const_packed}})}(w), \\
		&\mathbb{I}_2:=\min \Big\{  \mathbb{P}_{w_1}\big( \tau_{B(2\cref{ball_inner_length}\sqrt{n})\cup \partial_{\infty,*}^{\mathrm{o}} A_{(\bm{0})}} < \tau^{(k_{\dagger})} \big) :w_1 \in \Phi_n^{2k_{\dagger}} \Big\}, \\
		&\mathbb{I}_3:=\min \Big\{ \mathbb{P}_{w_2} \big( \tau_{\eta_{\dagger}(-1)}<\tau_{\partial^{\mathrm{o}}\mathbf{R}^{\mathrm{v}}(\eta_{\dagger})}\big)  : w_2\in  \big[\partial^{\mathrm{i}}B(2\cref{ball_inner_length}\sqrt{n})\big]\cup \big[\partial_{\infty,*}^{\mathrm{o}} A_{(\bm{0})}\big]  \Big\},\\
		&	\mathbb{I}_4:= \min \Big\{ \mathbb{P}_{w_3} \big(\tau_{A_{\mathcal{E}}^{\dagger}} =  \tau_{\bm{0}}\big): w_3\in  \partial_{\infty,*}^{\mathrm{o}} A_{(\bm{0})} \Big\}.
	\end{split}
\end{equation*}

\begin{figure}[h!]
 \centering
  \includegraphics[width=0.9\textwidth]{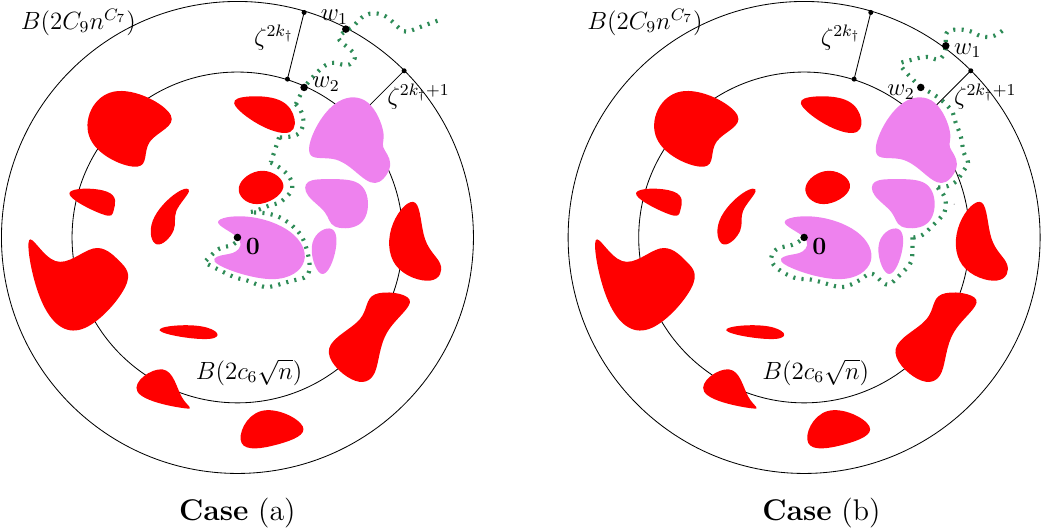}
 \caption{Illustrations for Part (3) of the passage  \label{fig:proof_thm1}}
 \end{figure}

For $\mathbb{I}_1$, by Lemma \ref{lemma_uniform} and the fact that $\frac{|\Phi_n^{2k_{\dagger}}|}{|\partial^{\mathrm{i}}B(2\Cref{const_A_D}n^{\Cref{const_packed}})|}\asymp n^{-\frac{1}{2}}$, we have $\mathbb{I}_1 \asymp n^{-\frac{1}{2}}$.

For $\mathbb{I}_2$, we denote $R_j=2^j \cref{ball_inner_length}\sqrt{n}$ and let $j_{\diamond}$ be the unique integer such that $R_{j_\diamond}\in \left(\tfrac{3}{4}\Cref{const_A_D}n^{\Cref{const_packed}},\tfrac{3}{2}\Cref{const_A_D}n^{\Cref{const_packed}}  \right]$. Note that $j_\diamond\asymp \ln(n)$. For $1 \le j\le j_{\diamond}$, let
$$\Theta_n^{2k_{\dagger},j}:= 	 \Big\{x \in B(R_{j+2})\setminus B(R_{j}):-\pi+ \tfrac{2k_{\dagger}\pi}{\lfloor \cref{ball_section}\sqrt{n}\rfloor}< \mathrm{Arg}(x) \le -\pi+ \tfrac{(2k_{\dagger}+1)\pi}{\lfloor \cref{ball_section}\sqrt{n}\rfloor} \Big\},$$
$$\Phi_n^{2k_{\dagger},j}:= \Big\{x \in \partial^{\mathrm{i}}B(R_j):-\pi+ \tfrac{(2k_{\dagger}+\frac{1}{3})\pi}{\lfloor \cref{ball_section}\sqrt{n}\rfloor}< \mathrm{Arg}(x) \le -\pi+ \tfrac{(2k_{\dagger}+\frac{2}{3})\pi}{\lfloor \cref{ball_section}\sqrt{n}\rfloor} \Big\}.$$
By the strong Markov property and the invariance principle, 
\begin{equation}\label{ineq_I2}
	\begin{split}
		\mathbb{I}_2 \ge & \min \Big\{\mathbb{P}_{w}\Big(  \tau_{[\Theta_n^{2k_{\dagger},j_{\diamond}}]^c}=\tau_{\Phi_n^{2k_{\dagger},j_{\diamond}}}\Big) : w\in \Phi_n^{2k_{\dagger}} \Big\}  \\
		&\cdot \prod\nolimits_{1\le j\le j_\diamond-1} \min \Big\{\mathbb{P}_{w}\Big(  \tau_{[\Theta_n^{2k_{\dagger},j}]^c}=\tau_{\Phi_n^{2k_{\dagger},j}}\Big) : w\in \Phi_n^{2k_{\dagger},j+1} \Big\}\\
		\ge & e^{-Cj_{\diamond}}\ge  n^{-C'}.
	\end{split}
\end{equation}

For $\mathbb{I}_3$, in Case (a), it follows from (\ref{final_5.24_dagger}) that 
\begin{equation}\label{final_5.26_dagger}
	\mathbf{L}(\eta_{\dagger})\le  \mathbf{L}(\eta_{\diamond})+\sum\nolimits_{1\le j\le j_{\diamond}-1} \mathbf{L}(\eta^j_{\dagger}). 
\end{equation}
Recall that $\mathbf{L}(\eta_{\diamond})\le \tfrac{\cref{remove_interlock2}\land \cref{prob_surrounding}}{100\Cref{remove_interlock1}}\cdot \sqrt{n}$ and $\mathbf{L}(\eta^j_{\dagger})\le 12\Cref{remove_interlock1}$ for all $1\le j\le j_\diamond-1$. Moreover, we have $j_\diamond-1 \le \mathbf{L}(\eta_{\diamond})$ since $t_{j_\diamond}=t_{\diamond}\le \mathbf{L}(\eta_{\diamond})$ and $t_{j}^{-}\ge t_{j-1}^++1$ for all $1\le j\le j_{\diamond}-1$. Combining these estimates with (\ref{final_5.26_dagger}), we have
\begin{equation}\label{final_5.27_dagger}
	\mathbf{L}(\eta_{\dagger})\le (12\Cref{remove_interlock1}+1)\cdot \tfrac{\cref{remove_interlock2}\land \cref{prob_surrounding}}{100\Cref{remove_interlock1}}\cdot \sqrt{n} \le \tfrac{\cref{remove_interlock2}\land \cref{prob_surrounding}}{4}\cdot \sqrt{n}. 
\end{equation}
Note that (\ref{final_5.27_dagger}) holds trivially in Case (b). Therefore, by Lemma \ref{lemma_move_along}, $\widehat{\lambda}(\mathscr{G})\le \lambda(\mathscr{G})$ and (\ref{final_5.27_dagger}), we obtain
\begin{equation}\label{ineq_I3}
	\mathbb{I}_3\gtrsim \big[\lambda(\mathscr{G})\big]^{-\tfrac{\cref{remove_interlock2}\land \cref{prob_surrounding}}{4}\cdot \sqrt{n}}.
\end{equation}

For $\mathbb{I}_4$, we need the following lemma, which provides an lower bound for the probability of bypassing an interlocked cluster. Recall that $\mathfrak{I}_A$ is the collection of all interlocked clusters of $A$.

\begin{lemma}\label{lemma_prob_surrounding}
	There exists a constant $\cl\label{prob_surrounding}(\mathscr{G})>0$ such that for any $A\subset \mathfs{V}$, $D\in \mathfrak{I}_A$ and $x,y\in \partial^{\mathrm{o}}_{\infty,*} D$, 
	\begin{equation}
		\mathbb{P}_x\left(\tau_{y}< \tau_A \right) \gtrsim  \big[\lambda(\mathscr{G})\big]^{-|D|+\cref{prob_surrounding}\sqrt{|D|}}. 
	\end{equation}
\end{lemma}

The proof of Lemma \ref{lemma_prob_surrounding} will be presented in Section \ref{section_length_bypass}. By Lemma \ref{lemma_prob_surrounding} we have 
\begin{equation}\label{ineq_I4}
	\mathbb{I}_4\gtrsim \big[\lambda(\mathscr{G})\big]^{-|A_{(\bm{0})}|+\cref{prob_surrounding}\sqrt{|A_{(\bm{0})}|}}.
\end{equation} 
Thus, combining (\ref{new_311}) and the aforementioned estimates for $\{\mathbb{I}_j\}_{j=1}^4$, we obtain 
\begin{equation}\label{ineq_ADD}
	\mathbb{H}_{A_{\mathcal{E}}^{\dagger}}(\bm{0}) \gtrsim  n^{-\frac{1}{2}-C'} \big[\lambda(\mathscr{G})\big]^{-\tfrac{\cref{remove_interlock2}\land \cref{prob_surrounding}}{4}\cdot \sqrt{n}-|A_{(\bm{0})}|+\cref{prob_surrounding}\sqrt{|A_{(\bm{0})}|}}. 
\end{equation}

For $A\in \mathcal{A}_n(\mathscr{G})$, we denote $n_1:=|A|-|A_{\mathcal{D}}|$, $n_2:=|A_{\mathcal{D}}|-|A_{\mathcal{E}}|$, $n_3:=|A_{\mathcal{E}}|-|A_{(\bm{0})}|$ and $n_4:=|A_{(\bm{0})}|$. Note that $n=\sum_{i=1}^{4}n_i$. By (\ref{ineq_3.11}), (\ref{ineq_AD-ADD}) and (\ref{ineq_ADD}), 
\begin{equation}\label{new_3.15}
	\begin{split}
		\mathbb{H}_A(\bm{0})\ge&  [\lambda(\mathscr{G})]^{-(n_1+n_2+n_4)+\cref{remove_interlock2}\sqrt{n_2}+\cref{const_4.2}n_1+\cref{prob_surrounding}\sqrt{n_4}-\tfrac{\cref{remove_interlock2}\land \cref{prob_surrounding}}{4}\cdot \sqrt{n}}\cdot e^{- \frac{\Cref{remove_one_vertex} n_3}{\lfloor \cref{ball_section}\sqrt{n}\rfloor}} \cdot n^{-\frac{1}{2}-C'}\\
		\ge &[\lambda(\mathscr{G})]^{-n-\tfrac{\cref{remove_interlock2}\land \cref{prob_surrounding}}{2}\cdot \sqrt{n}+(\cref{remove_interlock2}\land \cref{prob_surrounding})\cdot (\sqrt{n_1}+\sqrt{n_4})+\cref{const_4.2}( n_2+n_3)},
	\end{split}
\end{equation}
where in the second line we used the fact that for all large enough $n$,
\begin{equation*}
	\frac{\Cref{remove_one_vertex} n_3}{\lfloor \cref{ball_section}\sqrt{n}\rfloor} \le  (1-\cref{const_4.2}) n_3\ \ \text{and} \ \ n^{-\frac{1}{2}-C'}\ge [\lambda(\mathscr{G})]^{-\tfrac{\cref{remove_interlock2}\land \cref{prob_surrounding}}{4}\cdot \sqrt{n}}.  
\end{equation*}
Noting that $\sqrt{n_1}+\sqrt{n_4}\ge \sqrt{n_1+n_4}=\sqrt{n-(n_2+n_3)}$, we derive from (\ref{new_3.15}) that 
\begin{equation}\label{new_3.16}
	\begin{split}
		\mathbb{H}_A(\bm{0})
		\ge   [\lambda(\mathscr{G})]^{-n-\tfrac{\cref{remove_interlock2}\land \cref{prob_surrounding}}{2}\cdot \sqrt{n}+(\cref{remove_interlock2}\land \cref{prob_surrounding}) \cdot \sqrt{n-(n_2+n_3)}+\cref{const_4.2}( n_2+n_3)}. 
	\end{split}
\end{equation}
Since the function $f(x)=(\cref{remove_interlock2}\land \cref{prob_surrounding})\cdot \sqrt{n-x}+\cref{const_4.2}x$ for $x\in [0,n]$ takes the minimum at $x=0$ (for all sufficiently large $n$), one has 
\begin{equation}\label{new_3.17}
	(\cref{remove_interlock2}\land \cref{prob_surrounding})\cdot \sqrt{n-(n_2+n_3)}+\cref{const_4.2}( n_2+n_3) \ge (\cref{remove_interlock2}\land \cref{prob_surrounding})\cdot \sqrt{n}. 
\end{equation}
Thus, by (\ref{new_3.16}) and (\ref{new_3.17}) we get the lower bounds in Theorem \ref{theorem1.1} with $\cref{thm1_2}=\tfrac{\cref{remove_interlock2}\land \cref{prob_surrounding}}{2}$:
\begin{equation}
	\mathbb{H}_A(\bm{0})\ge [\lambda(\mathscr{G})]^{-n+\tfrac{\cref{remove_interlock2}\land \cref{prob_surrounding}}{2}\cdot \sqrt{n}}.       \pushQED{\qed} \qedhere
	\popQED
\end{equation}

To sum up, for the lower bounds in Theorem \ref{theorem1.1}, it suffices to prove Lemmas \ref{lemma_remove_distant}, \ref{lemma_packed3}, \ref{lemma_remove_interlock} and \ref{lemma_prob_surrounding}. The remainder of this paper is organized as follows. We prove Lemma \ref{lemma_packed3} in Section \ref{section_diameter}. Subsequently, Lemma \ref{lemma_prob_surrounding} is established in Section \ref{section_length_bypass}. Finally, we confirm Lemmas \ref{lemma_remove_distant} and \ref{lemma_remove_interlock} in Section \ref{section_qra}.

	\section{Diameter of the packed set}\label{section_diameter}

This section is devoted to the proof of Lemma \ref{lemma_packed3}, which shows that the diameter of a packed set is bounded by a polynomial function of its cardinality (recalling the definition of packed sets in Definition \ref{def_sparse}). This proof is based on induction (with respect to the cardinality of the set). Precisely, for any packed set $A$, we intend to find a packed proper subset $D\subsetneq A$ such that $A\setminus D$ is also packed. Note that the diameters of $D$ and $A\setminus D$ are both bounded according to the inductive hypothesis. Moreover, since $A$ is packed, the distance between $D$ and $A\setminus D$ is bounded by the diameters of both $D$ and $A\setminus D$ up to a constant factor. These estimates together imply the desired bound for the diameter of $A$.

As mentioned above, the key is to prove the existence of a packed subset $D\subsetneq A$ such that $A\setminus D$ is also packed. In light of this, we introduce the so-called ``prior subsets'' as follows, which is later proved to satisfy the desired property for $D$.

\begin{definition}\label{def_prior}
	(1) For any $D\subsetneq A\subset \mathfs{V} $, we say $D$ is a maximal packed (proper) subset of $A$ if $D$ is packed and for any $\widetilde{D}\subsetneq A$ with $D\subsetneq \widetilde{D}$, $\widetilde{D}$ is sparse.

	\noindent (2) For any $D\subsetneq A\subset \mathfs{V} $, we say $D$ is prior in $A$ if $D$ is a maximal packed subset of $A$ and satisfies either $\bm{0}\in D$ or $\mathrm{diam}(D)\ge \mathrm{diam}(\hat{D})$ for all packed $\hat{D}\subsetneq A$.

\end{definition}

Note that in any $A\subset \mathfs{V}$ with $|A|\ge 2$, there exists a packed proper subset (using Lemma \ref{lemma_packed0}). Consequently, there also exists a prior subset in $A$. Note that $A$ might have more than one prior subset.

\begin{lemma}\label{lemma_packed2}
	For any packed $A\subset \mathfs{V}$ and $D\subsetneq A$ with $\bm{0}\notin A\setminus D$, if $D$ is prior in $A$, then the following holds:
	\begin{enumerate}
		\item  For any packed $F\subsetneq A\setminus D$, $\mathbf{d}(F,D)> \Cref{delicate_green}[\mathrm{diam}(F)\vee 1]$.

		\item  $A\setminus D$ is packed.
	\end{enumerate} 
\end{lemma}

Before proving Lemma \ref{lemma_packed2}, we need the following lemma as preparation.

\begin{lemma}\label{lemma_packed1}
	For any disjoint $D_1,D_2\subset \mathfs{V}$, if $D_1$ and $D_2$ are packed and satisfy $\mathbf{d}(D_1,D_2)\le \Cref{delicate_green}[\mathrm{diam}(D_i)\vee 1]$ for $i\in \{1,2\}$ with $\bm{0}\notin D_i$, then $D_1\cup D_2$ is packed.
\end{lemma}
\begin{proof}
	Arbitrarily take $F\subsetneq D_1\cup D_2$ with $\bm{0}\notin F$. If $F=D_i$ for some $i\in \{1,2\}$, since $\mathbf{d}(D_1,D_2)\le \Cref{delicate_green}[\mathrm{diam}(D_i)\vee 1]$ one has 
	\begin{equation*}
		\mathbf{d}(F,D_1\cup D_2\setminus F)\le \Cref{delicate_green}[\mathrm{diam}(F)\vee 1]. 
	\end{equation*} 
	Otherwise, without loss of generality, assume that $F\cap D_1\neq \emptyset$ and $F\neq D_1$. When $D_1\setminus F\neq \emptyset$, since $D_1$ is packed, we have 
	\begin{equation*}
		\mathbf{d}(F,D_1\cup D_2\setminus F)\le \mathbf{d}(F\cap D_1,D_1\setminus F) \le  \Cref{delicate_green}[\mathrm{diam}(F\cap D_1)\vee 1] \le \Cref{delicate_green}[\mathrm{diam}(F)\vee 1]. 
	\end{equation*}
	When $D_1\setminus F= \emptyset$ and $F\neq D_1$ (i.e. $F\setminus D_1$ is a non-empty proper subset of $D_2$), since $D_2$ is packed, one has 
	\begin{equation*}
		\mathbf{d}(F,D_1\cup D_2\setminus F)= 	\mathbf{d}(F\setminus D_1 ,D_2\setminus F)\le \Cref{delicate_green}[\mathrm{diam}(F\setminus D_1)\vee 1] \le \Cref{delicate_green}[\mathrm{diam}(F)\vee 1]. 
	\end{equation*}
	In conclusion, $D_1\cup D_2$ is packed.
\end{proof}

Now we are ready to prove Lemma \ref{lemma_packed2}.

\begin{proof}[Proof of Lemma \ref{lemma_packed2}]
	 We prove Item (1) by contradiction. Assume that there is a packed $F_{\diamond}\subsetneq A\setminus D$ such that $\mathbf{d}(F_{\diamond},D)\le \Cref{delicate_green}[\mathrm{diam}(F_\diamond)\vee 1]$. For any $K\subsetneq D\cup F_{\diamond}$ with $\bm{0}\notin K$, we have  
	\begin{itemize}

		\item If $K\subsetneq F_{\diamond}$, since $F_{\diamond}$ is packed, one has 
		\begin{equation*}
			\mathbf{d}(K,F_{\diamond}\cup D\setminus K) \le	\mathbf{d}(K,F_{\diamond}\setminus K)\le  \Cref{delicate_green}[\mathrm{diam}(K)\vee 1]. 
		\end{equation*}

		\item If $K=F_{\diamond}$, since $\mathbf{d}(F_{\diamond}, D )\le \Cref{delicate_green}[\mathrm{diam}(F_{\diamond})\vee 1]$ we have 
		$$\mathbf{d}(K,F_{\diamond}\cup D\setminus K)\le \Cref{delicate_green}[\mathrm{diam}(K)\vee 1].$$

		\item If $K\cap D\neq \emptyset$ and $K\neq D$, since $D$ is packed, one has 
		\begin{equation*}
			\mathbf{d}(K,F_{\diamond}\cup D\setminus K) \le \mathbf{d}(K\cap D, D\setminus K)\le \Cref{delicate_green}[\mathrm{diam}(K\cap D)\vee 1] \le \Cref{delicate_green}[\mathrm{diam}(K)\vee 1].
		\end{equation*}

		\item  If $K=D$ and $\bm{0}\notin D$, by the priority of $D$, we have $\mathrm{diam}(F_\diamond)\le \mathrm{diam}(D)=\mathrm{diam}(K)$ and therefore, 
		\begin{equation*}
			\mathbf{d}(K,F_{\diamond}\cup D\setminus K)= \mathbf{d}(D,F_{\diamond})\le \Cref{delicate_green}[\mathrm{diam}(F_\diamond)\vee 1] \le \Cref{delicate_green}[\mathrm{diam}(K)\vee 1]. 
		\end{equation*}

	\end{itemize}
	Thus, $D\cup F_{\diamond}$ is packed, which is contradictory to the maximality of $D$.

	 We also prove Item (2) by contradiction. Assume that $A\setminus D$ is sparse. We decompose $A$ into a sequence of subsets $\{F_i\}_{i=0}^k$, which are defined recursively as follows. We set $F_0=D$. For each $i\ge 1$, suppose that we already get $F_0,...,F_{i-1}$. Then in Step $i$, $F_i$ is defined according to the following cases: 
	\begin{itemize}
		\item  If $A\setminus \cup_{j=0}^{i-1}F_j$ is packed, we set $F_i=A\setminus \cup_{j=0}^{i-1}F_j$ and stop the construction;

		\item  Otherwise, we take $F_i$ as a prior subset in $A\setminus \cup_{j=0}^{i-1}F_j$ (in a predetermined manner).

	\end{itemize}
	Here are some observations on $\{F_i\}_{i=0}^k$: 
	\begin{enumerate}
		\item[(a)]  $A=\cup_{i=0}^{k}F_i$; for any $0\le i\le k-1$, $F_i$ is prior in $\cup_{j=i}^{k} F_j$; $F_k$ is packed. These are straightforward by the construction.

		\item[(b)]   $k\ge 2$. In fact, since $A\setminus D$ is sparse, the construction for $\{F_i\}_{i=0}^k$ cannot stop in Step $1$, which implies that $k\ge 2$.

		\item[(c)]   For any $0\le i_1<i_2 \le k$ such that $i_2\le k-1$, or $i_2=k$ and $i_1\le k-2$,  
		\begin{equation}\label{new3.5}
			\mathbf{d}(F_{i_1},F_{i_2}) > \Cref{delicate_green}[\mathrm{diam}(F_{i_2})\vee 1].  
		\end{equation}
		To see this, note that $F_{i_1}$ is prior in $\cup_{j=i_1}^{k}F_j$, and $F_{i_2}\subsetneq \cup_{j=i_1+1}^{k}F_j$ is packed (by Observation (a)). Thus, by applying Item (1) (with $A=\cup_{j=i_1}^{k}F_j$ and $D=F_{i_1}$), we get (\ref{new3.5}).

		\item[(d)]  $\mathbf{d}(F_{k-1},F_{k}) \le  \Cref{delicate_green}[\mathrm{diam}(F_{k-1})\vee 1]$. In fact, by Observation (c), one has 
		\begin{equation*}
			\mathbf{d}(F_{k-1},\cup_{i=0}^{k-2} F_{i}) > \Cref{delicate_green}[\mathrm{diam}(F_{k-1})\vee 1].
		\end{equation*}
		Combined with the fact that (noting that $A$ is packed)
		$$\mathbf{d}(F_{k-1},\cup_{i=0}^{k-2} F_{i} \cup F_{k})=\mathbf{d}(F_{k-1},A\setminus F_{k-1})\le  \Cref{delicate_green}[\mathrm{diam}(F_{k-1})\vee 1],$$
		it yields $\mathbf{d}(F_{k-1},F_{k}) \le  \Cref{delicate_green}[\mathrm{diam}(F_{k-1})\vee 1]$.

	\end{enumerate}
	
	By Observation (c), we have 
	\begin{equation}\label{new_3.6}
		\mathbf{d}(\cup_{i=0}^{k-2}F_{i},F_{k}) > \Cref{delicate_green}[\mathrm{diam}(F_{k})\vee 1].   
	\end{equation}
	Note that in (\ref{new_3.6}) we need $k\ge 2$ (by Obvervation (b)) to ensure that $\cup_{i=0}^{k-2}F_{i}\neq \emptyset$. Moreover, since $A$ is packed, by Observation (a) one has 
	\begin{equation}\label{new_3.7}
		\mathbf{d}(\cup_{i=0}^{k-1}F_{i},F_{k})= \mathbf{d}(A\setminus F_{k},F_{k}) \le  \Cref{delicate_green}[\mathrm{diam}(F_{k})\vee 1]. 
	\end{equation}
	Combining (\ref{new_3.6}) and (\ref{new_3.7}), we obtain 
	\begin{equation}\label{new3.9}
		\mathbf{d}(F_{k-1},F_{k}) \le  \Cref{delicate_green}[\mathrm{diam}(F_{k})\vee 1].
	\end{equation}
	Thus, by Lemma \ref{lemma_packed1}, (\ref{new3.9}) and Observation (d), we know that $F_{k-1}\cup F_{k}$ is packed. However, when $A\setminus \cup_{i=0}^{k-2}F_i = F_{k-1}\cup F_{k}$ is packed, the construction will stop in Step $k-1$, which is a contradiction. To sum up, we prove that $A\setminus D$ is packed and complete the proof.
\end{proof}


Based on Lemma \ref{lemma_packed2}, now we implement the inductive proof for Lemma \ref{lemma_packed3}.

\begin{proof}[Proof of Lemma \ref{lemma_packed3}]
	We first use induction to prove Item (2). I.e., for any packed $A\subset \mathfs{V}$ with $\bm{0}\notin A$, one has 
	\begin{equation}\label{new3.11}
		\mathrm{diam}(A)\le \Cref{delicate_green}|A|^{\Cref{const_packed}}.
	\end{equation}
	When $|A|=1$, (\ref{new3.11}) holds since $\mathrm{diam}(A)=0$. For $k\ge 2$, assume that (\ref{new3.11}) is valid for all $A\subset \mathfs{V}$ with $|A|\le k-1$. For $A\subset \mathfs{V}$ with $|A|=k$, we arbitrarily take a prior subset $D\subsetneq A$. By Lemma \ref{lemma_packed2}, $A\setminus D$ is packed. Thus, by the inductive hypothesis and the fact that 
	$\mathbf{d}(D, A\setminus D)\le \min\{\Cref{delicate_green}[\mathrm{diam}(D)\vee 1],\Cref{delicate_green}[\mathrm{diam}(A\setminus D)\vee 1]\}$ (since $A$ is packed and $\bm{0}\notin A$), we get 
	\begin{equation}\label{new3.12}
		\begin{split}
			\mathrm{diam}(A)\le& \mathrm{diam}(D)+\mathrm{diam}(A\setminus D)+\mathbf{d}(D,A\setminus D)\\
			\le &\Cref{delicate_green}|D|^{\Cref{const_packed}}+ \Cref{delicate_green}(k-|D|)^{\Cref{const_packed}}+\Cref{delicate_green}^2  \min\{|D|^{\Cref{const_packed}}, (k-|D|)^{\Cref{const_packed}} \}\\
			\le & \Cref{delicate_green} \max \big\{(k-t)^{\Cref{const_packed}}+(\Cref{delicate_green}+1)t^{\Cref{const_packed}}:t\in [1,\tfrac{k}{2}] \big\}.
		\end{split}
	\end{equation}
	Since $f(t)=(k-t)^{\Cref{const_packed}}+(\Cref{delicate_green}+1)t^{\Cref{const_packed}}$ for $t\in [1,\tfrac{k}{2}]$ takes the maximum at $t=\tfrac{k}{2}$, we know that the right-hand side of (\ref{new3.12}) is at most (recall that $\Cref{const_packed}=\log_2(\Cref{delicate_green}+2)$)
	\begin{equation}
		\Cref{delicate_green}(\Cref{delicate_green}+2)\big(\tfrac{k}{2}\big)^{\Cref{const_packed}}=\Cref{delicate_green}k^{\Cref{const_packed}}.
	\end{equation}
	Thus, we complete the induction and confirm (\ref{new3.11}).

	Now we prove Item (1) using induction. Since (\ref{new3.11}) is established, we only need to focus on the case when $\bm{0}\in A$. For $k\ge 2$, assume that $\mathrm{diam}(A)\le (\Cref{delicate_green}^2+\Cref{delicate_green}) |A|^{\Cref{const_packed}}$ holds for all $A\subset \mathfs{V}$ with $|A|\le k-1$. For $A\subset \mathfs{V}$ with $|A|=k$, let $D$ be a maximal packed subset of $A$ that contains $\bm{0}$. Noting that $D$ is prior (recall Definition \ref{def_prior}) and applying Lemma \ref{lemma_packed2}, we know that $A\setminus D$ is packed. Therefore, by (\ref{new3.11}),
	\begin{equation}\label{ineq4.9}
		\mathrm{diam}(A\setminus D)\le  \Cref{delicate_green}(|A|-|D|)^{\Cref{const_packed}}. 
	\end{equation}
	Meanwhile, since $D$ is packed, by the inductive hypothesis one has
	\begin{equation}\label{ineq4.8}
		\mathrm{diam}(D)\le (\Cref{delicate_green}^2+\Cref{delicate_green}) |D|^{\Cref{const_packed}}. 
	\end{equation}
	Combining (\ref{ineq4.9}), (\ref{ineq4.8}) and the fact that $\mathbf{d}(D,A\setminus D)\le \Cref{delicate_green}[\mathrm{diam}(A\setminus D)\vee 1]$ (since $A$ is packed and $\bm{0}\notin A\setminus D$), we conclude Item (1):  
	\begin{equation*}
		\begin{split}
			\mathrm{diam}(A) \le &\mathrm{diam}(D)+  \mathbf{d}(D,A\setminus D)+\mathrm{diam}(A\setminus D) \\
			\le &\mathrm{diam}(D)+  (\Cref{delicate_green}+1)\mathrm{diam}(A\setminus D)    \\
			\le &(\Cref{delicate_green}^2+\Cref{delicate_green}) \left[ |D|^{\Cref{const_packed}}+(|A|-|D|)^{\Cref{const_packed}}\right] \le (\Cref{delicate_green}^2+\Cref{delicate_green}) |A|^{\Cref{const_packed}}.  \qedhere
		\end{split}
	\end{equation*}
\end{proof}

	\section{Construction of an efficient bypass}\label{section_length_bypass}

In this section, we aim to establish Lemma \ref{lemma_prob_surrounding}. The key is to show that between each pair of vertices on the exterior outer $*$-boundary of an interlocked cluster of a set $A$, it is feasible to construct a bypass (i.e. a path in $A^c$) with the following properties: (i) its length is bounded by a specific threshold; (ii) as the length approaches this threshold, the bypass includes sufficiently many positions that have an extra neighbor in $A^c$ besides the previous and next steps within the bypass. To be precise, we first introduce the definition of free vertices, which are candidates for the positions with an extra neighbor.

\begin{definition}[free vertex]\label{def_free_v}
	For any $A\subset \mathfs{V}$ and $z\in A^c$, we say $z$ is a free vertex (in $A$) if $|N(z)\setminus A|\ge 3$.
\end{definition}

In the subsequent proof, we mainly focus on the free vertices on the exterior outer $*$-boundary. To formulate the condition for a desired bypass, we define the following quantity. For $A\subset \mathfs{V}$ and a collection $\mathcal{E}$ of oriented edges in $A^c$, let 
\begin{equation}\label{statistic_X}
	\mathbf{X}_A(\mathcal{E}):= \big|\mathcal{E}\big| -c_{\ddagger}\cdot  \big| \big\{ \vec{e}_{x,y}\in\mathcal{E}: x\ \text{is free in}\ A  \big\}\big|,
\end{equation}
where $c_{\ddagger}=c_{\ddagger}(\mathscr{G}):= \log_{\widehat{\lambda}(\mathscr{G})}(\frac{[\mathrm{deg}(\mathscr{G})]^2}{[\mathrm{deg}(\mathscr{G})]^2-1})$ (recalling $\widehat{\lambda}(\mathscr{G})$ in Lemma \ref{lemma_move_along}). Note that $\mathbf{X}_A(\mathcal{E})$ is increasing with respect to $\mathcal{E}$ since $c_{\ddagger}\in(0,1)$.	Recall that $\mathfrak{I}_A$ denotes the collection of all interlocked clusters of $A$, and that $\nu(\mathscr{G}):=\frac{\mathrm{deg}(\mathscr{G})-2\cdot \mathbbm{1}_{\mathscr{G}=\mathscr{T}}}{\mathrm{deg}(\mathscr{G})-2}$ (i.e. $\nu(\mathbb{Z}^2)=2$, $\nu(\mathscr{T})=1$ and $\nu(\mathscr{H})=3$; see Lemma \ref{lemma_length_tunnel}). Now we present the main result of this section as follows:

\begin{proposition}\label{lemma_length_bypass}
	There exists $\cl\label{length_bypass}(\mathscr{G})>0$ such that for any $A\subset \mathfs{V}$, $D\in \mathfrak{I}_A$ and $x,y\in \partial_{\infty,*}^{\mathrm{o}} D$, there exists a path $\eta$ satisfying $\eta(0)=x$, $\eta(-1)=y$, $\mathbf{R}^{\mathrm{v}}(\eta)\subset A^c$ and $\mathbf{X}_A(\mathbf{R}^{\mathrm{e}}(\eta))\le \nu(\mathscr{G}) |D|- \cref{length_bypass}\sqrt{|D|}+ 4 $. 
\end{proposition}

With the help of Proposition \ref{lemma_length_bypass}, the proof of Lemma \ref{lemma_prob_surrounding} is straightforward.

\begin{proof}[Proof of Lemma \ref{lemma_prob_surrounding} assuming Proposition \ref{lemma_length_bypass}]
	For $A\subset \mathfs{V}$, $D\in \mathfrak{I}_A$ and $x,y\in \partial_{\infty,*}^{\mathrm{o}} D$, let $\eta$ be the path found by Proposition \ref{lemma_length_bypass}. We assume that for $1\le i\le \mathbf{L}(\eta)-1$, $\eta(i)$ is free in $A$ if and only if $i\in \{i_1,...,i_k\}$. For each $1\le j\le k$, since $|N(\eta(i_j))\setminus A|\ge 3$ (by the definition of free vertices), there exists $w_{j}\in N(\eta(i_j)) \setminus (A\cup \{\eta(i_j-1),\eta(i_j+1) \})$. We denote 
	$$q_i:=\mathbb{P}_{\eta(\mathbf{L}(\eta)-i)}\left( \tau_{y}<\tau_{A}\right),\ \ \forall 0\le i\le \mathbf{L}(\eta).$$
	We also set $q_{\mathbf{L}(\eta)+1}=0$. In fact, the number sequence $\{q_i\}_{i=0}^{\mathbf{L}(\eta)+1}$ satisfying all conditions in Lemma \ref{lemma_number_sequence} with $a=\mathrm{deg}(\mathscr{G})$ and $b= \frac{[\mathrm{deg}(\mathscr{G})]^2}{[\mathrm{deg}(\mathscr{G})]^2-1}$: 
	\begin{enumerate} 	
		\item $q_0=1$ (since $\tau_{y}=0$), $q_{\mathbf{L}(\eta)+1}=0$ (by definition) and $0\le q_i\le 1$ for all $1\le i\le \mathbf{L}(\eta)$ (since each $q_i$ is a probability);

		\item When $1\le i\le \mathbf{L}(\eta)$ and $i\notin\{i_1,...,i_k\}$, $q_i\ge [\mathrm{deg}(\mathscr{G})]^{-1}(q_{i-1}+q_{i+1})$ (when the random walk is at $\eta(i)$, it can move forward to $\eta (i-1)$ or move backward to $\eta (i + 1)$);

		\item When $1\le i\le \mathbf{L}(\eta)$ and $i\in\{i_1,...,i_k\}$, $q_i\ge\frac{\mathrm{deg}(\mathscr{G})}{[\mathrm{deg}(\mathscr{G})]^2-1} (q_{i-1}+q_{i+1})$ (when the random walk is at some $\eta(i_j)$, the random walk not only can move forward or backward, but also can first move to $w_j$ and then move back $\eta(i_j)$, which implies that $q_i\ge [\mathrm{deg}(\mathscr{G})]^{-2}q_i+ [\mathrm{deg}(\mathscr{G})]^{-1}(q_{i-1}+q_{i+1})$).

	\end{enumerate}
	Note that $\widehat{\lambda}(\mathscr{G})= \alpha(\mathrm{deg}(\mathscr{G}))$. Thus, by Lemma \ref{lemma_number_sequence}, we obtain
	\begin{equation*}
		\mathbb{P}_x\left( \tau_{y}<\tau_{A}\right)= q_{\mathbf{L}(\eta)}  \gtrsim  \bigg( \frac{[\mathrm{deg}(\mathscr{G})]^2}{[\mathrm{deg}(\mathscr{G})]^2-1}\bigg)^k \Big[ \widehat{\lambda}(\mathscr{G}) \Big]^{-\mathbf{L}(\eta)} \gtrsim \Big[ \widehat{\lambda}(\mathscr{G})\Big] ^{-\mathbf{X}_A(\mathbf{R}^{\mathrm{e}}(\eta))}. 
	\end{equation*}
	Combined with $\mathbf{X}_A(\mathbf{R}^{\mathrm{e}}(\eta))\le \nu(\mathscr{G}) |D|- \cref{length_bypass}\sqrt{|D|}+ 4 $ (since $\eta$ satisfies the conditions in Proposition \ref{lemma_length_bypass}) and $[\widehat{\lambda}(\mathscr{G})]^{\nu(\mathscr{G})}=\lambda(\mathscr{G})$, it completes the proof.
\end{proof}

The rest of this section is devoted to the proof of Proposition \ref{lemma_length_bypass}, which highly relies on a tool named ``surrounding loop'', as we introduce in the next subsection.

\subsection{Surrounding loop}

We first present the definition of surrounding loops as follows. Recall that $\kappa(\mathbb{Z}^2)=\kappa(\mathscr{T})=2$ and $\kappa(\mathscr{H})=3$ in Item (ii) of Remark \ref{remark_interlock}.

\begin{definition}[surrounding loop]\label{def_sl}
	For any finite and $*$-connected $A\subset \mathfs{V}$, and any $x\in \partial_{\infty,*}^{\mathrm{o}}A$, we say an edge circuit $\eta$ with $\eta(0)=x$ is a surrounding loop of $A$ starting from $x$ if it satisfies the following conditions:
	\begin{enumerate}
		\item   $\mathbf{R}^{\mathrm{v}}(\eta)=\partial_{\infty,*}^{\mathrm{o}}A$;

		\item   For each $0\le i\le \mathbf{L}(\eta)-1$, there exists a face $\mathcal{S}$ such that $\vec{e}_{\eta(i),\eta(i+1)}\in \vec{\mathbf{e}}(\mathcal{S})$ and $\mathbf{v}(\mathcal{S})\cap A\neq \emptyset$;

		\item  $\mathbf{L}(\eta)\le 2[\nu(\mathscr{G}) |A|+ \kappa(\mathscr{G})]$.

	\end{enumerate}

\end{definition}

\begin{remark}
	Notably, Condition (3) in Definition \ref{def_sl} is sharp. In other words, for $\mathscr{G}\in \{\mathbb{Z}^2,\mathscr{T},\mathscr{H}\}$ and any integer $n\ge 1$, there exists a $*$-connected $A\subset \mathfs{V}$ with $|A|=n$ such that $\mathbf{L}(\eta)=2[\nu(\mathscr{G})n+ \kappa(\mathscr{G})]$ holds for every edge circuit $\eta$ satisfying Conditions (1) and (2) in Definition \ref{def_sl}. Illustrations for the case $n=4$ are given in Figure \ref{fig:surround_loop_sharp}. Precisely, for each set $A$ presented in Figure \ref{fig:surround_loop_sharp}, the corresponding red path $\eta$ is the unique edge circuit satisfying Conditions (1) and (2) in Definition \ref{def_sl}, and obeys the equality $\mathbf{L}(\eta)=2[\nu(\mathscr{G})|A|+ \kappa(\mathscr{G})]$. Similar examples can be constructed for other values of $n$. In fact, due to their optimality in terms of the length of the surrounding loop, the structures of these examples are employed in the construction of the spirals $D_n^{\mathscr{G}}$ for $\mathscr{G}\in \{\mathbb{Z}^2,\mathscr{T},\mathscr{H}\}$ (see Figures \ref{fig:spiral_Z2} and \ref{fig:spirals}).
\end{remark}

\include*{tikz_surrounding_loop_sharp}

The following lemma ensures the existence of surrounding loops.

\begin{lemma}\label{lemma_exist_sl}
	For any $*$-connected $A\subset \mathfs{V}$ and $x\in \partial_{\infty,*}^{\mathrm{o}}A$, there exists a surrounding loop of $A$ starting from $x$. 
\end{lemma}

\begin{remark}[An operable way to construct a surrounding loop] 
	Although Lemma \ref{lemma_exist_sl} ensures that there exists at least one edge circuit which satisfies all conditions in Definition \ref{def_sl}, it remains unclear how it forms. In fact, we can construct a surrounding loop recursively in the following way. For any $*$-connected $A\subset \mathfs{V}$ and $x\in \partial_{\infty,*}^{\mathrm{o}}A$, we start from an oriented edge $\vec{e}_{x,y}$ satisfying $y\in \partial_{\infty,*}^{\mathrm{o}}A$ and Condition (2) in Definition \ref{def_sl} (i.e. there exists a face $\mathcal{S}$ such that $\vec{e}_{x,y}\in \vec{\mathbf{e}}(\mathcal{S})$ and $\mathbf{v}(\mathcal{S})\cap A\neq \emptyset$). Suppose that we already construct the first $k$ steps of the path, which we denote by $(v_{0},...,v_{k})$. For the next step, we enumerate the vertices in $N(v_{k})$ in the counter-clockwise direction as $w_1,...,w_{\mathrm{deg}(v_k)}$, where $w_1$ is the vertex following $v_{k-1}$ within $\mathbf{v}(\mathcal{S})$ in the counter-clockwise direction, and $w_{\mathrm{deg}(v_k)}=v_{k-1}$.	Let $i_{\dagger}$ be the smallest integer in $[1,\mathrm{deg}(v_k)]$ such that $w_{i_{\dagger}}\in \partial_{\infty,*}^{\mathrm{o}}A$. If $\vec{e}_{v_{k},w_{i_{\dagger}}}=\vec{e}_{v_{j},v_{j+1}}$ for some $0\le j\le k-1$, stop the construction and take $\eta=(v_{0},...,v_{k})$; otherwise, let $v_{k+1}:=w_{i_{\dagger}}$ and continue the construction. In fact, we can prove that this $\eta$ is an edge circuit satisfying all the conditions in Definition \ref{def_sl}, and hence is a surrounding loop (however, this approach would make the proof more complicated). See Figure \ref{fig:surround loop} for some examples of surrounding loops constructed in this way. Readers may refer to \cite{biskup2015isoperimetry} for a similar object called ``rightmost paths''.

	\end{remark}

\include*{tikz_surrounding_loop}

To prove Lemma \ref{lemma_exist_sl}, we need the so-called ``marginal vertex'', which is introduced in \cite{psi} and is defined as follows.

\begin{definition}[marginal vertex]
	For any $*$-connected $A\subset \mathfs{V}$ and $z\in A$, $z$ is a marginal vertex (of $A$) if $N_{\infty}^A(z)$ is connected in $N_{\infty,*}^A(z)$.
\end{definition}

Note that $z\in A$ is a marginal vertex of $A$ when $N_{\infty}^A(z)= \emptyset$ (since $\emptyset$ is considered connected). For $\mathbb{Z}^2$, \cite[Lemma 3.2] {psi} shows that any non-empty, $*$-connected $A$ contains at least one vertex $z$ such that $A\setminus \{z\}$ is still $*$-connected. Moreover, \cite[Lemma 3.3] {psi} proves that such a vertex $z$ is marginal. These results are extended to all planar graphs in \cite[Remark 3.8]{psi}. To sum up, we have

\begin{lemma}\label{lemma_exist_marginal}
	For any non-empty, $*$-connected $A\subset \mathfs{V}$, there exists a marginal vertex $z\in A$ such that $A\setminus \{z\}$ is $*$-connected. 
\end{lemma}

In what follows, we construct the surrounding loop recursively. Before showing the details of the construction, we first provide an overview as follows. For any $*$-connected $A\subset \mathfs{V}$ with $|A|\ge 2$, let $z$ be a marginal vertex found by Lemma \ref{lemma_exist_marginal}. Suppose that we already construct a surrounding loop $\eta'$ for $A\setminus \{z\}$. If $N_{\infty}^A(z)= \emptyset$, it follows from Definition \ref{def_sl} that $\eta'$ is also a surrounding loop of $A$. If $N_{\infty}^A(z)\neq \emptyset$, we first prove that $\eta'$ must intersect $z$. When $\eta'$ reaches the neighbor of $z$, instead of traversing $z$, we force it to bypass $z$, thereby obtaining a circuit surrounding $A$. Due to the marginality of $z$, we can show that such a modification for the surrounding loop indeed preserves all properties of surrounding loops.

Now we provide the details of the construction. Suppose that $A\subset \mathfs{V}$ ($|A|\ge 2$) is $*$-connected, and $z$ is a marginal vertex of $A$ such that $A\setminus \{z\}$ is $*$-connected (the existence of such $z$ is ensured by Lemma \ref{lemma_exist_marginal}). As described in the overview, we only need to consider the case  $N_{\infty}^{A}(z)\neq \emptyset$. We enumerate $N(z)$ in the clockwise direction as $\{z_i\}_{i=1}^{\mathrm{deg}(\mathscr{G})}$. For convenience, for $1\le i\le \mathrm{deg}(\mathscr{G})$, we denote $z_j:= z_i$ for $j\equiv i \pmod{\mathrm{deg}(\mathscr{G})}$. Since $z$ is a marginal vertex of $A$, there exist $k_1,k_2\in [1,2\mathrm{deg}(\mathscr{G})-2]$ with $0\le k_2-k_1\le \mathrm{deg}(\mathscr{G})-1$ such that the following holds:
\begin{enumerate}
	\item  $z_i\in \partial^{\mathrm{o}}_{\infty}A$ if and only if $k_1\le i\le k_2$;

	\item    For each $k_1\le i\le k_2-1$, there exist a path $\eta_z^i$ and a face $\mathcal{S}_i$ such that 
	\begin{equation}\label{7.1}
		\eta_z^i(0)=z_i,\ \eta_z^i(-1)=z_{i+1},\ z\in \mathbf{v}(\mathcal{S}_i),\ \mathbf{R}^{\mathrm{e}}(\eta_z^i)\subset \vec{\mathbf{e}}(\mathcal{S}_i)\ \text{and}\ \mathbf{R}^{\mathrm{v}}(\eta_z^i)\subset \partial^{\mathrm{o}}_{\infty,*}A.
	\end{equation}
	Intuitively speaking, $\eta_z^i$ is the path connecting $z_i$ and $z_{i+1}$ within the unique face that includes $\{z,z_i,z_{i+1}\}$.

	\item  Let $\mathcal{S}_{k_1-1}$ (resp. $\mathcal{S}_{k_2}$) be the unique face satisfying $\vec{e}_{z_{k_1},z}\in \vec{\mathbf{e}}(\mathcal{S}_{k_1-1})$ (resp. $\vec{e}_{z,z_{k_2}}\in \vec{\mathbf{e}}(\mathcal{S}_{k_2})$). Then $\mathbf{v}(\mathcal{S}_{k_1-1})$ and $\mathbf{v}(\mathcal{S}_{k_2})$ both intersect $A':=A\setminus \{z\}$.

\end{enumerate}
Without loss of generality, we assume $k_1=1$ and $k_2=k'$ for some $k'\in [1,\mathrm{deg}(\mathscr{G})]$.

 \begin{lemma}\label{lemma7.3}
	Keep the notations above. Then we have
\begin{enumerate}
	\item $\{z,z_1,z_{k'}\}\subset \partial_{\infty,*}^{\mathrm{o}}A'$;

	\item For any $v\in (\partial_{\infty,*}^{\mathrm{o}}A')\setminus (\{z\}\cup \partial_{\infty,*}^{\mathrm{o}}A)$, $v$ and $\partial_{\infty,*}^{\mathrm{o}}A$ are not connected by $(\partial_{\infty,*}^{\mathrm{o}}A')\setminus \{z\}$. 

\end{enumerate}
\end{lemma}
\begin{proof}
	We first prove Item (1). Since $N_{\infty}^{A}(z)\neq \emptyset$, there exists $w\in A_{\infty}^c\subset (A')_{\infty}^c$ such that $w\sim z$,  which implies $z\in (A')_{\infty}^c$. Moreover, since $A$ is $*$-connected, there exists $y\in A'$ such that $y\sim_* z$, and thus $z\in \partial_{\infty,*}^{\mathrm{o}}A'$. For $z_1$, by Condition (3) for $\{z_i\}_{i=1}^{k'}$, we know that $\mathbf{v}(\mathcal{S}_{0})$ contains both $z_1$ and some vertex $v\in A'$. Therefore, since $z_1\sim_* v\in A'$ and $z_1\in (A')_{\infty}^c$ (which follows from $z_1\sim z$ and $z\in \partial_{\infty,*}^{\mathrm{o}}A'$), we get $z_1\in \partial_{\infty,*}^{\mathrm{o}}A'$. For the same reason, we also have $z_{k'}\in \partial_{\infty,*}^{\mathrm{o}}A'$.

	We prove Item (2) by contradiction. Suppose that $v\in (\partial_{\infty,*}^{\mathrm{o}}A')\setminus (\{z\}\cup \partial_{\infty,*}^{\mathrm{o}}A)$ and $\partial_{\infty,*}^{\mathrm{o}}A$ are connected by $(\partial_{\infty,*}^{\mathrm{o}}A')\setminus \{z\} \subset A^c$. Therefore, there exists a path within $A^c$ connecting $v$ and $\partial_{\infty,*}^{\mathrm{o}}A$. Thus, we have $v\in A_\infty^c$. Meanwhile, since $v\in \partial_{\infty,*}^{\mathrm{o}}A'$, we know that $v$ ($\neq z$) is $*$-adjacent to $A'$ and thus is also $*$-adjacent to $A$. Combined with $v\in A_\infty^c$, it yields that $v\in \partial_{\infty,*}^{\mathrm{o}}A$, which is contradictory to $v\in (\partial_{\infty,*}^{\mathrm{o}}A')\setminus (\{z\}\cup \partial_{\infty,*}^{\mathrm{o}}A)$.
\end{proof}

Suppose that $\eta'$ is a surrounding loop of $A'$. Since $\mathbf{R}^{\mathrm{v}}(\eta')= \partial_{\infty,*}^{\mathrm{o}}A'$ (by Condition (2) in Definition \ref{def_sl}) contains $z$, $z_1$ and $z_{k'}$ (by Item (1) of Lemma \ref{lemma7.3}), where $z_1,z_{k'}\in \partial_{\infty,*}^{\mathrm{o}}A$, we know that $\eta'$ intersects both $z$ and $\partial_{\infty,*}^{\mathrm{o}}A$.

\begin{lemma}\label{lemma_final_7.11}
	$\mathbf{R}^{\mathrm{e}}(\eta')$ contains an oriented edge $\vec{e}_{z,w}$ with $w\in \partial_{\infty,*}^{\mathrm{o}}A$.
\end{lemma}
\begin{proof}
	When $\partial_{\infty,*}^{\mathrm{o}}A'=\{z\}\cup \partial_{\infty,*}^{\mathrm{o}}A$, this lemma directly follows from $\mathbf{R}^{\mathrm{v}}(\eta')=\partial_{\infty,*}^{\mathrm{o}}A'$. When $\partial_{\infty,*}^{\mathrm{o}}A'\supsetneq\{z\}\cup \partial_{\infty,*}^{\mathrm{o}}A$, we prove this lemma by contradiction. Suppose that such $\vec{e}_{z,w}$ does not exist. Thus, since $\mathbf{R}^{\mathrm{v}}(\eta')$ is connected and intersects both $z$ and $\partial_{\infty,*}^{\mathrm{o}}A$, the edge circuit $\eta'$ must include an oriented edge $\vec{e}_{v,v'}$ such that $v\in (\partial_{\infty,*}^{\mathrm{o}}A')\setminus (\{z\}\cup \partial_{\infty,*}^{\mathrm{o}}A)$ and $v'\in \partial_{\infty,*}^{\mathrm{o}}A$. However, this implies that $v\in (\partial_{\infty,*}^{\mathrm{o}}A')\setminus (\{z\}\cup \partial_{\infty,*}^{\mathrm{o}}A)$ and $\partial_{\infty,*}^{\mathrm{o}}A$ are connected by $(\partial_{\infty,*}^{\mathrm{o}}A')\setminus \{z\}$, which is contradictory to Item (2) of Lemma \ref{lemma7.3}. Now we complete the proof.	\end{proof}

By Lemma \ref{lemma_final_7.11}, without loss of generality, we can further assume that $\eta'(0)=z$ and $\eta'(1)\in \partial_{\infty,*}^{\mathrm{o}}A$ (if not, we only need to apply a time-shift to $\eta'$ such that $\vec{e}_{\eta'(0),\eta'(1)}$ exactly equals the oriented edge $\vec{e}_{z,w}$ found by Lemma \ref{lemma_final_7.11}).

Recall the paths $\eta_z^{i}$ for $1\le i\le k'-1$ in (\ref{7.1}).

\begin{lemma}\label{lemma_7.4}
Keep the notations above. Let $t':= \{t\ge 1: \eta'(t)=z\}$ be the first time when $\eta'$ hits $z$ after the first step (where $t'\le \mathbf{L}(\eta')$ since $\eta'$ is an edge circuit). Then we have
	\begin{enumerate}
		\item  $\eta'(t'-1)=z_1$ and $\eta'(1)=z_{k'}$;

		\item $\eta'':=\eta'[1,t'-1]\circ \eta_z^{1}\circ ... \circ \eta_z^{k'-1}$ is an edge circuit;

		\item $\mathbf{R}^{\mathrm{v}}(\eta'')=\partial_{\infty,*}^{\mathrm{o}}A$;

		\item  $\mathbf{L}(\eta'')\le \mathbf{L}(\eta')+2\nu(\mathscr{G}) $. 
		
	\end{enumerate}  
\end{lemma}
\begin{proof}
	For Item (1), since $\mathbf{R}^{\mathrm{v}}(\eta'[1,t'-1])$ is disjoint from $A'$ and does not contain $z$ (by the minimality of $t'$), we have $\mathbf{R}^{\mathrm{v}}(\eta'[1,t'-1])\subset A^c$. Moreover, $\mathbf{R}^{\mathrm{v}}(\eta'[1,t'-1])$ is connected and contains $\eta'(1) \in \partial_{\infty,*}^{\mathrm{o}}A$. As a result, one has $\mathbf{R}^{\mathrm{v}}(\eta'[1,t'-1])\subset \partial_{\infty,*}^{\mathrm{o}}A$ and hence, $\eta'(1),\eta'(t'-1)\in N_{\infty}^{A}(z)= \{z_i\}_{i=1}^{k'}$. For each $2\le i\le k'$, note that $\mathcal{S}_i$ defined in (\ref{7.1}) is exactly the unique face satisfying $\vec{e}_{z_i,z}\in \vec{\mathbf{e}}(\mathcal{S}_i)$. Therefore, since $\mathbf{v}(\mathcal{S}_i)=\mathbf{R}^{\mathrm{e}}(\eta_z^i)\cup \{z\}\subset (A')^c$, it follows from Condition (2) in Definition \ref{def_sl} that $\vec{e}_{z_i,z}\notin \mathbf{R}^{\mathrm{e}}(\eta')$, which implies that $\eta'(t'-1)\neq z_i$ for $2\le i\le k'$ (noting that $\vec{e}_{\eta'(t'-1),z}=\vec{e}_{\eta'(t'-1),\eta'(t')}\in \mathbf{R}^{\mathrm{e}}(\eta')$). Combined with $\eta'(t'-1)\in \{z_i\}_{i=1}^{k'}$, it yields $\eta'(t'-1)=z_1$. For the same reason, we can also show that $\vec{e}_{z,z_i}\notin \mathbf{R}^{\mathrm{e}}(\eta')$ for $1\le i\le k'-1$, and thus obtain $\eta'(1)=z_{k'}$.

	For Item (2), by Item (1) and (\ref{7.1}) one has 
	$$\eta''(0)=\eta'(1)=z_{k'}=\eta_z^{k'-1}(-1)=\eta''(-1).$$ 
	Moreover, since $\eta'$ is an edge circuit, $\eta'[1,t'-1]$ does not traverse an edge more than once. In addition, for any $1\le i\le k'-1$ and $\vec{e}\in \mathbf{R}^{\mathrm{e}}(\eta_z^{i})$, (\ref{7.1}) implies that $\mathcal{S}_i$ is the unique face satisfying $\vec{e}\in \vec{\mathbf{e}}(\mathcal{S}_i)$. Thus, since $\mathbf{v}(\mathcal{S}_i)\cap A'=\emptyset$ (recalling that $\mathbf{v}(\mathcal{S}_i)\cap A= \{z\}$), it follows from Condition (2) in Definition \ref{def_sl} that such $\vec{e}\notin \mathbf{R}^{\mathrm{e}}(\eta')$. As a result, $\eta'[1,t'-1]$ does not traverse any edge in $\mathbf{R}^{\mathrm{e}}(\eta_z^{1}\circ ... \circ \eta_z^{k'-1})$. Meanwhile, it directly follows from (\ref{7.1}) that $\eta_z^{1}\circ ... \circ \eta_z^{k'-1}$ does not traverse the same edge more than once. To sum up, we conclude that $\eta''$ is an 	edge circuit.

	Now we prove Item (3). Since $\mathbf{R}^{\mathrm{v}}(\eta'[0,t'-1])\subset \partial_{\infty,*}^{\mathrm{o}}A$ (recalling the proof of Item (1)) and $\cup_{i=1}\mathbf{R}^{\mathrm{v}}(\eta_z^i)\subset \partial^{\mathrm{o}}_{\infty,*}A $ (by (\ref{7.1})), we have $\mathbf{R}^{\mathrm{v}}(\eta'')\subset\partial_{\infty,*}^{\mathrm{o}}A$. Therefore, since $\mathbf{R}^{\mathrm{v}}(\eta')=\partial_{\infty,*}^{\mathrm{o}}A'$ (by Condition (1) in Definition \ref{def_sl}), it remains to show
		\begin{equation}\label{final_add_7.5}
		\mathbf{R}^{\mathrm{v}}(\eta')\setminus \mathbf{R}^{\mathrm{v}}(\eta'')\subset (\partial_{\infty,*}^{\mathrm{o}}A')\setminus (\partial_{\infty,*}^{\mathrm{o}}A).
	\end{equation}
	In fact, by the definition of $\eta''$, one has 
	\begin{equation}\label{final_add_7.6}
		\mathbf{R}^{\mathrm{v}}(\eta')\setminus \mathbf{R}^{\mathrm{v}}(\eta'') \subset \mathbf{R}^{\mathrm{v}}(\eta'[t',\mathbf{L}(\eta')]).
	\end{equation}
	We define a sequence of return times $\{\tau_j\}_{1\le j\le m}$ as follows. Let $\tau_0=t'$. For $j\ge 0$, if $\tau_{j}=\mathbf{L}(\eta')$, we stop the construction; otherwise, define $\tau_{j+1}:=\{t>\tau_{j}:\eta'(t)=z\}$. It directly follows from the construction that 
	\begin{equation}\label{final_add_7.7}
		\mathbf{R}^{\mathrm{v}}(\eta'[t',\mathbf{L}(\eta')]) = \cup_{j=0}^{m-1} \mathbf{R}^{\mathrm{v}}(\eta'[\tau_j,\tau_{j+1}]). 
	\end{equation} 
 For each $0\le j\le m-1$, one has $\eta'(\tau_j)=z$ and $\eta'(\tau_j+1)\neq z_{k'}$ (since $\eta'$ is an edge circuit and $\vec{e}_{\eta'(0),\eta'(1)}=\vec{e}_{z,z_{k'}}$). However, it is shown during the proof of Item (1) that $\vec{e}_{z,w}\in \mathbf{R}^{\mathrm{e}}(\eta')$ and $w\in \partial_{\infty,*}^{\mathrm{o}}A$ implies $w=z_{k'}$. Therefore, we have $\eta'(\tau_j+1)\notin \partial_{\infty,*}^{\mathrm{o}}A$. Thus, since $\mathbf{R}^{\mathrm{v}}(\eta'[\tau_j+1,\tau_{j+1}-1])\subset A^c$ is connected, one has $\mathbf{R}^{\mathrm{v}}(\eta'[\tau_j+1,\tau_{j+1}-1])\cap \partial_{\infty,*}^{\mathrm{o}}A=\emptyset$ for all $0\le j\le m-1$. As a result,  
 \begin{equation}\label{final_add_7.8}
 	\cup_{j=0}^{m-1} \mathbf{R}^{\mathrm{v}}(\eta'[\tau_j,\tau_{j+1}]) \subset (\partial_{\infty,*}^{\mathrm{o}}A')\setminus (\partial_{\infty,*}^{\mathrm{o}}A). 
 \end{equation} 
 Combining (\ref{final_add_7.6}), (\ref{final_add_7.7}) and (\ref{final_add_7.8}), we get (\ref{final_add_7.5}) and thus conclude Item (3).

For Item (4), it follows from the definition of $\eta''$ that 
	\begin{equation*}\label{7.5}
		\mathbf{L}(\eta'')\le \mathbf{L}(\eta')-2+ \sum\nolimits_{1\le i\le k'-1} \mathbf{L}(\eta_z^i).  
	\end{equation*} 
	Thus, since $k'=|N^A_{\infty}(z)|\le \mathrm{deg}(\mathscr{G})-\mathbbm{1}_{\mathscr{G}=\mathscr{T}}$ (by Lemma \ref{new7.1}) and $\mathbf{L}(\eta_z^i)\le \xi(\mathscr{G})-2$ for all $1\le i\le k'-1$ (recall that $\xi(\mathscr{G})$ be the number of edges surrounding a face), 
	\begin{equation}\label{final_7.8}
		\mathbf{L}(\eta'') \le \mathbf{L}(\eta')+ [\xi(\mathscr{G})-2]\cdot [\mathrm{deg}(\mathscr{G})-\mathbbm{1}_{\mathscr{G}=\mathscr{T}}-1]-2.
	\end{equation}
	Meanwhile, note that (\ref{final_3.2}) implies
	\begin{equation}\label{final_7.8_new}
	\begin{split}
			&[\xi(\mathscr{G})-2]\cdot [\mathrm{deg}(\mathscr{G})-\mathbbm{1}_{\mathscr{G}
			=\mathscr{T}}-1]-2\\
			=& \frac{4[\mathrm{deg}(\mathscr{G})-\mathbbm{1}_{\mathscr{G}=\mathscr{T}}-1]}{\mathrm{deg}(\mathscr{G})-2}-2 = 2\cdot \frac{\mathrm{deg}(\mathscr{G})-\mathbbm{1}_{\mathscr{G}=\mathscr{T}}}{\mathrm{deg}(\mathscr{G})-2}=2\nu(\mathscr{G}). 
	\end{split}
	\end{equation}
Combining (\ref{final_7.8}) and (\ref{final_7.8_new}), we conclude Item (4). 
\end{proof}

With these preparations, now we are ready to prove Lemma \ref{lemma_exist_sl}. 

\begin{proof}[Proof of Lemma \ref{lemma_exist_sl}]
We prove this lemma by induction (with respect to the cardinality of $A$). When $|A|=1$, without loss of generality, assume that $A=\{\bm{0}\}$. For any $x\sim_* \bm{0}$, we take $\eta$ as the edge circuit which starts from $x$ and surrounds $\bm{0}$ clockwise within $N_{*}(\bm{0})$. It is easy to check that $\eta$ satisfies all conditions for a surrounding loop in Definition \ref{def_sl}.

	For any $*$-connected $A$ with $|A|\ge 2$, by Lemma \ref{lemma_exist_marginal}, there exists a marginal vertex $z\in A$ such that $A'=A\setminus \{z\}$ is $*$-connected. By the inductive hypothesis, there exists a surrounding loop $\eta'$ of $A'$. When $N^A_{\infty}(z)=\emptyset$, $\eta'$ satisfies Condition (1) in Definition \ref{def_sl} for $A$ since $\mathbf{R}^{\mathrm{v}}(\eta')=\partial^{\mathrm{o}}_{\infty,*}A'=\partial^{\mathrm{o}}_{\infty,*}A$. Moreover, $N^A_{\infty}(z)=\emptyset$ also implies that $z$ is not $*$-adjacent to $\partial^{\mathrm{o}}_{\infty,*}A=\mathbf{R}^{\mathrm{v}}(\eta')$, and thus $\mathbf{v}(\mathcal{S})\cap A=\mathbf{v}(\mathcal{S})\cap A'$ for all face $\mathcal{S}$ with $\mathbf{v}(\mathcal{S})\cap \mathbf{R}^{\mathrm{v}}(\eta')\neq \emptyset$. As a result, $\eta'$ also satisfies Condition (2) in Definition \ref{def_sl} for $A$. Note that Condition (3) in Definition \ref{def_sl} for $A$ is weaker than the corresponding condition for $A'$. To sum up, $\eta'$ is a surrounding loop of $A$ in this case. When $N^A_{\infty}(z)\neq \emptyset$, recall the path $\eta''$ defined in Item (2) of Lemma \ref{lemma_7.4}. In fact, by Lemma \ref{lemma_7.4} and Definition \ref{def_sl}, we know that $\eta''$ is a surrounding loop of $A$. In conclusion, we construct a surrounding loop $\eta$ of $A$ in both cases. For any $x\in \partial_{\infty,*}^{\mathrm{o}}A$, since $\mathbf{R}^{\mathrm{v}}(\eta)=\partial_{\infty,*}^{\mathrm{o}}A$, there exists $j\in [0,\mathbf{L}(\eta)]$ such that $\eta(j)=x$. Noting that the path $\eta[j,\mathbf{L}(\eta)]\circ \eta[0,j]$ is a surrounding loop of $A$ starting from $x$ (since it has the same range and length as $\eta$), we complete the induction.  
\end{proof}

	\subsection{Proof of Proposition \ref{lemma_length_bypass}}

In this subsection, we use the surrounding loops of $*$-clusters in $A$ to compose the desired path in Proposition \ref{lemma_length_bypass}. To bound $\mathbf{X}_A$ for the range of this path, two properties are required: (i) its length is at most $\nu(\mathscr{G}) |A|+ \kappa(\mathscr{G})$; (ii) as the length approaches the threshold $\nu(\mathscr{G}) |A|+ \kappa(\mathscr{G})$, the path includes sufficiently many free vertices. For the first property, we use Condition (3) in Definition \ref{def_sl} for the lengths of surrounding loops. For the second one, we use the maximality of the interlocked cluster to ensure that there won't be too many overlaps between an interlocked cluster and the remaining $*$-clusters in $A$.

For any finite $A\subset \mathfs{V}$, let $D\in \mathfrak{I}_A$ be an arbitrary interlocked cluster of $A$. Then we arbitrarily take $x,y\in \partial_{\infty,*}^{\mathrm{o}}D$.

\begin{lemma}\label{lemma7.6}
	There exists a sequence of $*$-clusters of $D$, denoted by $\mathcal{C}_*^{i}$ for $1\le i\le m$, such that the following conditions hold:
	\begin{enumerate}
		\item    $x\in \partial_{\infty,*}^{\mathrm{o}}\mathcal{C}_*^{0}$ and $y \in \partial_{\infty,*}^{\mathrm{o}}\mathcal{C}_*^{m}$;

		\item   For any $1\le i_1< i_2\le m$, $\mathcal{C}_*^{i_1}   \leftrightsquigarrow \mathcal{C}_*^{i_2}$ holds if and only if $i_2=i_1+1$.

	\end{enumerate}
\end{lemma}
\begin{proof}
	Since $D$ is an interlocked cluster, there exists a sequence of $*$-clusters of $D$ (say $\widetilde{\mathcal{C}}_*^{j}$ for $1\le j\le \widetilde{m}$) satisfying Condition (1) and a slightly weaker version of Condition (2) as follows: 
	\begin{equation}\label{7.8}
		\widetilde{\mathcal{C}}_*^{j}   \leftrightsquigarrow \widetilde{\mathcal{C}}_*^{j+1}\ \  \text{for all}\ 1\le j\le \widetilde{m}-1. 
	\end{equation}
	If Condition (2) fails for the sequence $\{\widetilde{\mathcal{C}}_*^{j}\}_{j=1}^{\widetilde{m}}$, then there must exist $1\le j_1< j_2\le \widetilde{m}$ with $j_2\ge j_1+2$ such that $\mathcal{C}_*^{j_1}   \leftrightsquigarrow \mathcal{C}_*^{j_2}$. However, in this case we can remove $\widetilde{\mathcal{C}}_*^{j}$ for $j_1+1\le j \le j_2-1$ from $\{\widetilde{\mathcal{C}}_*^{j}\}_{j=1}^{\widetilde{m}}$, and the new sequence still satisfies (\ref{7.8}). By repeating this removal until getting a sequence where Condition (2) holds true, we obtain the desired sequence $\{\mathcal{C}_*^{j}\}_{j=1}^{m}$.
\end{proof}


Next, we construct an edge circuit $\eta_{\ddagger}$ intersecting $x$ and $y$, and composed of surrounding loops of $*$-clusters of $D$. For convenience, we first fix some notations. 
\begin{itemize}
	\item  By Lemma \ref{lemma_exist_sl}, for any $*$-cluster $\mathcal{C}_*$ of $D$, and any $z\in \partial_{\infty,*}^{\mathrm{o}}\mathcal{C}_*$, we take a surrounding loop (denoted by $\ell_{\mathcal{C}_*}^{z}$) of $\mathcal{C}_*$ starting from $z$. We also require that $\ell_{\mathcal{C}_*}^{z}$ for $z\in \partial_{\infty,*}^{\mathrm{o}}\mathcal{C}_*$ can be transformed into one another by a time shift.

	\item    Recall that $\mathcal{I}(\mathcal{C}_*,\mathcal{C}_*')$ is the collection of interlocking vertices between the $*$-clusters $\mathcal{C}_*$ and $\mathcal{C}_*'$ (see the definition above Remark \ref{remark_interlock}).

	\item    Let $\{\mathcal{C}_*^{i}\}_{i=1}^{m}$ be the sequence of $*$-clusters found by Lemma \ref{lemma7.6}.

\end{itemize}

The construction of $\eta_{\ddagger}$ takes the following three steps. An illustration for this construction can be found in Figure \ref{fig:surround_interlock_cluster}.

\textbf{Step 1:}  We define a sequence $\{(z^i,t^i)\}_{i=1}^{m}$ as follows ($z^i\in \partial_{\infty,*}^{\mathrm{o}}D$ and $t_i\in \mathbb{N}$).

When $m=1$, we take $z^1=x$, and let $t^1$ be the first time when $\ell_{\mathcal{C}_*^1}^{z^1}$ intersects $y$.

When $m \ge 2$, we construct this sequence recursively. We set $z^1=x$ and let $t^1$ be the first time when $\ell_{\mathcal{C}_*^1}^{z^1}$ intersects $\mathcal{I}(\mathcal{C}_*^{1},\mathcal{C}_*^{2})$ (see the red dots in Figure \ref{fig:surround_interlock_cluster}). For $2\le i\le m$, suppose that we already construct $(z^{i-1},t^{i-1})$ satisfying $z^{i-1}\in \partial_{\infty,*}^{\mathrm{o}}\mathcal{C}_*^{i-1}$ and $\ell_{\mathcal{C}_*^{i-1}}^{z^{i-1}}(t^{i-1})\in \mathcal{I}(\mathcal{C}_*^{i-1},\mathcal{C}_*^{i})$ (which hold for $(z^{1},t^{1})$). We take $z^{i}:=\ell_{\mathcal{C}_*^{i-1}}^{z^{i-1}}(t^{i-1}-1)$. Note that $z_i\in \partial_{\infty,*}^{\mathrm{o}} \mathcal{C}_*^{i}$ since $\ell_{\mathcal{C}_*^{i-1}}^{z^{i-1}}(t^{i-1})$ is an interlocking vertex between $\mathcal{C}_*^{i-1}$ and $\mathcal{C}_*^{i}$. Moreover, when $2\le i\le m-1$ (resp. $i=m$), let $t^i$ be the first time when $\ell_{\mathcal{C}_*^{i}}^{z^{i}}$ intersects $\mathcal{I}(\mathcal{C}_*^{i},\mathcal{C}_*^{i+1})$ (resp. $y$).

\textbf{Step 2:} Based on $\{(z^i,t^i)\}_{i=1}^{m}$, we define the path
\begin{equation*}
	\eta_{x\to y}:=\ell_{\mathcal{C}_*^{1}}^{z_1}[0,t^1-1]\circ \ell_{\mathcal{C}_*^{2}}^{z_2}[0,t^2-1]\circ ... \circ \ell_{\mathcal{C}_*^{m}}^{z_{m}}[0,t^{m}]. 
\end{equation*}
We also define $\eta_{y\to x}$ as the analogue of $\eta_{x\to y}$, obtained by exchanging the roles of $x$ and $y$, and replacing $\mathcal{C}_*^{i}$ with $\mathcal{C}_*^{m+1-i}$ for all $1\le i\le m$. See the green (resp. yellow) path in Figure \ref{fig:surround_interlock_cluster} for an illustration of $\eta_{x\to y}$ (resp. $\eta_{y\to x}$).

\textbf{Step 3:} Let $\eta_{\ddagger}:=\eta_{x\to y} \circ \eta_{y\to x}$.

\include*{interlocked}

\begin{lemma}\label{lemma_7.7}
	The path $\eta_{\ddagger}$ satisfies the following properties:
	\begin{enumerate}
		\item   $\eta_{\ddagger}$ is an edge circuit; 
		
		\item   $\mathbf{R}^{\mathrm{e}}(\eta_{\ddagger})\subset \cup_{i=1}^{m} \mathbf{R}^{\mathrm{e}}(\ell_{\mathcal{C}_*^i}^{z^i})$;

	\item  $|\cup_{i=1}^{m} \mathbf{R}^{\mathrm{e}}(\ell_{\mathcal{C}_*^i}^{z^i})\setminus \mathbf{R}^{\mathrm{e}}(\eta_{\ddagger})|\ge 2\kappa(\mathscr{G})(m-1)$.

	\item  For any $w\in \mathbf{R}^{\mathrm{v}}(\eta_{\ddagger})$, if $w$ is not free in $A$, then either $N(w)\cap  A\subset D$ or $N(w)\cap  A\subset A\setminus D$ holds.

	\end{enumerate}
\end{lemma}
\begin{proof}
For Item (1), it suffices to show that for any $1\le i<i'\le m$, 
	\begin{equation}
		\mathbf{R}^{\mathrm{e}}(\ell_{\mathcal{C}_*^i}^{z}) \cap \mathbf{R}^{\mathrm{e}}(\ell_{\mathcal{C}_*^{i'}}^{z'}) =\emptyset,\  \forall  z\in \partial_{\infty,*}^{\mathrm{o}}\mathcal{C}_*^i\ \text{and}\  z'\in  \partial_{\infty,*}^{\mathrm{o}}\mathcal{C}_*^{i'}. 
	\end{equation}
	We prove this by contradiction. Assume that there eixsts $\vec{e}_{w,v}\in \mathbf{R}^{\mathrm{e}}(\ell_{\mathcal{C}_*^i}^{z}) \cap \mathbf{R}^{\mathrm{e}}(\ell_{\mathcal{C}_*^{i'}}^{z'})$ for some $i,i'$ and $z,z'$. By Condition (2) for surrounding loops (see Definition \ref{def_sl}), we know that the unique face $\mathcal{S}_{w,v}$ with $\vec{e}_{w,v}\in \vec{\mathbf{e}}(\mathcal{S}_{w,v})$ must intersect both $\mathcal{C}_*^i$ and $\mathcal{C}_*^{i'}$. This indicates that $\mathcal{C}_*^i\sim_* \mathcal{C}_*^{i'}$, which is a contradiction.

	Item (2) directly follows from the construction of $\eta_{\ddagger}$.

	Now we prove Item (3). For $1\le i \le m-1$, by the construction of $\eta_{\ddagger}$, $\eta_{\ddagger}$ does not intersect $\mathcal{I}(\mathcal{C}_*^{i},\mathcal{C}_*^{i+1})$. Therefore, by Item (ii) of Remark \ref{remark_interlock}, there are at least $\kappa(\mathscr{G})$ edges that have both end points in $\partial_{\infty,*}^{\mathrm{o}}\mathcal{C}_*^i\cap \partial_{\infty,*}^{\mathrm{o}}\mathcal{C}_*^{i+1}$ and intersect $\mathcal{I}(\mathcal{C}_*^{i},\mathcal{C}_*^{i+1})$. Thus, these $\kappa(\mathscr{G})$ edges are not traversed by $\eta_{\ddagger}$. We denote by $\mathcal{Z}_i$ the collection of all oriented edges covered by these edges. Note that $|\mathcal{Z}_i|\ge 2\kappa(\mathscr{G})$. Moreover, by Item (iii) of Remark \ref{remark_interlock}, we know that $\mathcal{Z}_i$ for $1\le i\le m-1$ are disjoint from each other, and hence $|\cup_{i=1}^{m-1}\mathcal{Z}_i|\ge 2\kappa(\mathscr{G})(m-1)$. Thus, since every oriented edge in $\cup_{i=1}^{m-1}\mathcal{Z}_i$ is contained in $\cup_{i=1}^{m} \mathbf{R}^{\mathrm{e}}(\ell_{\mathcal{C}_*^i}^{z^i})\setminus \mathbf{R}^{\mathrm{e}}(\eta_{\ddagger})$, we conclude Item (3).

	We establish Item (4) using proof by contradiction. Since $w$ is not free, one has $|N(z)\setminus A|\le 2$. Therefore, by the $*$-connectivity of $N(w)$, we know that $N(w)\cap  A$ consists of at most two $*$-clusters. Thus, if we assume that the property stated in Item (4) fails, then $N(w)\cap A$ consists of exactly two $*$-clusters that intersect $D$ and $A\setminus D$ respectively, which implies $|N(w)\setminus A|=2$. As a result, $D$ and $A\setminus D$ are interlocked. However, this causes a contradiction to the maximality of the interlocked cluster $D$. Now we establish all properties of $\eta_{\ddagger}$ in this lemma.
\end{proof}

To establish a lower bound for the number of free vertices in $\mathbf{R}^{\mathrm{e}}(\eta_\ddagger)$, we define $\mathcal{C}_\ddagger:=\{[\mathbf{R}^{\mathrm{v}}(\eta_{\ddagger})]_{\infty}^c\}^c$ as the set obtained from the range of $\eta_{\ddagger}$ by filling all holes. It follows from the definition that $\partial_{\infty}^{\mathrm{i}} \mathcal{C}_\ddagger\subset \mathbf{R}^{\mathrm{v}}(\eta_{\ddagger})$. Noting that $\mathcal{C}_\ddagger$ is connected (by the connectivity of $\mathbf{R}^{\mathrm{v}}(\eta_{\ddagger})$) and applying Item (1) of Lemma \ref{lemma2.2}, we know that $\partial_{\infty}^{\mathrm{i}} \mathcal{C}_\ddagger$ is $*$-connected. Thus, since $\mathscr{G}$ is two-dimensional, we have 
\begin{equation}\label{lower_bound_C_ddagger}
	|\partial_{\infty}^{\mathrm{i}} \mathcal{C}_\ddagger|   \gtrsim\mathrm{diam}(\mathcal{C}_\ddagger)  \gtrsim    \sqrt{|\mathcal{C}_\ddagger|} \ge \sqrt{|\cup_{i=1}^{m}\mathcal{C}_*^i|},
\end{equation}
where in the last inequality we used the fact $\mathcal{C}_*^i \subset \mathcal{C}_\ddagger$ for all $1\le i\le m$ (since $\eta_{\ddagger}$ surrounds every $\mathcal{C}_*^i$). The subsequent lemma shows that $\partial_{\infty}^{\mathrm{i}}\mathcal{C}_\ddagger$ indeed contains numerous free vertices when $\cup_{i=1}^{m}\mathcal{C}_*^i$ occupies a large part of $D$.

\begin{lemma}\label{lemma_source_free}
	For any $\vec{e}_{w,v}\in \mathbf{R}^{\mathrm{e}}(\eta_\ddagger)$ with $w\in \partial_{\infty}^{\mathrm{i}} \mathcal{C}_\ddagger$, one of the following holds:
	\begin{enumerate}[(i)]
		\item  $w$ or $v$ is a free vertex in $A$;
		
		\item    there exists $z\in D\setminus \cup_{i=1}^{m}\mathcal{C}_*^i$ such that $z\sim w$. 
		
	\end{enumerate}
\end{lemma}
\begin{proof}
	we prove this lemma by contradiction. Precisely, we assume that both of $w$ and $v$ are not free and that $D\setminus \cup_{i=1}^{m}\mathcal{C}_*^i$ is not adjacent to $w$. Then we show that these assumptions will lead to a contradiction with the maximality of the $*$-clusters in $\{\mathcal{C}_*^{i}\}_{i=1}^{m}$ or the interlocked cluster $D$.

Since $w\in \partial_{\infty}^{\mathrm{i}}\mathcal{C}_\ddagger$, there exists $z\in (\mathcal{C}_\ddagger)_\infty^c\subset (\cup_{i=1}^m\mathcal{C}_*^i)^c$ with $z\sim w$. Combined with the assumption that $D\setminus \cup_{i=1}^m\mathcal{C}_*^i$ is not adjacent to $w$, it yields that such $z$ must be in $N(w)\setminus D$. Thus, by Item (4) of Lemma \ref{lemma_7.7}, we have $N(w)\cap A \subset A\setminus D$. Moreover, by $v\in N(w)\setminus A$, $w\in N(v)\setminus A$ and the assumption that both of $w$ and $v$ are not free, we have $|N(w)\setminus A|,|N(v)\setminus A|\in \{1,2\}$. In what follows, we divide the proof into separate cases based on the cardinalities of $N(w)\setminus A$ and $N(v)\setminus A$.

	\textbf{Case (a)}: $|N(w)\setminus A|=1$. By Condition (2) for surrounding loops in Definition \ref{def_sl}, there exists a face $\mathcal{S}$ such that $\vec{e}_{w,v}\in \vec{\mathbf{e}}(\mathcal{S})$ and $ \mathbf{v}(\mathcal{S})\cap (\cup_{i=1}^m\mathcal{C}_*^i) \neq \emptyset$. However, $\mathbf{v}(\mathcal{S})$ contains two vertices in $N(w)$ and hence, it must intersect $N(w)\cap A$ (since $|N(w)\setminus A|=1$). As a result, $N(w)\cap A$ ($\subset A\setminus D$) and $\cup_{i=1}^m\mathcal{C}_*^i$ ($\subset D$) both intersect the same face $\mathcal{S}$ and thus are $*$-connected to each other, which is contradictory to the maximality of the $*$-clusters in $\{\mathcal{C}_*^{i}\}_{i=1}^{m}$.

	\textbf{Case (b)}: $|N(w)\setminus A|=2$ and $|N(v)\setminus A|=1$. In fact, it follows from $|N(v)\setminus A|=1$ that $\vec{e}_{v,w}\in \mathbf{R}^{\mathrm{e}}(\eta_{\ddagger})$ (otherwise, since $\eta_{\ddagger}$ is an edge circuit, there must be some vertex $z\notin \{w,v\}$ such that $\vec{e}_{v,z}\in \mathbf{R}^{\mathrm{e}}(\eta_{\ddagger})$, and thus $|N(v)\setminus A|\ge |\{z,w\}| =2$). We denote by $\mathcal{S}_1$ (resp. $\mathcal{S}_2$) the unique face with $\vec{e}_{w,v}\in \vec{\mathbf{e}}(\mathcal{S}_1)$ (resp. $\vec{e}_{v,w}\in \vec{\mathbf{e}}(\mathcal{S}_2)$). Let $u_1$ (resp. $u_2$) be the unique vertex in $\mathbf{v}(\mathcal{S}_1)$ (resp. $\mathbf{v}(\mathcal{S}_2)$) with $\vec{e}_{u_1,w}\in \vec{\mathbf{e}}(\mathcal{S}_1)$ (resp. $\vec{e}_{w,u_2}\in \vec{\mathbf{e}}(\mathcal{S}_2)$). Note that $u_1$, $u_2$ and $v$ are distinct to each other (by considering the coordinates in the direction perpendicular to $\vec{e}_{w,v}$) and thus $u_{j_\diamond}\in N(w)\cap A$ holds for some $j_\diamond\in \{1,2\}$ (otherwise, $|N(w)\setminus A|\ge |\{v,u_1,u_2\}|=3$). Recalling that $N(w)\cap A \subset A\setminus D$, we have $u_{j_\diamond}\in A\setminus D$. Combined with $\mathbf{v}(\mathcal{S}_{j_{\diamond}})\cap (\cup_{i=1}^{m}\mathcal{C}_*^i)\neq \emptyset$ (which follows from Condition (2) in Definition \ref{def_sl}), it yields that $A\setminus D$ is $*$-connected to $\cup_{i=1}^m\mathcal{C}_*^i$. This causes the same contradiction as in Case (a).

	\textbf{Case (c)}: $|N(w)\setminus A|=2$ and $|N(v)\setminus A|=2$. By Item (4) of Lemma \ref{lemma_7.7}, we have either $N(v)\cap  A\subset A\setminus D$ or $N(v)\cap A\subset D$. When $N(v)\cap  A\subset A\setminus D$, similar to Cases (a) and (b), using Condition (2) in Definition \ref{def_sl} we can show that $A\setminus D$ is connected $\cup_{i=1}^m\mathcal{C}_*^i$, thereby causing a contradiction. When $N(v)\cap  A\subset D$, we have $|N(v)\setminus D|=|N(v)\setminus A|=2$ and $N(v)\cap (A\setminus D)=\emptyset$. Combined with $|N(w)\setminus (A\setminus D)|=|N(w)\setminus A|=2$ and $N(w)\cap D=\emptyset$ (since $N(w)\cap A\subset A\setminus D$), it yields $D\leftrightsquigarrow  A\setminus D$ (recall the definition of interlocked sets above Remark \ref{remark_interlock}). However, this is contradictory to the maximality of the interlocked cluster $D$.

	To sum up, we complete the proof of this lemma.
\end{proof}

It directly follows from Lemma \ref{lemma_source_free} that 
\begin{equation*}\label{new_add_C_ddagger}
	\begin{split}
		|\partial_{\infty}^{\mathrm{i}} \mathcal{C}_\ddagger|   \le& 2|\{\vec{e}_{w,v}\in \mathbf{R}^{\mathrm{e}}(\eta_\ddagger): w\ \text{is free in}\ A\}|+  |\{(w,z):z\in D\setminus \cup_{i=1}^{m}\mathcal{C}_*^i, w\sim z\}|\\
		\le &2|\{\vec{e}_{w,v}\in \mathbf{R}^{\mathrm{e}}(\eta_\ddagger): w\ \text{is free in}\ A\}|+ \mathrm{deg}(\mathscr{G})   |D\setminus \cup_{i=1}^{m}\mathcal{C}_*^i|.
	\end{split}
\end{equation*}
Combined with (\ref{lower_bound_C_ddagger}), it implies that 
\begin{equation}\label{lower_free}
|\{\vec{e}_{w,v}\in \mathbf{R}^{\mathrm{e}}(\eta_\ddagger): w\ \text{is free in}\ A\}| \ge c\sqrt{|\cup_{i=1}^{m}\mathcal{C}_*^i|}- \tfrac{1}{2}\mathrm{deg}(\mathscr{G}) |D\setminus \cup_{i=1}^{m}\mathcal{C}_*^i|. 
\end{equation}

Now we are ready to bound the quantity $\mathbf{X}_A$ in (\ref{statistic_X}) for the range of $\eta_{\ddagger}$.

\begin{lemma}\label{lemma_7.8}
	There exists $\cl\label{const_eta_ddagger}(\mathscr{G})>0$ such that 
	\begin{equation}\label{add_7.11}
		\mathbf{X}_A(\mathbf{R}^{\mathrm{e}}(\eta_\ddagger))\le 2\big[\nu(\mathscr{G})|D|+\kappa(\mathscr{G})\big]- \cref{const_eta_ddagger}\sqrt{|D|}.  
	\end{equation}
\end{lemma}
\begin{proof}
	By Items (2) and (3) of Lemma \ref{lemma_7.7} and Condition (3) for surrounding loops (see Definition \ref{def_sl}), we have   
	\begin{equation}\label{7.11}
		\begin{split}	
			|\mathbf{R}^{\mathrm{e}}(\eta_\ddagger)|\le & \sum\nolimits_{1\le i\le m}|\mathbf{R}^{\mathrm{e}}(\ell_{\mathcal{C}_*^i}^{z^i})| - |\cup_{i=1}^{m} \mathbf{R}^{\mathrm{e}}(\ell_{\mathcal{C}_*^i}^{z^i})\setminus \mathbf{R}^{\mathrm{e}}(\eta_{\ddagger})|   \\
			\le &\sum\nolimits_{1\le i\le m}  2\big[\nu(\mathscr{G})|\mathcal{C}_*^i|+ \kappa(\mathscr{G})\big] -2\kappa(\mathscr{G})(m-1)\\
			=& 2\big[\nu(\mathscr{G})|\cup_{i=1}^{m}\mathcal{C}_*^i| + \kappa(\mathscr{G})\big]. 
		\end{split}
	\end{equation} 
	Note that $\delta(\mathscr{G}):=2\nu(\mathscr{G})-\tfrac{1}{2}c_\ddagger(\mathscr{G})\mathrm{deg}(\mathscr{G})>0$ for all $\mathscr{G}\in \{\mathbb{Z}^2,\mathscr{T},\mathscr{H}\}$ (recalling $c_\ddagger$ in (\ref{statistic_X})). We denote $K:=|\cup_{i=1}^{m}\mathcal{C}_*^i|$. By (\ref{lower_free}) and (\ref{7.11}), we obtain  
	\begin{equation}\label{final_7.16}
		\begin{split}
			\mathbf{X}_A(\mathbf{R}^{\mathrm{e}}(\eta_\ddagger))\le & 2\big[\nu(\mathscr{G})|D| + \kappa(\mathscr{G})\big] - \delta(\mathscr{G})\cdot (|D|- K)-c\sqrt{K}.
		\end{split}
	\end{equation}
	In fact, we have the following inequality: 
	\begin{equation}\label{final_7.17}
		\delta(\mathscr{G})\cdot (|D|- K)+c\sqrt{K}\ge c'\sqrt{|D|}.
	\end{equation}
	To see this, if $K\ge \frac{1}{2}|D|$, one has $c\sqrt{K}\ge 2^{-1/2}c\sqrt{|D|}$; otherwise (i.e. $|D|- K\ge  \frac{1}{2}|D|$), one has $\delta(\mathscr{G})\cdot (|D|- K)\ge \frac{1}{2}\delta(\mathscr{G}) |D|\ge \frac{1}{2}\delta(\mathscr{G})\sqrt{|D|}$. Combining (\ref{final_7.16}) and (\ref{final_7.17}), we conclude this lemma.
\end{proof}

Based on Lemma \ref{lemma_7.8}, it becomes straightforward to prove Proposition \ref{lemma_length_bypass}.

\begin{proof}[Proof of Proposition \ref{lemma_length_bypass}]
	Recalling that $\eta_{\ddagger}=\eta_{x\to y} \circ \eta_{y\to x}$, we have 
	\begin{equation*}
		\begin{split}
			&\big|\big\{ \vec{e}_{w,v}\in \mathbf{R}^{\mathrm{e}}(\eta_\ddagger): w\ \text{is free in}\ A  \big\}\big|\\
			\le &\big|\big\{ \vec{e}_{w,v}\in \mathbf{R}^{\mathrm{e}}(\eta_{x\to y}): w\ \text{is free in}\ A  \big\}\big|+ \big|\big\{ \vec{e}_{w,v}\in \mathbf{R}^{\mathrm{e}}(\eta_{y\to x}): w\ \text{is free in}\ A  \big\}\big|. 
		\end{split}	 
	\end{equation*}
	Therefore, since $|\mathbf{R}^{\mathrm{e}}(\eta_\ddagger)|= |\mathbf{R}^{\mathrm{e}}(\eta_{x\to y})|+|\mathbf{R}^{\mathrm{e}}(\eta_{y\to x})|$ (by Item (1) of Lemma \ref{lemma_7.7}), 
	\begin{equation}\label{7.15}
		\mathbf{X}_A(\mathbf{R}^{\mathrm{e}}(\eta_{x\to y}))+\mathbf{X}_A(\mathbf{R}^{\mathrm{e}}(\eta_{y\to x}))	\le \mathbf{X}_A(\mathbf{R}^{\mathrm{e}}(\eta_\ddagger)). 
	\end{equation}
	Note that $\mathbf{X}_A(\mathbf{R}^{\mathrm{e}}(\cev{\eta}_{y\to x}))\le \mathbf{X}_A(\mathbf{R}^{\mathrm{e}}(\eta_{y\to x}))+ c_\ddagger$ (where $\cev{\eta}_{y\to x}$ denotes the reversed path of $\eta_{y\to x}$). Combined with (\ref{7.15}), Lemma \ref{lemma_7.8} and $\kappa(\mathscr{G})+\tfrac{1}{2}c_\ddagger(\mathscr{G}) \le 4$ for $\mathscr{G}\in \{\mathbb{Z}^2,\mathscr{T},\mathscr{H}\}$, it implies that there exists $\eta\in \{\eta_{x\to y},\cev{\eta}_{y\to x}\}$ satisfying
	\begin{equation}\label{7.16}
	\begin{split}
		\mathbf{X}_A(\mathbf{R}^{\mathrm{e}}(\eta)) \le \nu(\mathscr{G})|D|- \tfrac{1}{2}\cref{const_eta_ddagger}\sqrt{|D|}+4. 
	\end{split}
	\end{equation}
	Moreover, $\eta$ connects $x$ and $y$ within $A^c$ since both $\eta_{x\to y}$ and $\cev{\eta}_{y\to x}$ do. In conclusion, $\eta$ satisfies all conditions in Proposition \ref{lemma_length_bypass}. 
\end{proof}

\section{Quantitative removal arguments}\label{section_qra}

In this section, we establish the remaining two technical lemmas, namely, Lemmas \ref{lemma_remove_distant} and \ref{lemma_remove_interlock}. As mentioned in Section \ref{subsection_outline}, these two lemmas are both quantitative versions of the removal argument presented in \cite{psi}.


\subsection{Removal of a distant subset}\label{section_qra_distant}

The aim of this subsection is to establish Lemma \ref{lemma_remove_distant}. We first cite two lemmas from \cite{psi}, which together provide a general approach to estimate the increase to the harmonic measure after removing a subset. Although in \cite{psi} these two lemmas are stated for $\mathbb{Z}^d$ ($d\ge 2$), their proofs do not rely on the graph structure unique to $\mathbb{Z}^d$ and can be easily adapted to $\mathscr{T}$ and $\mathscr{H}$. Precisely, let $A\in \mathcal{A}(\mathscr{G})$ and $D\subset A\setminus \{\bm{0}\}$. We denote $\widetilde{A}:=A\setminus D$. Arbitrarily take $F_1\subset F_2\subset \mathfs{V}$ such that $F_1\cap A=F_2\cap A=D$. For $i\in \{1,2\}$, define 
\begin{equation}\label{new_4.1}
	\widehat{F}_i:= [\partial^{\mathrm{i}}_{\infty} (A\cup F_i)]\setminus \widetilde{A}\ \ \text{and}\ \ \widecheck{F}_i:= [\partial^{\mathrm{i}}_{\bm{0}} (A\cup F_i)]\setminus \widetilde{A}.
\end{equation}
\begin{lemma}[{\cite[Lemma 3.5]{psi}}] \label{lemma_psi_3.5}
	Keep the notations above. We have 
	\begin{equation}
		\frac{\mathbb{H}_{A\setminus D}(\bm{0})}{\mathbb{H}_{A}(\bm{0})} \le \max_{ v_1\in \widehat{F}_1, v_2\in \widecheck{F}_2}  \frac{G_{A\setminus D}(v_1,v_2)}{G_{A}(v_1,v_2)}. 
	\end{equation}
\end{lemma}

For $v_2\in \widecheck{F}_2$, let $\mathring{v}_1=\mathring{v}_1(A,D,v_2)$ be the vertex in $\widehat{F}_1\cup \widecheck{F}_1$ maximizing $G_{A\setminus D} (\cdot, v_2)$ (we break the tie in a predetermined manner).

\begin{lemma}[{\cite[Lemma 3.6]{psi}}]\label{lemma_psi_3.6}
	Keep the notations above. We have 
	\begin{equation}\label{ineq_psi3.6_1}
		G_{A\setminus D}(v_1,v_2) \le\left[  \mathbb{P}_{\mathring{v}_1}\left(\tau_{A\setminus D} <\tau_D \right)\right]^{-1}	G_{A}(\mathring{v}_1,v_2),  
	\end{equation}
	\begin{equation}\label{ineq_psi3.6_2}
		G_{A\setminus D}(v_1,v_2) \le \left[  \mathbb{P}_{\mathring{v}_1}\left(\tau_{A\setminus D} <\tau_D \right) \mathbb{P}_{v_1}\left(\tau_{v_2}<\tau_A \right) \right]^{-1}G_A(v_1,v_2). 
	\end{equation}
\end{lemma}

We present a quantitative version of the discrete Harnack's inequality as follows:

\begin{lemma}\label{lemma_green_2}
	For any $\epsilon>0$, there exists $\Cl\label{const_green2}(\mathscr{G},\epsilon)>0$ such that for any $K\ge \Cref{const_green2}$, $r\ge 1$, $A\subset B(r)\cup [B(Kr)]^c$, $v_1,v_1'\in \partial^{\mathrm{i}}B\big(\frac{Kr}{\ln(K)}\big)$ and $v_2\in \partial^{\mathrm{i}}B\big( \frac{2Kr}{\ln(K)}\big)$, 
	\begin{equation}
		G_A(v_1',v_2)\le (1+\epsilon)G_A(v_1,v_2).
	\end{equation}
\end{lemma}
\begin{proof}
	Suppose that $K>0$ is sufficiently large. We denote $\ln^{(2)}(K):=\ln\ln(K)$ and $\ln^{(3)}(K):=\ln\ln\ln(K)$. Let $R_0=Kr$, $R_1=\frac{Kr}{\ln(K)}$, $R_2=\frac{Kr}{\ln(K)\ln^{(2)}(K)}$ and $R_3=\frac{Kr}{\ln(K)[\ln^{(2)}(K)]^2}$. For any $v_1,v_1'\in \partial^{\mathrm{i}}B(R_1)$, since $A\subset B(r)\cup [B(R_0)]^c$, one has 
	\begin{equation*}
		\begin{split}
			\mathbb{P}_{v_1}\big(\tau_{B_{v_1'}(R_3)}< \tau_{A} \big)\ge& \mathbb{P}_{v_1}\big(\tau_{B_{v_1'}(R_3)}< \tau_{B(r)\cup \partial^{\mathrm{i}}B(R_0) } \big) \\
			\ge & 1 -  \mathbb{P}_{v_1} \big(\tau_{B(r)} <\tau_{\partial^{\mathrm{i}}B(R_0)} \big)- \mathbb{P}_{v_1} \big(\tau_{B_{v_1'}(R_3)} > \tau_{\partial^{\mathrm{i}}B(R_0)}\big). 
		\end{split} 
	\end{equation*}
	For the first probability on the right-hand side, by Lemma \ref{lemma_hit}, we have 
	\begin{equation*}
		\mathbb{P}_{v_1} \big(\tau_{B(r)} <\tau_{\partial^{\mathrm{i}}B(R_0)} \big)  \lesssim   \frac{\ln^{(2)}(K)}{\ln(K)}. 
	\end{equation*}
	Moreover, since $ B_{v_1'}(R_0-R_1)\subset B(R_0)$ and $|v_1-v_1'|\le 2R_1$, by Lemma \ref{lemma_hit},
	\begin{equation*}
		\begin{split}
			\mathbb{P}_{v_1} \big(\tau_{B_{v_1'}(R_3)} > \tau_{\partial^{\mathrm{i}}B(R_0)} \big)\le& \mathbb{P}_{v_1}\big(\tau_{B_{v_1'}(R_3)} > \tau_{\partial^{\mathrm{i}}B_{v_1'}(R_0-R_1)} \big)
			\lesssim  \frac{\ln^{(3)}(K)}{\ln^{(2)}(K)}. 
		\end{split}
	\end{equation*}
	Combining these three inequalities, we get 
	\begin{equation}\label{4.4}
		\mathbb{P}_{v_1}\big(\tau_{B_{v_1'}(R_3)}< \tau_{A} \big) \ge 1-C'(\mathscr{G})\bigg[\frac{\ln^{(2)}(K)}{\ln(K)}+\frac{\ln^{(3)}(K)}{\ln^{(2)}(K)} \bigg],
	\end{equation}
	which together with the strong Markov property implies that
	\begin{equation}\label{4.5}
		\begin{split}
			G_A(v_1,v_2) \ge & \mathbb{P}_{v_1}\big(\tau_{B_{v_1'}(R_3)}< \tau_{A} \big)\min_{w_1\in B_{v_1'(R_3)}} G_{A}(w_1,v_2)\\ 
			\ge & \bigg( 1-C'\bigg[\frac{\ln^{(2)}(K)}{\ln(K)}+\frac{\ln^{(3)}(K)}{\ln^{(2)}(K)} \bigg]\bigg)\min_{w_1\in B_{v_1'(R_3)}} G_{A}(w_1,v_2).
		\end{split}
	\end{equation}

	For any $w_1\in B_{v_1'(R_3)}$, by the strong Markov property and $A\subset [B_{v_1'}(R_2)]^c$, 
	\begin{equation}\label{4.6}
		G_{A}(w_1,v_2) = \sum_{w_2\in \partial^{\mathrm{i}} B_{v_1'}(R_2) } \mathbb{P}_{w_1}\big(\tau_{\partial^{\mathrm{i}} B_{v_1'}(R_2)} =\tau_{w_2} \big) G_{A}(w_2,v_2).  
	\end{equation}
	By Lemma \ref{lemma_last-exit} (taking $A_1=\partial^{\mathrm{i}}B_{v_1'}(R_2)$, $A_2=B_{v_1'}(R_3)\cup \partial^{\mathrm{i}}B_{v_1'}(R_2)$, $x=w_1$ and $y=w_2$), we know that $\mathbb{P}_{w_1}\big(\tau_{\partial^{\mathrm{i}} B_{v_1'}(R_2)} =\tau_{w_2} \big)$ equals to 
	\begin{equation}\label{4.7}
		\begin{split}
			\sum_{w_3\in \partial^{\mathrm{i}}B_{v_1'}(2R_3) } G_{\partial^{\mathrm{i}} B_{v_1'}(R_2)}(w_1,w_3) \mathbb{P}_{w_3}\big(\tau^+_{\partial^{\mathrm{i}} B_{v_1'}(R_2)\cup \partial^{\mathrm{i}} B_{v_1'}(2R_3)} =\tau_{w_2} \big). 
		\end{split}
	\end{equation} 
	In addition, for each fixed $w_3\in \partial^{\mathrm{i}}B_{v_1'}(2R_3) $ and any $w_1\in  B_{v_1'(R_3)}$, by $B_{v_1'}(R_2)\subset B_{w_3}(R_2+R_3)$, $|w_1-w_3|\ge R_3$ and Lemma \ref{lemma_green}, we have 
	\begin{equation}\label{4.8}
		G_{\partial^{\mathrm{i}} B_{v_1'}(R_2)}(w_1,w_3)\le G_{\partial^{\mathrm{i}} B_{w_3}(R_2+R_3)}(w_1,w_3)  \le \Cref{asym_green}(\mathscr{G}) \ln(\tfrac{R_2+R_3}{R_3}) +O(R_3^{-1}). 
	\end{equation}
	Similarly, noting that $B_{w_3}(R_2-R_3)\subset B_{v_1'}(R_2)$ and $|w_1-w_3|\le 3R_3$, we also have
	\begin{equation}\label{4.9}
		G_{\partial^{\mathrm{i}} B_{v_1'}(R_2)}(w_1,w_3) \ge \Cref{asym_green}(\mathscr{G}) \ln(\tfrac{R_2-R_3}{3R_3}) +O(R_3^{-1}).  
	\end{equation}
	By (\ref{4.7}), (\ref{4.8}) and (\ref{4.9}), we obtain that for any $w_2\in \partial^{\mathrm{i}} B_{v_1'}(R_2)$, 
	\begin{equation*}
		1\le 	\frac{\max_{w_1\in  B_{v_1'(R_3)}}\mathbb{P}_{w_1}\big(\tau_{\partial^{\mathrm{i}} B_{v_1'}(R_2)} =\tau_{w_2} \big)}{\min_{w_1\in  B_{v_1'(R_3)}}\mathbb{P}_{w_1}\big(\tau_{\partial^{\mathrm{i}} B_{v_1'}(R_2)} =\tau_{w_2} \big)}	 \le  1+C''(\mathscr{G})[\ln^{(3)}(K)]^{-1}. 
	\end{equation*}
	Combined with (\ref{4.6}), it yields that for any $v_2\in \partial^{\mathrm{i}}B(2R_1)$, 
	\begin{equation*}
		1\le 	\frac{\max_{w_1\in  B_{v_1'(R_3)}}G_{A}(w_1,v_2)}{\min_{w_1\in  B_{v_1'(R_3)}}G_{A}(w_1,v_2)}	 \le  1+C''[\ln^{(3)}(K)]^{-1}. 
	\end{equation*}
	Especially, since $v_1'\in B_{v_1'}(R_3)$, one has 
	\begin{equation}\label{4.10}
		G_{A}(v_1',v_2) \le 	\big\{1+C''[\ln^{(3)}(K)]^{-1}\big\}\min_{w_1\in B_{v_1'(R_3)}} G_{A}(w_1,v_2).
	\end{equation}

	Note that $\tfrac{\ln^{(2)}(K)}{\ln(K)}$, $\tfrac{\ln^{(3)}(K)}{\ln^{(2)}(K)}$ and $[\ln^{(3)}(K)]^{-1}$ all converge to $0$ as $K\to \infty$. Thus, for any $\epsilon>0$, there exists $\Cref{const_green2}(\mathscr{G},\epsilon)>0$ such that for all $K\ge \Cref{const_green2}$, 
	\begin{equation}\label{4.11}
		\bigg( 1-C'\bigg[\frac{\ln^{(2)}(K)}{\ln(K)}+\frac{\ln^{(3)}(K)}{\ln^{(2)}(K)} \bigg]\bigg)^{-1} \big\{1+C''[\ln^{(3)}(K)]^{-1}\big\} \le 1+\epsilon.
	\end{equation}
	Combining (\ref{4.5}), (\ref{4.10}) and (\ref{4.11}), we conclude this lemma.
\end{proof}

Now we are ready to prove Lemma \ref{lemma_remove_distant}.

\begin{proof}[Proof of Lemma \ref{lemma_remove_distant}]
	Let $L\ge \Cref{remove_distant}$, where $\Cref{remove_distant}\ge 10$ is a large constant that will be determined later. We arbitrarily take $x_{\diamond}\in D$ and denote $l:=\Cref{ball_1}[\mathrm{diam}(D)\vee 1]$. Note that $D\subset \mathbf{B}_{x_{\diamond}}(\mathrm{diam}(D)) \subset B_{x_{\diamond}}(l)$ by (\ref{ineq_ball}). Let  $J_0:=\mathbf{d}(x_{\diamond},A\setminus D)$. Since $\mathbf{d}(D,A\setminus D)\ge L[\mathrm{diam}(D)\vee 1]$, one has
	\begin{equation}\label{ineq_K513}
		K:= l^{-1}J_0 \ge l^{-1}\mathbf{d}(D,A\setminus D) \ge  \frac{L[\mathrm{diam}(D)\vee 1]}{\Cref{ball_1}[\mathrm{diam}(D)\vee 1]} = \Cref{ball_1}^{-1}L.  
	\end{equation} 
	 We also denote $J_1=J_0\ln^{-1}(K)$, $J_2=2J_1$ and $J_3=J_0\ln(K)$. For each $i\in \{1,2\}$, we take $F_i=B_{x_{\diamond}}(J_i)$. Since $F_i\cap (A\setminus D)=\emptyset$, one has $\widehat{F}_i=\widecheck{F}_i=\partial^{\mathrm{i}} B_{x_{\diamond}}(J_i)$ (recalling $\widehat{F}_i$ and $\widecheck{F}_i$ in (\ref{new_4.1})). Thus, by Lemma \ref{lemma_psi_3.5} and (\ref{ineq_psi3.6_1}), we have 
	\begin{equation}\label{4.13}
		\frac{\mathbb{H}_{A\setminus D}(\bm{0})}{\mathbb{H}_{A}(\bm{0})} \le \max_{v_1,v_1'\in  \partial^{\mathrm{i}}B_{x_{\diamond}}(J_1), v_2\in  \partial^{\mathrm{i}}B_{x_{\diamond}}(J_2)}  \left[  \mathbb{P}_{v_1'}\left(\tau_{A\setminus D} <\tau_D \right)\right]^{-1}  \frac{G_{A}(v_1',v_2)}{G_{A}(v_1,v_2)}.	  
	\end{equation}
	For any $\epsilon\in (0,1)$, we require that $\Cref{remove_distant}\ge  \Cref{ball_1}\Cref{const_green2}(\mathscr{G}, \frac{1}{8}\epsilon)$. Thus, by Lemma \ref{lemma_green_2} and the fact that $K\ge \Cref{ball_1}^{-1}L \ge \Cref{ball_1}^{-1}\Cref{remove_distant}  \ge  \Cref{const_green2}(\mathscr{G}, \frac{1}{8}\epsilon) $ (recalling (\ref{ineq_K513})), we have 
	\begin{equation*}
		\max_{v_1,v_1'\in \partial^{\mathrm{i}} B_{x_{\diamond}}(J_1), v_2\in \partial^{\mathrm{i}} B_{x_{\diamond}}(J_2)} \frac{G_{A}(v_1',v_2)}{G_{A}(v_1,v_2)} \le 1+ \frac{1}{8}\epsilon. 
	\end{equation*}
	Combined with (\ref{4.13}), it implies that 
	\begin{equation}\label{4.20}
			\frac{\mathbb{H}_{A\setminus D}(\bm{0})}{\mathbb{H}_{A}(\bm{0})} \le \Big(1+ \frac{1}{8}\epsilon\Big) \Big[\min_{v_1\in  \partial^{\mathrm{i}}B_{x_{\diamond}}(J_1)}   \mathbb{P}_{v_1'}\left(\tau_{A\setminus D} <\tau_D \right) \Big]^{-1} .
	\end{equation}

	In what follows, we prove a lower bound for the minimum on the right-hand side of (\ref{4.20}). Since $J_0=\mathbf{d}(x_{\diamond},A\setminus D)$, there exists $z_\dagger\in A\setminus D$ such that $\mathbf{d}(z_\dagger,x_{\diamond})=J_0$. By $z_\dagger\in A\setminus D$ and $D\subset B_{x_{\diamond}}(l)$, we have: for any $v_1'\in \partial^{\mathrm{i}} B_{x_{\diamond}}(J_1)$,
	\begin{equation}\label{4.14}
		\mathbb{P}_{v_1'}\left(\tau_{A\setminus D} <\tau_D \right) \ge \mathbb{P}_{v_1'}\big(\tau_{z_{\dagger}} <\tau_{B_{x_{\diamond}}(l)} \big) 
	\end{equation}
	Intuitively, for the random walk which starts from $v_1'$ and hits $z_\dagger$ before $B_{x_{\diamond}}(l)$, it may first escape to a faraway place (say $\partial^{\mathrm{i}} B(J_3)$), then hit $B_{z_{\dagger}}(l)$ before $B_{x_{\diamond}}(l)$ (by symmetry, the probability of this step is approximately $\tfrac{1}{2}$), and finally reach $z_{\dagger}$. To be precise, by the strong Markov property, $\mathbb{P}_{v_1'}\big(\tau_{z_{\dagger}} <\tau_{B_{x_{\diamond}}(l)} \big)$ is bounded from below by 
	\begin{equation}\label{4.15}
		\begin{split}
			&\mathbb{P}_{v_1'} \big(\tau_{\partial^{\mathrm{i}} B(J_3)}< \tau_{z_{\dagger}} \land \tau_{B_{x_{\diamond}}(l)} \big) 	\min_{w_1\in \partial^{\mathrm{i}} B(J_3)}   \mathbb{P}_{w_1} \big(\tau_{B_{z_{\dagger}}(l)} < \tau_{ B_{x_{\diamond}}(l)} \big) \\
			&\cdot \min_{w_2\in \partial^{\mathrm{i}} B_{z_{\dagger}}(l)  } \mathbb{P}_{w_2}\big(\tau_{z_{\dagger}} < \tau_{B_{x_{\diamond}}(l)} \big).
		\end{split}
	\end{equation}
	For the first probability, with the similar arguments as proving (\ref{4.4}), we have 
	\begin{equation}
		\mathbb{P}_{v_1'} \big(\tau_{\partial^{\mathrm{i}} B(J_3)}< \tau_{z_{\dagger}} \land \tau_{B_{x_{\diamond}}(l)} \big) \ge 1- O\bigg(\frac{\ln^{(2)}(K)}{\ln(K)}\bigg). 
	\end{equation}
	For the second one, by Lemma \ref{lemma_2srw_hm}, one has: for any $w_1\in \partial^{\mathrm{i}} B(J_3)$,
	\begin{equation}
		\begin{split}
			\mathbb{P}_{w_1} \big(\tau_{B_{z_{\dagger}}(l)} < \tau_{ B_{x_{\diamond}}(l)} \big)  =& \sum_{w_2\in \partial^{\mathrm{i}}B_{x_{\diamond}}(l)  } \mathbb{P}_{w_1} \big(\tau_{B_{z_{\dagger}}(l)\cup B_{x_{\diamond}}(l) } = \tau_{ w_2} \big)  \\
			=&\Big[ 1- O\big([\ln(K)]^{-1}\big)  \Big] \sum_{w_2\in \partial^{\mathrm{i}}B_{x_{\diamond}}(l)  } \mathbb{H}_{B_{z_{\dagger}}(l)\cup B_{x_{\diamond}}(l)} (w_2)\\
			=&\frac{1}{2} \Big[ 1- O\big([\ln(K)]^{-1}\big)  \Big], 
		\end{split}
	\end{equation}
	where the last equality follows from the symmetry of $\mathscr{G}$. For the third probability in (\ref{4.15}), by $B_{x_{\diamond}}(l)\subset [B_{z_\dagger}(\frac{1}{2}J_0)]^c$ and Lemma \ref{lemma_hit}, we have: for any $w_2\in \partial^{\mathrm{i}} B_{z_{\dagger}}(l)$,
	\begin{equation}\label{4.18}
		\mathbb{P}_{w_2}\big(\tau_{z_{\dagger}} < \tau_{B_{x_{\diamond}}(l)} \big) \ge \mathbb{P}_{w_2}\big(\tau_{z_{\dagger}} < \tau_{B_{z_\dagger}(\frac{1}{2}J_0)}  \big)\ge  \frac{\ln(K)- O(1) }{\ln(\frac{1}{2}J_0)}.
	\end{equation}
	Recall that $J_0=Kl= K\Cref{ball_1}[\mathrm{diam}(D)\vee 1]$. Thus, by (\ref{4.14})-(\ref{4.18}), we obtain 
	\begin{equation*}
		\begin{split}
			\min_{v_1'\in  \partial^{\mathrm{i}} B_{x_{\diamond}}(J_1) }	\mathbb{P}_{v_1'}\left(\tau_{A\setminus D} <\tau_D \right) \ge & \bigg[\frac{1}{2}- O\bigg(  \frac{\ln^{(2)}(K)}{\ln(K)}\bigg)\bigg]\cdot \frac{\ln(K\Cref{ball_1})}{\ln(K\Cref{ball_1})+ \ln(\mathrm{diam}(D)\vee 1)}.
		\end{split}	
	\end{equation*}	
	Note that $\frac{\ln^{(2)}(K)}{\ln(K)}$ converges to $0$ as $K\to \infty$, and that $\frac{\ln(K\Cref{ball_1})}{\ln(K\Cref{ball_1})+ \ln(\mathrm{diam}(D)\vee 1)}$ is increasing with respect to $K$. Therefore, for any $\epsilon\in (0,1)$, there exists a sufficiently large $\Cref{remove_distant}(\mathscr{G},\epsilon)$ such that for all $L\ge \Cref{remove_distant}$ (recalling that $K\ge \Cref{ball_1}^{-1}L$), 
	\begin{equation}\label{4.19}
		\min_{v_1'\in  \partial^{\mathrm{i}} B_{x_{\diamond}}(J_1) }	\mathbb{P}_{v_1'}\left(\tau_{A\setminus D} <\tau_D \right) \ge  \Big( \frac{1}{2}- \frac{1}{8}\epsilon\Big) \cdot \frac{\ln(L)}{\ln(L)+ \ln(\mathrm{diam}(D)\vee 1)}.
	\end{equation}
	By (\ref{4.20}), (\ref{4.19}) and $\frac{1+\frac{1}{8}\epsilon}{\frac{1}{2}-\frac{1}{8}\epsilon}\le 2+\epsilon$ for all $\epsilon \in (0,1)$, we conclude this lemma.
\end{proof}

\subsection{Removal of an interlocked cluster}\label{section_qra_interlocked}

In this subsection, we present the proof of Lemma \ref{lemma_remove_interlock}. Arbitrarily take $A\subset \mathfs{V}$ and $D\in \mathfrak{I}_A$ satisfying $\bm{0}\notin D$ and $|D|\ge \Cref{remove_interlock1}$, where $\Cref{remove_interlock1}(\mathscr{G})\ge \Cref{delicate_green}(\mathscr{G})$ is a large constant that will be determined later (recalling $\Cref{delicate_green}$ above Definition \ref{def_sparse}). Note that $\mathrm{diam}(D)\le 10 |D|$ since $D$ is an interlocked cluster. We arbitrarily take $x_{\diamond}\in D$.

We define the set $A^+$ according to the following cases:  
\begin{enumerate}
	\item  If $\left[ B_{x_{\diamond}}(400\Cref{ball_1}|D|^2)\setminus B_{x_{\diamond}}(200\Cref{ball_1}|D|^2)\right]  \cap  \left( A\setminus D\right) = \emptyset$ (recall $\Cref{ball_1}$ in (\ref{ineq_ball})), we arbitrarily take $y_{\diamond}\in \partial^{\mathrm{i}}B_{x_{\diamond}}(300\Cref{ball_1}|D|^2)$ and then set $A^+:= A\cup \{y_{\diamond}\}$.

	\item  Otherwise, set $A^+=A$.

\end{enumerate}
Note that $A^+\in \mathcal{A}(\mathscr{G})$ and $\mathbb{H}_{A^+}(\bm{0})\le \mathbb{H}_{A}(\bm{0})$ (since $A\subset A^+$). Moreover, for the same reason as proving (\ref{ineq_AD-ADD}), removing $y_{\diamond}$ from $A^+\setminus D$ (in Case (1)) increases the harmonic measure by a uniformly bounded factor. Thus, we have 
\begin{equation}\label{6.1}
	\frac{\mathbb{H}_{A\setminus D}(\bm{0})}{\mathbb{H}_{A}(\bm{0})}= \frac{\mathbb{H}_{A^+}(\bm{0})}{\mathbb{H}_{A}(\bm{0})}\cdot  \frac{\mathbb{H}_{A\setminus D}(\bm{0})}{\mathbb{H}_{A^+\setminus D}(\bm{0})} \cdot\frac{\mathbb{H}_{A^+\setminus D}(\bm{0})}{\mathbb{H}_{A^+}(\bm{0})} \lesssim \frac{\mathbb{H}_{A^+\setminus D}(\bm{0})}{\mathbb{H}_{A^+}(\bm{0})}.
\end{equation}

We define $F\subset \mathfs{V}$ as the cluster of the set $$\big\{ z\in \mathfs{V}:z\in \partial_{\infty,*}^{\mathrm{o}}D\ \text{or} \  \mathbb{P}_{z}\left(\tau_{A^+\setminus D} <\tau_D \right)\le 0.1\big\}$$ 
that contains $D$. Note that $F\cap A^+=D$ (since $A^+\cap \partial_{\infty,*}^{\mathrm{o}}D=\emptyset$, $\mathbb{P}_{z}\left(\tau_{A^+\setminus D} <\tau_D \right)=0$ for $z\in D$, and $\mathbb{P}_{z}\left(\tau_{A^+\setminus D} <\tau_D \right)=1$ for $z\in A^+\setminus D$).

We claim that there exists $C'(\mathscr{G})>0$ such that 
\begin{equation}\label{new_6.9}
	F\subset B_{x_{\diamond}}(C'|D|^2). 
\end{equation}
In fact, by the definition of $A^+$, there exists 
\begin{equation}\label{exist_z_diamond}
	z_{\diamond}\in \left[B_{x_{\diamond}}(400\Cref{ball_1}|D|^2)\setminus B_{x_{\diamond}}(200\Cref{ball_1}|D|^2)\right]  \cap  \left( A^+\setminus D\right).
\end{equation}
Notice that if we do not construct $A^+$ but just consider $A$ itself, then (\ref{exist_z_diamond}) may fail. By Lemma \ref{lemma_2srw_hm}, for a large enough $C'(\mathscr{G})>0$, one has 
\begin{equation}\label{new_6.11}
	\mathbb{P}_z\left(\tau_{z_{\diamond}}< \tau_{D} \right)  \ge 0.9 \mathbb{H}_{D\cup \{z_{\diamond}\}}(z_{\diamond}), \ \forall z\in \big[B_{x_{\diamond}}(C'|D|^2)\big]^c.
\end{equation}
Moreover, by $z\in [B_{x_{\diamond}}(200\Cref{ball_1}|D|^2)]^c$, $D\subset B_{x_{\diamond}}(10\Cref{ball_1}|D|)$ and (\ref{ineq_ball}), one has
\begin{equation*}
	\mathbf{d}(D,z_{\diamond})\ge 100|D|^2\ge [\mathrm{diam}(D)\vee \Cref{delicate_green} ]^2.
\end{equation*}
Therefore, it follows from Lemma \ref{lemma_remove_distant} (with $L=\mathrm{diam}(D)\vee \Cref{delicate_green}$) that 
\begin{equation}\label{new_6.12}
	\mathbb{H}_{D\cup \{z_{\diamond}\}}(z_{\diamond}) \ge (4.2)^{-1}. 
\end{equation}
Combining (\ref{new_6.11}) and (\ref{new_6.12}), we obtain that for any $z\in \big[B_{x_{\diamond}}(C'|D|^2)\big]^c$, 
\begin{equation}\label{new_6.13}
	\mathbb{P}_z\left(\tau_{A^+\setminus D}< \tau_{D} \right)\ge 	\mathbb{P}_z\left(\tau_{z_{\diamond}}< \tau_{D} \right)\ge 0.9\cdot (4.2)^{-1}>0.1,
\end{equation}
and hence $z\notin F$. Now we conclude the claim (\ref{new_6.9}).

Let $\widetilde{A}^+=A^+\setminus D$. Similar to (\ref{new_4.1}), we define 
\begin{equation*}
	\widehat{F}^+:= [\partial^{\mathrm{i}}_{\infty} (A^+\cup F)]\setminus \widetilde{A}^+\ \ \text{and}\ \ \widecheck{F}^+:= [\partial^{\mathrm{i}}_{\bm{0}} (A^+\cup F)]\setminus \widetilde{A}^+.
\end{equation*}
By the strong Markov property, one has: for any $v\in \widehat{F}^+\cup \widecheck{F}^+$ and $v_2\in \widecheck{F}^+$,
\begin{equation*}
	G_{A\setminus D} (v , v_2) \le 	G_{A\setminus D} (v_2, v_2).
\end{equation*}
Thus, by Lemma \ref{lemma_psi_3.5} and (\ref{ineq_psi3.6_2}) (replacing $A$ with $A^+$, and taking $F_1=F_2=F$), 
\begin{equation}\label{new_6.1}
	\frac{\mathbb{H}_{A^+\setminus D}(\bm{0})}{\mathbb{H}_{A^+}(\bm{0})} \le \max_{ v_1\in \widehat{F}^+, v_2\in \widecheck{F}^+}   \left[  \mathbb{P}_{v_2}\left(\tau_{A^+\setminus D} <\tau_D \right) \mathbb{P}_{v_1}\left(\tau_{v_2}<\tau_{A^+} \right) \right]^{-1}.
\end{equation}

For the first probability on the right-hand side of (\ref{new_6.1}), since $v_2\in \widecheck{F}^+=[\partial^{\mathrm{i}}_{\bm{0}} (A^+\cup F)]\setminus \widetilde{A}^+$, there exists $w\in (A^+\cup F)_{\bm{0}}^c$ such that $w\sim v_2$. By the definition of $F$, we know that $\mathbb{P}_{w}\left(\tau_{A^+\setminus D} <\tau_D \right)>0.1$, which implies 
\begin{equation}\label{new_6.2}
	\mathbb{P}_{v_2}\left(\tau_{A^+\setminus D} <\tau_D \right) \ge [\mathrm{deg}(\mathscr{G})]^{-1} \mathbb{P}_{w}\left(\tau_{A^+\setminus D} <\tau_D \right)>0.1[\mathrm{deg}(\mathscr{G})]^{-1}. 
\end{equation}
Let $\overline{D}:=D\cup \partial_{\infty,*}^{\mathrm{o}}D$. For the second probability, by the strong Markov property, 
\begin{equation}\label{final_6.3_1}
	\mathbb{P}_{v_1}\left(\tau_{v_2}<\tau_{A^+} \right)\ge \mathbb{P}_{v_1}\left(\tau_{A^+\setminus D}>\tau_{\overline{D}} \right) \min_{w_1\in \partial_{\infty,*}^{\mathrm{o}}D } \mathbb{P}_{w_1} \left(\tau_{v_2}<\tau_{A^+}\right). 
\end{equation}
Since $D\subset \overline{D}$ and $v_1\in \widehat{F}^+ \subset F$, we have either $\mathbb{P}_{v_1}\left(\tau_{A^+\setminus D}>\tau_{\overline{D}} \right)=1$ (when $v_1\in \partial_{\infty,*}^{\mathrm{o}}D$), or $\mathbb{P}_{v_1}\left(\tau_{A^+\setminus D}>\tau_{\overline{D}} \right) \ge \mathbb{P}_{v_1}\left(\tau_{A^+\setminus D}>\tau_{D} \right) \ge 0.9$. Moreover, for any $w_1\in \partial_{\infty,*}^{\mathrm{o}}D$, by Markov property and Lemma \ref{lemma_prob_surrounding}, 
\begin{equation}\label{final_6.3_2}
	\begin{split}
		\mathbb{P}_{w_1} \left(\tau_{v_2}<\tau_{A^+}\right) \ge &\max_{w_2\in \partial_{\infty,*}^{\mathrm{o}}D} \mathbb{P}_{w_1} \left(\tau_{w_2}<\tau_{A^+}\right)  \mathbb{P}_{w_2} \left(\tau_{v_2}<\tau_{A^+}\right)\\
		\gtrsim  &\big[\lambda(\mathscr{G})\big]^{-|D|+\cref{prob_surrounding}\sqrt{|D|}}  \max_{w_2\in \partial_{\infty,*}^{\mathrm{o}}D} \mathbb{P}_{w_2} \left(\tau_{v_2}<\tau_{A^+}\right). 
	\end{split}
\end{equation}
Combining (\ref{final_6.3_1}) and (\ref{final_6.3_2}), we get: for any $v_1\in  \widehat{F}^+$,
\begin{equation}\label{new_6.5}
	\mathbb{P}_{v_1}\left(\tau_{v_2}<\tau_{A^+} \right)\gtrsim \big[\lambda(\mathscr{G})\big]^{-|D|+\cref{prob_surrounding}\sqrt{|D|}}  \max_{w_2\in \partial_{\infty,*}^{\mathrm{o}}D} \mathbb{P}_{w_2} \left(\tau_{v_2}<\tau_{A^+}\right). 
\end{equation}
By (\ref{6.1}), (\ref{new_6.1}), (\ref{new_6.2}) and (\ref{new_6.5}), 
\begin{equation}\label{6.6}
	\frac{\mathbb{H}_{A\setminus D}(\bm{0})}{\mathbb{H}_{A}(\bm{0})} \lesssim  \big[\lambda(\mathscr{G})\big]^{|D|-\cref{prob_surrounding}\sqrt{|D|}} \max_{v_2\in \widecheck{F}^+ } \Big[ \max_{w_2\in \partial_{\infty,*}^{\mathrm{o}}D} \mathbb{P}_{w_2} \left(\tau_{v_2}<\tau_{A^+}\right)\Big]^{-1}. 
\end{equation}

In what follows, we prove that for any $v_2\in F$, 
\begin{equation}\label{new_6.6}
	\max_{w_2\in \partial_{\infty,*}^{\mathrm{o}}D}\mathbb{P}_{w_2} \left(\tau_{v_2}<\tau_{A^+}\right) \gtrsim \big[|D|\ln(|D|)\big]^{-1}. 
\end{equation}
Since $G_{A^+}(w_2,v_2)=\mathbb{P}_{w_2} \left(\tau_{v_2}<\tau_{A^+}\right)G_{A^+}(v_2,v_2)=\mathbb{P}_{v_2} \left(\tau_{w_2}<\tau_{A^+}\right)G_{A^+}(w_2,w_2)$, 
\begin{equation}\label{new_6.7}
	\begin{split}
		\max_{w_2\in \partial_{\infty,*}^{\mathrm{o}}D}	\mathbb{P}_{w_2} \left(\tau_{v_2}<\tau_{A^+}\right) =&  \max_{w_2\in \partial_{\infty,*}^{\mathrm{o}}D} \frac{G_{A^+}(w_2,w_2)}{G_{A^+}(v_2,v_2)}\cdot \mathbb{P}_{v_2} \left(\tau_{w_2}<\tau_{A^+}\right)\\
		\ge & \frac{ \max_{w_2\in \partial_{\infty,*}^{\mathrm{o}}D}\mathbb{P}_{v_2} \left(\tau_{w_2}<\tau_{A^+}\right)}{G_{A^+}(v_2,v_2)}, 
	\end{split}
\end{equation}
where in the second line we used the fact that $G_A(w_2,w_2)\ge 1$.

\textbf{For $\mathbb{P}_{v_2} \left(\tau_{w_2}<\tau_{A^+}\right)$:} Recall that $\mathbb{P}_{v_2}\left(\tau_{A^+\setminus D}>\tau_{\overline{D}} \right)\ge 0.9$ for all $v_2\in F$. Therefore, by the union bound and $|\partial_{\infty,*}^{\mathrm{o}}D|\lesssim  |D|$, we have
\begin{equation}\label{new_6.8}
	\begin{split}
		\max_{w_2\in \partial_{\infty,*}^{\mathrm{o}}D} \mathbb{P}_{v_2} \left(\tau_{w_2}<\tau_{A^+}\right)  \ge & \max_{w_2\in \partial_{\infty,*}^{\mathrm{o}}D}    \mathbb{P}_{v_2} \left(\tau_{\partial_{\infty,*}^{\mathrm{o}}D}= \tau_{w_2}<\tau_{A^+\setminus D}\right)      \\
		\ge & |\partial_{\infty,*}^{\mathrm{o}}D|^{-1} \mathbb{P}_{v_2}\big(\tau_{A^+\setminus D}>\tau_{\partial_{\infty,*}^{\mathrm{o}}D} \big) \gtrsim |D|^{-1}. 
	\end{split}
\end{equation}

\textbf{For $G_{A^+}(v_2,v_2)$:} Recall in (\ref{new_6.9}) that $F\subset B_{x_{\diamond}}(C'|D|^2)$. Thus, similar to (\ref{new_6.13}), there exists $C''(\mathscr{G})>0$ such that for any $w_3\in  \partial^{\mathrm{i}}B'':= \partial^{\mathrm{i}}B_{v_2}(C''|D|^2)$, 
\begin{equation}\label{new_6.14}
	\mathbb{P}_{w_3}\left(\tau_{\bm{0}}< \tau_{v_2} \right)>0.1.
\end{equation}
Meanwhile, by the strong Markov property, 
\begin{equation*}
	\begin{split}
		G_{\{\bm{0}\}}(v_2,v_2) \le & G_{\partial^{\mathrm{i}}B''}(v_2,v_2) + \max_{w_3\in \partial^{\mathrm{i}}B''} \mathbb{P}_{w_3}\left(\tau_{\bm{0}}> \tau_{v_2} \right)G_{\{\bm{0}\}}(v_2,v_2).
	\end{split}
\end{equation*}
Combined with (\ref{new_6.14}) and Lemma \ref{lemma_green}, it implies 
\begin{equation*}
	G_{\{\bm{0}\}}(v_2,v_2) \le 10 G_{\partial^{\mathrm{i}}B''}(v_2,v_2)  \lesssim  \ln(|D|). 
\end{equation*}
Thus, since $\bm{0}\in A^+$, we obtain 
\begin{equation}\label{new_6.15}
	G_{A^+}(v_2,v_2) \le G_{\{\bm{0}\}}(v_2,v_2) \lesssim  \ln(|D|). 
\end{equation}
Combining (\ref{new_6.7}), (\ref{new_6.8}) and (\ref{new_6.15}), we conclude (\ref{new_6.6}).

By (\ref{6.6}) and (\ref{new_6.6}), there exists a constant $C'''(\mathscr{G})>0$ such that 
\begin{equation*}
	\frac{\mathbb{H}_{A\setminus D}(\bm{0})}{\mathbb{H}_{A}(\bm{0})} \le  C'''\big[\lambda(\mathscr{G})\big]^{|D|-\cref{prob_surrounding}\sqrt{|D|}}\cdot |D|\ln(|D|). 
\end{equation*}
We require $\Cref{remove_interlock1}$ to be large enough such that $C'''|D|\ln(|D|)< \big[\lambda(\mathscr{G})\big]^{\frac{1}{2}\cref{prob_surrounding}\sqrt{|D|}}$ when $|D|\ge \Cref{remove_interlock1}$. To sum up, we complete the proof of Lemma \ref{lemma_remove_interlock} with $\cref{remove_interlock2}=\frac{1}{2}\cref{prob_surrounding}$.    \qed

\section*{Acknowledgments}

The authors would like to thank Gady Kozma for fruitful discussions.

\bibliographystyle{plain}
\bibliography{ref}

\end{document}